\documentclass[12pt]{amsart}
\usepackage{amsmath,amsfonts,amsthm,amssymb,nicefrac}
\usepackage{graphicx,color,url}
\usepackage{xcolor}
\usepackage{accents}
\usepackage{enumerate}
\usepackage{comment}
\usepackage{hyperref,tikz}
\newtheorem{theorem}{Theorem}
\newtheorem{lemma}[theorem]{Lemma}
\newtheorem{corollary}[theorem]{Corollary}
\newtheorem{proposition}[theorem]{Proposition}
\theoremstyle{definition}
\newtheorem{remark}[theorem]{Remark}

\newtheorem*{claim}{Claim}
\numberwithin{equation}{section}\numberwithin{theorem}{section}
\allowdisplaybreaks
\newcounter{stepctr}
{\end{list}}

\def\XXint#1#2#3{{\setbox0=\hbox{$#1{#2#3}{\int}$}
 \vcenter{\hbox{$#2#3$}}\kern-.5\wd0}}

\DeclareMathOperator{\tr}{tr}
\DeclareMathOperator{\spt}{spt}
\newcommand{\e}{\varepsilon}

\DeclareMathOperator{\Div}{div}
\author{Huy The Nguyen}
\address{School of Mathematical Sciences\\
	Queen Mary University of London\\
	Mile End Road\\
	London E1 4NS}
\email{h.nguyen@qmul.ac.uk}
\author{Shengwen Wang} \address{School of Mathematical Sciences\\
	Queen Mary University of London\\
	Mile End Road\\
	London E1 4NS}\email{shengwen.wang@qmul.ac.uk}
\begin{document}
\title[Quantization for the Allen--Cahn Equation ]{Quantization of the Energy for the inhomogeneous Allen--Cahn mean curvature}
\begin{abstract}
We consider the varifold associated to the Allen--Cahn phase transition problem in $\mathbb R^{n+1}$(or $n+1$-dimensional Riemannian manifolds with bounded curvature) with integral $L^{q_0}$ bounds on the Allen--Cahn mean curvature (first variation of the Allen--Cahn energy) in this paper. It is shown here that there is an equidistribution of energy between the Dirichlet and Potential energy in the phase field limit and that the associated varifold to the total energy converges to an integer rectifiable varifold with mean curvature in $L^{q_0}, q_0 > n$. The latter is a diffused version of Allard's convergence theorem for integer rectifiable varifolds.
\end{abstract}
\maketitle
\section{Introduction}
Let $\Omega\subset\mathbb(M^{n+1},g)$ be an open subset in a Riemannian manifold with bounded curvature. Consider $u\in W^{2,p}(\Omega)$ satisfying the following equation
	\begin{align}\label{PFVe}
	\e\Delta u_\e -\frac{W'(u_\e)}{\e}=f_\e,
	\end{align}
where $W(t)=\frac{(1-t^2)^2}{2}$ is a double-well potential. The equation \eqref{PFVe} can be viewed as a prescribed first variation problem to the Allen--Cahn energy
	\begin{align*}
	E_\e(u_\e)=\int_\Omega\left( \frac{\e|\nabla u_\e |^2}{2}+\frac{W(u_\e)}{\e}\right)dx.
	\end{align*}
For any compactly supported test vector field $\eta\in C_c^\infty (\Omega,\mathbb R^{n+1})$, we have a variation $u_s(x)=u\left(x+s\eta(x)\right)$ and the first variation formula at $u_0=u_\e$ is given by
	\begin{align}\label{FirstVariationAC}
	\begin{split}
	\frac{d}{ds}\bigg|_{s=0}E_\e\left(u_s\right)&=\int_\Omega\left(-\e\Delta u_\e +\frac{W'(u_\e)}{\e}\right)\langle\nabla u_\e , \eta\rangle dx\\
	&=-\int_\Omega \left(\frac{f_\e}{\e|\nabla u_\e |}\right)\langle\nu,\eta \rangle\e|\nabla u_\e |^2dx,
	\end{split}
	\end{align}
where $\nu=\frac{\nabla u_\e }{|\nabla u_\e |}$ is a unit normal to the level sets at non-critical points of $u$.

By \cite{Modica1977}, \cite{Modica1987}, \cite{Sternberg1988} using the framework of \cite{DeGiorgi1979}, the sequence of functionals $E_\e$ $\Gamma$-converges to the $n$-dimensional area functional as $\e\rightarrow0$. This shows that minimizing solutions to \eqref{PFVe} with $f_\e =0$ converge as $ \e\rightarrow 0$ to area minimizing hypersurfaces.
For general critical points ($ f_\e=0$) a deep theorem of Hutchinson--Tonegawa \cite[Theorem 1]{Hutchinson2000} shows the diffuse varifold obtained by smearing out the level sets of $u$ converges to limit which is a stationary varifold with $a.e.$ integer density. The main result of this paper is to prove Hutchinson--Tonegawa's Theorem \cite[Theorem 1]{Hutchinson2000} in the context of natural integrability conditions on the first variation of $E_\e$. Under suitable controls on the first variation of the energy functional $E_\e$ (the diffuse mean curvature) we can show comparable behaviour for the limit. In the case where $n=2,3$ R\"oger--Sch\"atzle \cite{roger2006modified} have shown under the assumption
	\begin{align*}
	\liminf_{\e\rightarrow 0^+} \left( E_\e(u_\e)+ \frac{1}{\e} \| f_\e\|^2_{L^2(\Omega)} \right) < \infty
	\end{align*}
that the limit is an integer rectifiable varifold with $L^2$ generalised mean curvature.

The main focus of this paper is to generalise this result to higher dimensions. Before we state our main theorem, we give a choice of the diffused analogue of ``mean curvature" in the Allen--Cahn setting, which will be used to state our bounded $L^{q_0}$ Allen--Cahn mean curvature condition in the theorem.

Recall that for an embedded hypersurface $\Sigma^n\subset\Omega\subset\mathbb R^{n+1}$ restricted to a bounded domain $\Omega$ and a compactly supported variation $\Sigma_s$ with $\Sigma_0=\Sigma$, we have the first variation area at $s=0$ given by
	\begin{align}\label{FirstVariationHypersurface}
	\frac{d}{ds}\bigg |_{s=0}\operatorname{Area}(\Sigma_s\cap\Omega)=-\int_{\mathbb R^{n+1}} \langle\mathbf H,\eta\rangle d\mu_\Sigma=\int_{\mathbb R^{n+1}} H\langle\nu,\eta\rangle d\mu_\Sigma,
	\end{align}
where $H$ is the mean curvature scalar, $\mathbf H=-H\nu$ is the mean curvature vector, $\nu$ is a unit normal vector field, $\eta$ is the variation vector field, and $d\mu_\Sigma$ is the hypersurface measure. By comparing the first variation formula \eqref{FirstVariationAC} for Allen--Cahn energy and the first variation formula \eqref{FirstVariationHypersurface} for area , we can see that $\left(\frac{f_\e}{\e|\nabla u|}\right)$ roughly plays the role of the mean curvature scalar in the Allen--Cahn setting. In \cite{Allard1972}, a result of Allard implies that if a sequence of integral varifolds has $L^{q_0}$ integrable mean curvature scalar with $q_0>n$, then after passing to a subsequence, there is a limit varifold which is also integer rectifiable.

Under similar conditions on $L^{q_0}$ integrability of the term $\left(\frac{f_\e}{\e|\nabla u|}\right)$ with $q_0>n$, we prove the integer rectifiability of the limit of sequences of Allen--Cahn varifolds :
\begin{theorem}\label{main}
Let $u_{\e}\in W^{1,2}(\Omega), \Omega\subset\mathbb R^{n+1}$ satisfy equation \eqref{PFVe} with $\e\rightarrow0$ and $f_\e\in L^1(\Omega)$.
If any one of the following holds:
\begin{enumerate}
\item Bounds on the total energy
	\begin{align}
	\int_\Omega \left( \frac{\e|\nabla u_\e |^2}{2}+\frac{W(u_\e)}{\e}\right) dx\leq E_0;
	\end{align}
\item Uniform $L^\infty$ bounds
	\begin{align}
	 \|u_\e\|_{L^\infty(\Omega)}\leq c_0;
	\end{align}
\item $L^{q_0}$ bounds on the diffuse mean curvature
	\begin{align}\label{IntegralBound}
	\int_\Omega\left(\frac{|f_\e|}{\e|\nabla u_\e |}\right)^{q_0}\e|\nabla u_\e |^2dx \leq \Lambda_0
	\end{align}
 for some $q_0> n$;
\end{enumerate}
then after passing to a subsequence, we have for the associated varifolds (see \cite{Ilmanen1993} for the definition) $V_{u_\e}\rightarrow V_\infty$ weakly and
\begin{enumerate}
\item $V_\infty$ is an integral $n$-rectifiable varifold;
\item For any $B_r(x_0)\subset\subset\Omega$, the $L^{q_0}$ norm of the generalized mean curvature of $V_\infty$ is bounded by $\Lambda_0$;
\item The discrepancy measure $\left(\frac{\e|\nabla u_\e |^2}{2}-\frac{W(u_\e)}{\e}\right)dx\rightarrow0$ weakly, i.e. there is an equidistribution of energy as $\e\rightarrow0$ (c.f. Proposition \ref{VanishingDiscrepancy}).
\end{enumerate}
\end{theorem}
This theorem shows we can prove a result analagous to Hutchinson--Tonegawa \cite{Hutchinson2000}, Tonegawa \cite{tonegawa2003integrality} and show as $\e\rightarrow 0$, the diffuse varifold associated to the Allen--Cahn functional converges to an integer rectifiable varifold. This has some similarities with Allard's compactness theorem for rectifiable varifolds and for integral varifolds but here the sequence consists of diffuse varifolds and hence we require stronger conditions on the proposed mean curvature. As we shall see in a later paper, these conditions are exactly what is required to prove a version of Allard's regularity theorem for Allen--Cahn varifolds. As an application, we have the following Corollary
\begin{corollary}\label{LqConditions}
If $u_\e$ satisfies \eqref{PFVe} and of one of the following conditions holds:
\begin{enumerate}
\item
 	\begin{align*}
	&\|f_\e\|_{L^s(\Omega)}\leq C_1\e^\frac{1}{2},\quad\text{ for some $2<s<n$}\\
	&\left\|\frac{f_\e}{\e|\nabla u_\e|}\right\|_{L^t(\Omega)}\leq C_2,\quad\text{ for some $t>\frac{n-2}{s-2}s>\max\{s,n-2\}$};
 \end{align*}
\item
	\begin{align*}
	\left\|\frac{f_\e}{\e|\nabla u_\e|}\right\|_{W^{1,p}(\Omega)}\leq C,\quad\text{ for some $p>\frac{n+1}{2}$},\quad\text{(c.f. \cite{Tonegawa2020})};
	\end{align*}
\item
 	\begin{align*}
	&\|f_\e\|_{L^2(\Omega)}\leq C_1\e^\frac{1}{2},\quad\text{ if the ambient dimension $n+1=2$},\quad\text{(c.f. \cite{roger2006modified})}\\
	&\left\|\frac{f_\e}{\e|\nabla u_\e|}\right\|_{L^\infty(\Omega)}\leq C_2,\quad\text{ if the ambient dimension $n+1\geq3$};
 \end{align*}
 \end{enumerate}
then after passing to a subsequence as $\e\rightarrow 0$, the associated varifolds $V_\e$ converge to an integral $n$-rectifiable varifold with generalized mean curvature in $L^{q_0}$ for some $q_0>n$.
\end{corollary}
\begin{proof}
\begin{enumerate}
\item To see the first condition implies the conditions in Theorem \ref{main}, we choose $q_0=\frac{t(s-2)}{s}+2$ ($q_0>n$ is satisfied due to the choice of $t$ and $s$ above). Then we have
\begin{align*}
\int_\Omega \left|\frac{f_\e}{\e|\nabla u_\e |}\right|^{q_0}\e|\nabla u_\e |^2dx&=\int_\Omega \left|\frac{f_\e}{\e|\nabla u_\e |}\right|^{q_0-2}\frac{|f_\e|^2}{\e}dx\\
	&\leq \frac{1}{\e}\left(\int_\Omega\left|\frac{f_\e}{\e|\nabla u_\e |}\right|^\frac{(q_0-2)s}{s-2}\right)^\frac{s-2}{s}\left(\int_\Omega f_\e^{2\frac{s}{2}}\right)^\frac{2}{s}\\
	&=\frac{1}{\e}\cdot\left\|\frac{f_\e}{\e |\nabla u_\e |}\right\|_{L^t(\Omega)}^{q_0-2}\cdot\|f_\e\|_{L^s(\Omega)}^2\\
	&\leq C_1^2C_2^{q_0-2}\leq\Lambda_0
	\end{align*}
where we used H\"older's inequality in the second line with exponent $\frac{s}{s-2}$.
\item In the paper \cite{Tonegawa2020}, assuming condition (2) above, the authors proved the same integer rectifiability and $L^{q_0}$ mean curvature bound for the limit varifold. We show this conditions implies the integral bounds in the hypothesis of Theorem \ref{main} for some $q_0>n$. To see this, we compute
	\begin{align*}
	\nabla \left(\phi^{\frac{np}{n+1-p}} \right) = \frac{np}{n+1-p} \phi^{\frac{(n+1)(p-1)}{n+1-p}} \nabla \phi.
	\end{align*}
and applying \cite[5.12.4]{Ziemer1989}(c.f.\cite[Theorem 3.7]{Tonegawa2020}) and \cite[Theorem 3.8]{Tonegawa2020}, and H\"older's inequality, with $\varphi = \phi^{\frac{np}{n+1-p}}$ and $ d \mu = \e |\nabla u_\e|^2 d\mathcal L^{n+1}$.
	\begin{align*}
	\left|\int_{\mathbb{R}^n}\varphi d \mu\right| \leq c(n) K(\mu) \int_{\mathbb{R}^n}|\nabla\varphi | d \mathcal{L}^n \quad \forall \varphi \in C_c^1\left(\mathbb{R}^{n+1}\right)
	\end{align*}
which implies
	\begin{align*}
	\left| \int_{\mathbb R^{n+1}} |\phi|^{\frac{np}{n+1-p}} \e |\nabla u_\e |^2 d \mathcal L^{n+1} \right|&\leq \left|\int_{\mathbb R^{n+1}}\varphi d \mu \right| \\
	&\leq C(n)K(\mu)\left|\int_{\mathbb R^{n+1}} \frac{np}{n+1-p} |\nabla \phi| |\phi|^{\frac{(n+1)(p-1)}{n+1-p}}d \mathcal L^{n+1}\right|\\
	&\leq C(n,p)K(\mu)\left| \int_{\mathbb R^{n+1}} |\nabla \phi| |\phi|^{\frac{(n+1)(p-1)}{n+1-p}} d \mathcal L^{n+1}\right|\\
	&\leq C(n,p) \left( \int_{\mathbb R^{n+1}}|\nabla \phi |^p \right)^{1/p} \left( \int_{\mathbb R^{n+1}}|\phi|^{\frac{p(n+1)}{n+1-p}} \right)^\frac{p-1}{p}\\
	&= C(n,p) \|\nabla \phi \|_{L^p(\mathbb R^{n+1})} \|\phi\|_{L^{\frac{p(n+1)}{n+1-p}}(\mathbb R^{n+1})}^{\frac{(p-1)(n+1)}{n+1-p}}.
	\end{align*}
where $C(n,p)\rightarrow \infty$ as $p\rightarrow n+1$. We apply the above inequality with $ \phi =\psi \frac{f_\e}{\e |\nabla u_\e|}$ and $ d \mu = \e |\nabla u_\e|^2$ together the Sobolev inequality to get for $\psi\in C_0^1(\Omega)$
	\begin{align*}
	\int_\Omega\left|\psi\frac{f_\e}{\e|\nabla u_\e |}\right|^{\frac{pn}{n+1-p}}\e|\nabla u_\e |^2 d \mathcal L^{n+1}&\leq C\left\|\nabla\left(\psi\frac{f_\e}{\e|\nabla u_\e |}\right)\right\|_{L^p(\Omega)}\left\|\psi\frac{f_\e}{\e|\nabla u_\e |}\right\|_{L^\frac{p(n+1)}{n+1-p}(\Omega)}^\frac{(p-1)(n+1)}{n+1-p}\\
	&\leq C\left\|\nabla\left(\psi\frac{f_\e}{\e|\nabla u_\e |}\right)\right\|_{L^p(\Omega)}\left\|\nabla\left(\psi\frac{f_\e}{\e|\nabla u_\e |}\right)\right\|^\frac{(p-1)(n+1)}{n+1-p}_{L^p(\Omega)}\\
	&\leq C_\psi\left\|\frac{f_\e}{\e|\nabla u_\e |}\right\|_{W^{1,p}(\Omega}
	\end{align*}
where we have $q_0=\frac{pn}{n+1-p}>n$ since $p>\frac{n+1}{2}$.
\item If $n+1=2$ then this is proven in \cite{roger2006modified}. For $n+1\geq 3$ it can be directly verified that the condition (3) implies the the conditions in Theorem \ref{main}.
\end{enumerate}
\end{proof}
	 Here we give an overview of our proof. In Section \ref{sec_Prelim}, we gather together some standard notation on varifolds and the first variation. In section \ref{sec_Discrepancy}, we prove the main estimates required for the proof of the integrality and rectifiability. Specifically we will need a monotonicity formula. For the homogeneous Allen--Cahn equation and Allen--Cahn flow, a strict monotonicity formula can be proven due to Modica's estimate showing the discrepancy is negative. This estimate is not true without a homogeneous left hand side to equation \eqref{PFVe}. Instead we will use the integral bound \eqref{IntegralBound} to derive a decay bound for $L^1$ norm of the discrepancy which we eventually show vanishes in the limit $\e\rightarrow 0$. This estimate constitutes one of the main advances of this paper. In section \ref{sec:rectifiability} we show the limiting varifold we obtain as $\e\rightarrow 0$ is a rectifiable set and in section \ref{sec:integrality} we show the limiting varifold is in addition integral.

\textbf{Acknowledgements.}
The first author was supported by EPSRC grant EP/S012907/1. The second author was supported by EPSRC grants EP/S012907/1 and EP/T019824/1.

\section{Preliminaries and notations}\label{sec_Prelim}
Throughout the paper, we will denote a constant by $C$ if it only depend on the constants $n,E_0,c_0,\Lambda_0$ which appear in the conditions of Theorem \ref{main}. A certain points we may increase this constant in some steps of the argument, but we will not relabel the constant unless there is a risk of confusion from the context.
We associate to each solution of \eqref{PFVe} a varifold in the following way : let $G(n+1,n)$ denote the Grassmannian (the space of unoriented $n$-dimensional subspaces in $ \mathbb R^{n+1}$). We regard $S\in G(n+1,n)$ as the $(n+1)\times(n+1)$ matrix representing orthogonal projection of $ \mathbb R^{n+1}$ onto $S$, that is
	\begin{align*}
	S^2=S, \quad S^TS=I
	\end{align*}
and write $ S_1\cdot S_2 = \tr(S_1^T\cdot S_2)$.
We say $V$ is an $n$-varifold in $\Omega\subset \mathbb R^{n+1}$ if $V$ is a Radon measure on $ G_n(\Omega) = \Omega \times G(n+1,n)$. Varifold convergence means convergence of Radon measures or weak-$*$ convergence. We let $V\in\mathbb V_n(\Omega)$ and let $\|V\|$ denote the weight measure of $V$ and we define the first variation of $V$ by
	\begin{align*}
	\delta V(\eta) \equiv \int_{G_n(\Omega)}\nabla \eta(x) \cdot S dV(x) \quad \forall \eta \in C_c^1(\Omega;\mathbb R^{n+1}).
	\end{align*}
We let $ \| \delta V\|$ be the total variation of $\delta V$. If $\|\delta V\|$ is absolutely continuous with respect to $\|\delta V\|$ then the Radon--Nikodym derivative $\frac{\delta V}{\| V\|}$ exists as vector valued measure. We denote by $ H_V = -\frac{\delta V}{\|V\|}$, the generalised mean curvature.

Let $ u$ be a function, we define the associated energy measure as a Radon measure given by
	\begin{align*}
	d\mu \equiv \left( \frac{\e |\nabla u|^2}{2}+ \frac{W(u_\e)}{\e}\right) d \mathcal L^{n+1}
	\end{align*}
where $ \mathcal L^{n+1}$ is the $(n+1)$ dimensional Lebesgue measure. We also denote the the energy of the $1$ dimensional solution by
	\begin{align*}
	\sigma = \int_{-1}^1\sqrt{2 W(s)} ds.
	\end{align*}
There is an associated varifold $ V \in \mathbb V_n(\Omega)$ to the functions $u$ given by
	\begin{align*}
	V(\phi)&= \int_{\{| \nabla u |\neq 0\}} \phi \left(x,\left(\frac{\nabla u(x)}{|\nabla u(x)|}\right)^\perp\right) d\mu(x) \\
	&= \int_{\{| \nabla u |\neq 0\}} \phi \left(x,I - \frac{\nabla u(x)}{|\nabla u(x)|}\otimes\frac{\nabla u(x)}{|\nabla u(x)|}\right) d\mu(x), \quad \phi \in C_c(G_n(\Omega)).
	\end{align*}
where $I$ is the $(n+1)\times(n+1)$ identity matrix and
	\begin{align*}
	I - \frac{\nabla u(x)}{|\nabla u(x)|}\otimes\frac{\nabla u(x)}{|\nabla u(x)|}
	\end{align*}
is orthogonal projection onto the space orthogonal to $\nabla u(x)$, that is $ \{a\in \mathbb R^{n+1}\mid \langle a, \nabla u(x)\rangle = 0 \}.$ By definition $\| V \| = \mu \llcorner_{\{|\nabla u| \neq 0 \} }$ and the first variation may be computed as
	\begin{align}\label{eqn_FirstVariationGeneral}
	\delta V( \eta) = \int_{\{| \nabla u |\neq 0\}} \nabla \eta \cdot \left(I - \frac{\nabla u(x)}{|\nabla u(x)|}\otimes\frac{\nabla u(x)}{|\nabla u(x)|}\right) d\mu(x), \quad \forall \eta \in C_c^1(\Omega;
\mathbb R^{n+1}).
	\end{align}
\section{Discrepancy bounds and monotonicity formula}\label{sec_Discrepancy}
In this section, we deduce integral bounds on the discrepancy. There exists an almost monotonicity formula for the Allen--Cahn energy functional, we will give estimates on the terms appearing in the almost monotonicity formulas under the assumptions in Theorem \ref{main} and obtain a monotonicity formula for the $n$-dimensional volume ratio. It will be used in the next section to deduce rectifiability and integrality of the limit varifold as $\e\rightarrow0$. Conditions $(1)$-$(3)$ in Theorem \ref{main} are assumed to hold throughout this section.

The $n$-dimensional volume ratio of the energy measure satisfies the following almost monotonicity formula.
\begin{proposition}[Almost Monotonicity Formula]
If $u_\e$ satisfies \eqref{PFVe} in $B_1\subset\mathbb R^{n+1}$, then for $r<1$, we have
	\begin{align}\label{AlmostMonotonicity}
	\frac{d}{dr}\left(\frac{\mu_\e(B_r)}{r^n}\right)=-\frac{1}{r^{n+1}}\xi(B_r)+\frac{\e}{r^{n+2}}\int_{\partial B_r}\langle x,\nabla u_\e \rangle^2-\frac{1}{r^{n+1}}\int_{B_r}\langle x,\nabla u_\e \rangle f_\e.
	\end{align}
Here $\mu_\e(B_r)=\int_{B_r}d\mu_\e=\int_{B_r}\left(\frac{\e|\nabla u_\e |^2}{2}+\frac{W(u_\e)}{\e}\right)$ is the total energy and $\xi(B_r)=\int_{B_r}\left(\frac{\e|\nabla u_\e |^2}{2}-\frac{W(u_\e)}{\e}\right)$ is the discrepancy measure (difference between the Dirichlet and potential energy) in $B_r$.
\end{proposition}
\begin{proof}
Multiplying equation \eqref{PFVe} by $\langle x,\nabla u_\e\rangle$ and integrating by parts on $B_r$, we get
	\begin{align*}
	&\int_{B_r}\langle x,\nabla u_\e \rangle f_\e\\
	&=\int_{B_r}\e\Delta u_\e \langle x,\nabla u_\e \rangle-\int_{B_r}\left\langle\frac{\nabla(W(u_\e))}{\e}, x\right\rangle\\
	&=\int_{\partial B_r}\left(\e r\left|\frac{\partial u_\e }{\partial\nu}\right|^2-r\frac{W(u_\e)}{\e}\right)-\int_{B_r}\left(\e\delta_{ij}u_{x_i}u_{x_j}+\e\nabla^2u(\nabla u_\e ,x)-\frac{(n+1)W(u_\e)}{\e}\right)\\
	&=\int_{\partial B_r}\left(\e r\left|\frac{\partial u_\e }{\partial\nu}\right|^2-r\frac{W(u_\e)}{\e}\right)-\int_{B_r}\left(\e|\nabla u_\e |^2+\e\left\langle\nabla\frac{|\nabla u_\e |^2}{2}, x\right\rangle-\frac{(n+1)W(u_\e)}{\e}\right)\\
	&=r\int_{\partial B_r}\left(\e \left|\frac{\partial u_\e }{\partial\nu}\right|^2-\frac{W(u_\e)}{\e}-\e\frac{|\nabla u_\e |^2}{2}\right)+\int_{B_r}\left(\e\frac{(n-1)|\nabla u_\e |^2}{2}+\frac{(n+1)W(u_\e)}{\e}\right)\\
	&=n\int_{B_r}\left(\frac{\e|\nabla u_\e |^2}{2}+\frac{W(u_\e)}{\e}\right)-r\int_{\partial B_r}\left(\frac{\e|\nabla u_\e |^2}{2}+\frac{W(u_\e)}{\e}\right)+\frac{\e}{r}\int_{\partial B_r}\langle x,\nabla u_\e \rangle^2-\xi(B_r).
	\end{align*}
The conclusion then follows by dividing both sides by $r^{n+1}$ and noticing
	\begin{align*}
	\frac{d}{dr}\left(\frac{\mu(B_r)}{r^n}\right)=-\frac{n}{r^{n+1}}\int_{B_r}\left(\frac{\e|\nabla u|^2}{2}+\frac{W(u_\e)}{\e}\right)+\frac{1}{r^n}\int_{\partial B_r}\left(\frac{\e|\nabla u|^2}{2}+\frac{W(u_\e)}{\e}\right).
	\end{align*}
\end{proof}
Integrating the almost monotonicity formula \eqref{AlmostMonotonicity} from $\e$ to $r_0$ for $0<\e<r_0<1$, we have
	\begin{align}\label{IntegralAlmostMonotonicity2}
	&\frac{\mu_\e(B_{r_0})}{r_0^n}-\frac{\mu_\e(B_\e)}{\e^n}\\\nonumber
&=\int_\e^{r_0}\left(-\frac{1}{r^{n+1}}\xi(B_r)+\frac{\e}{r^{n+2}}\int_{\partial B_r}\langle x,\nabla u_\e \rangle^2-\frac{1}{r^{n+1}}\int_{B_r}\langle x,\nabla u_\e \rangle f_\e\right)dr\\\nonumber
&\geq -r_0\sup_{B_{r_0}}\omega_{n+1}\left(\frac{\e|\nabla u_\e |^2}{2}-\frac{W(u_\e)}{\e}\right)_+ + \int_{B_{r_0}\setminus B_\e }\frac{\e\langle x,\nabla u_\e \rangle^2}{|x|^{n+2}}-\int_\e^{r_0}\frac{1}{r^{n+1}}\int_{B_r}\langle x,\nabla u_\e \rangle f_\e dr,
	\end{align}
where $\omega_{n+1}$ denotes the volume of unit ball in $\mathbb R^{n+1}$.	
	
We need to estimate the first and third term on the right hand side to obtain a monotonicity formula.
In order to estimate the third term, we derive an a priori gradient bound for $u$. Condition $(3)$ of Theorem \ref{main} states a combined integrability for the inhomogeneity $f_\e$ and $|\nabla u|$. The following theorem allows us to obtain separate integrability and regularity for each quantity.
\begin{theorem}\label{GradientBound}
There exists $ C, \e_0>0$ depending on $E_0,c_0,\Lambda_0$ as defined in Theorem \ref{main} such that if $u_\e$ satisfies \eqref{PFVe} in $B_1\subset\mathbb R^{n+1}$ with $\e<\e_0$ and if $q_0>n+1$, then
	\begin{align}\label{GradBound}
	\sup_{B_{1-\e}}\e|\nabla u_\e |\leq C,
	\end{align}
and
	\begin{align}\label{GradHolderBound}
	\e^{2-\frac{n+1}{q_0}}\|u_\e\|_{C^{1,1-\frac{n+1}{q_0}}(B_{1-\e})}\leq C.
	\end{align}
If $n<q_0\leq n+1$, then
	\begin{align}\label{HolderBound}
	\e^\frac{1}{2}\|u_\e\|_{C^{0,\frac{1}{2}}(B_{1-\e})}\leq C.
	\end{align}
Furthermore, there exists a $\delta_0>0$ so that $f$ has the following improved integrability
	\begin{align}\label{fLpBound}
	\| f_\e\|_{L^{\frac{n+1}{2}+\delta_0}(B_{1-\e}(x_0))}\leq C\e^{-\frac{n}{q_0}}.
	\end{align}
\end{theorem}
\begin{proof}
We first consider the case $q_0>n+1$:
Define the rescaled solution $\tilde u(x):=u(\e x)$ and $\tilde f(x)=\e f_\e(\e x)$ which satisfies the equation
	\begin{align}\label{RescaledEquation}
	\Delta\tilde u-W'(\tilde u)=\tilde f,\quad\text{ in $B_\frac{1}{\e}\subset\mathbb R^{n+1}$}.
	\end{align}
By condition (3) in Theorem \ref{main}, we have by rescaling
	\begin{align}\label{RescaledLqCondition}
	\int_{B_\frac{1}{\e}}\tilde f^{q_0}\e^{n-q_0}|\nabla\tilde u|^{2-q_0}=\int_{B_{\frac{1}{\e}}}\e^{-2q_0}\tilde f^{q_0}\e|\nabla \tilde u|^{2-q_0}\e^{q_0-2}\e^{n+1}=\int_{B_1}\e^{-q_0}f^{q_0}\e|\nabla u|^{2-q_0}\leq\Lambda_0.
	\end{align}
\begin{claim}
For any $\bar B_1(x_0)\subset B_{\frac{1}{\e}-1}$, we have
	\begin{align*}
	\|\nabla\tilde u\|_{L^2(B_1(x_0))}\leq C(c_0,\Lambda_0,q_0,n).
	\end{align*}
\end{claim}
\begin{proof}[Proof of Claim]
By the hypothesis $\bar B_1(x_0)\subset B_{\frac{1}{\e}-1}$ we have $B_2(x_0)\subset B_\frac{1}{\e}$. We choose a smooth cutoff function $\phi\in C_c^\infty\left(B_2(x_0)),[0,1]\right)$ with $\phi\equiv1$ in $B_1(x_0)$ and $|\nabla\phi|\leq4$. By integration by parts and Young's inequality, we obtain
	\begin{align}\label{eqn_L2Bound}
	\begin{split}
\int_{B_2(x_0)}|\nabla\tilde u|^2\phi^2&\leq\int_{B_2(x_0)}2c_0|\nabla\tilde u||\phi||\nabla\phi|+\int_{B_2(x_0)}c_0\phi^2|\Delta\tilde u|\\
	&\leq\int_{B_2(x_0)}2c_0|\nabla\tilde u||\phi||\nabla\phi|+\int_{B_2(x_0)}c_0\phi^2|W'(\tilde u)|+\int_{B_2(x_0)}c_0\phi^2|\tilde f|\\
	&\leq\frac{1}{2}\int_{B_2(x_0)}|\nabla\tilde u|^2\phi^2+\int_{B_2(x_0)}2c_0^2|\nabla\phi|^2+\int_{B_2(x_0)}c_0\phi^2 C_{c_0}+\int_{B_2(x_0)}c_0\phi^2|\tilde f|.
\end{split}
	\end{align}
We write $ c_0 \phi^2 | \tilde f| = c_0| \tilde f |\e^{\frac{n}{q_0}-1}|\nabla\tilde u|^{\frac{2}{q_0}-1} \times \phi^2\e^{1-\frac{n}{q_0}}|\nabla\tilde u|^{1-\frac{2}{q_0}}$ and use Young's inequality with exponent $q_0$ to get
	\begin{align*}
	\int_{B_2(x_0)}c_0\phi^2|\tilde f|&\leq \frac{1}{\delta q_0}\int_{B_2(x_0)}\left|c_0|\tilde f| \e^{\frac{n}{q_0}-1}|\nabla\tilde u|^{\frac{2}{q_0}-1}\right|^{q_0}+\frac{\delta(q_0-1)}{q_0}\int_{B_2(x_0)}\left|\phi^2\e^{1-\frac{n}{q_0}}|\nabla\tilde u|^{1-\frac{2}{q_0}}\right|^\frac{q_0}{q_0-1}\\
	&\leq \frac{c_0^{q_0}}{\delta q_0}\Lambda_0+\frac{\delta(q_0-1)}{q_0}\int_{B_2(x_0)}\phi^\frac{2q_0}{q_0-1}|\nabla\tilde u|^\frac{q_0-2}{q_0-1}\\
	&\leq\frac{c_0^{q_0}}{\delta q_0}\Lambda_0+\frac{C_n\delta(q_0-1)}{q_0}\left(\int_{B_2(x_0)}\phi^\frac{4q_0}{q_0-2}|\nabla\tilde u|^2\right)^\frac{q_0-2}{2(q_0-1)}\\
	&\leq \frac{4C_n(q_0-1)c_0^n}{q_0^2}\Lambda_0+\frac{1}{4}\max\left\{\left[\int_{B_2(x_0)}\phi^2|\nabla\tilde u|^2\right]^\frac{4q_0}{q_0-2},1\right\}.
	\end{align*}
Here we used \eqref{RescaledLqCondition} to bound $\int_{B_2(x_0)}\left|c_0\tilde f\e^{\frac{n}{q_0}-1}|\nabla\tilde u|^{\frac{2}{q_0}-1}\right|^{q_0}$ and the fact that $\e^{1-\frac{n}{q_0}}<1$ in the second inequality, H\"older's inequality with exponent $\frac{2(q_0-1)}{q_0-2}$ in the third inequality. And in the fourth inequality we used $\phi^\frac{4q_0}{q_0-2}\leq\phi^2$, and chose $\delta$ to be $\frac{q_0}{4C_n(q_0-1)}$.
We insert the above inequality into \eqref{eqn_L2Bound} and get
	\begin{align*}
	\int_{B_2(x_0)}|\nabla\tilde u|^2\phi^2&\leq \frac{1}{2}\int_{B_2(x_0)}|\nabla\tilde u|^2\phi^2+\int_{B_2(x_0)}2c_0^2|\nabla\phi|^2+\int_{B_2(x_0)}c_0\phi^2 C_{c_0}\\
	&+\frac{4C_n(q_0-1)c_0^n}{q_0^2}\Lambda_0+\frac{1}{4}\max\left\{\int_{B_2(x_0)}\phi^2|\nabla\tilde u|^2,1\right\}.
	\end{align*}
We assume $\int_{B_2(x_0)}\phi^2|\nabla\tilde u|^2\geq1$, otherwise the desired bound holds trivially.
Then by moving the first term $\frac{1}{2}\int_{B_2(x_0)}|\nabla\tilde u|^2\phi^2$ and the fifth term $\int_{B_2(x_0)}\phi^2|\nabla\tilde u|^2$ on the right to the left, we prove the claim.
\end{proof}
Now suppose $\|\nabla\tilde u\|_{L^{p_0}(B_1(x_0))}\leq C(c_0,\Lambda_0,q_0,n)$ (independent of $\e$) for some $p_0>1$ ($p_0$ can be chosen to be 2 by the claim above).
For any $B_2(x_0)\in B_{\frac{1}{\e}}(0)$, we have by H\"older's inequality
	\begin{align}\label{fLp}
	&\|\tilde f\|_{L^\frac{p_0q_0}{p_0+q_0-2}(B_1(x_0))}=\left(\int_{B_1(x_0)} |\tilde f|^\frac{p_0q_0}{p_0+q_0-2}\right)^\frac{p_0+q_0-2}{p_0q_0}\\\nonumber
&\leq\left[\left\|\left|\tilde f\e^{\frac{n-q_0}{q_0}}|\nabla\tilde u|^{\frac{2}{q_0}-1}\right|^\frac{p_0q_0}{p_0+q_0-2}\right\|_{L^\frac{p_0+q_0-2}{p_0}(B_1(x_0))}\cdot\left\|\left(\e^{\frac{q_0-n}{q_0}}|\nabla\tilde u|^{1-\frac{2}{q_0}}\right)^\frac{p_0q_0}{p_0+q_0-2}\right\|_{L^\frac{p_0+q_0-2}{q_0-2}(B_1(x_0))}\right]^\frac{p_0+q_0-2}{p_0q_0}\\\nonumber
&\leq\left[\Lambda_0^\frac{p_0}{p_0+q_0-2}\e^\frac{(q_0-n)p_0}{p_0+q_0-2}\cdot\left(\int_{B_1(x_0)}|\nabla \tilde u|^{p_0}\right)^\frac{q_0-2}{p_0+q_0-2}\right]^\frac{p_0+q_0-2}{p_0q_0}\\\nonumber
&=\Lambda_0^\frac{1}{q_0}\cdot\e^\frac{q_0-n}{q_0}\cdot\left(\int_{B_1(x_0)}|\nabla\tilde u|^{p_0}\right)^\frac{q_0-2}{p_0q_0}\\\nonumber
&\leq C(c_0,\Lambda_0,q_0,n)\e^\frac{q_0-n}{q_0}\leq C(c_0,\Lambda_0,q_0,n).
	\end{align}
{\begin{remark}
Here $q_0>n$ will make the scaling subcritical and ensures a uniform bound of $\|\tilde f\|_{L^\frac{p_0q_0}{p_0+q_0-2}(B_1(x_0))}$ independent of $\e$.
\end{remark}
}
Thus $\tilde f$ is uniformly bounded in $L^\frac{p_0q_0}{p_0+q_0-2}(B_1(x_0))$ independent of $\e$. By applying the Sobolev inequality to \eqref{RescaledEquation}, standard Calderon--Zygmund estimates and finally using the $L^\infty$ bound of $u$ in condition (2) of Theorem \ref{main}, we have
	\begin{align}\label{LpGradu}
	\|\nabla\tilde u\|_{L^{\frac{p_0q_0}{p_0+q_0-2-p_0\frac{q_0}{n+1}}}({B_1(x_0)})}&\leq\|\tilde u\|_{W^{1,\frac{p_0q_0}{p_0+q_0-2-p_0\frac{q_0}{n+1}}}({B_1(x_0)})}\\\nonumber
&\leq C\|\tilde u\|_{W^{2,\frac{p_0q_0}{p_0+q_0-2}}(B_1(x_0))}\\\nonumber
&\leq C\|\tilde f\|_{L^\frac{p_0q_0}{p_0+q_0-2}(B_1(x_0))}+C\|W'(\tilde u)\|_{L^\frac{p_0q_0}{p_0+q_0-2}(B_1(x_0))}\\\nonumber
&\leq C\Lambda_0^\frac{1}{q_0}\cdot\e^\frac{q_0-n}{q_0}\cdot\left(\int_{B_1(x_0)}|\nabla \tilde u|^{p_0}\right)^\frac{q_0-2}{p_0q_0}+C\|W'(\tilde u)\|_{L^\infty(B_1(x_0))}\\\nonumber
&\leq C(c_0,\Lambda_0,q_0,n)(\e^\frac{q_0-n}{q_0}+1)\leq \tilde C(c_0,\Lambda_0,q_0,n).
	\end{align}
We remark that $ q_0 > n$ ensures the coefficient $\e^\frac{q_0-n}{q_0}$ stays uniformly bounded as $ \e \rightarrow 0$.

In the case $\frac{p_0q_0}{p_0+q_0-2}>n+1$, by Calderon--Zygmund estimates we have
	\begin{align*}
	\|\tilde u\|_{W^{2,\frac{p_0q_0}{p_0+q_0-2}}(B_1(x_0))}\leq C(c_0,\Lambda_0,q_0,n)(\e^\frac{q_0-n}{q_0}+1)\leq \tilde C(c_0,\Lambda_0,q_0,n).
	\end{align*}
The Sobolev inequality then gives $\|\nabla \tilde u\|_{L^\infty}\leq C$.

In the case $\frac{p_0q_0}{p_0+q_0-2}\leq n+1$, using $q_0>n+1$, we have $p_0<p_0\frac{q_0}{n+1}$. Namely
	\begin{align}\label{IterationP0}
	\frac{q_0}{p_0+q_0-2-p_0\frac{q_0}{n+1}}p_0=\frac{q_0}{(p_0-p_0\frac{q_0}{n+1})+(q_0-2)}p_0\geq\frac{q_0}{q_0-2}p_0.
	\end{align}
So we improved $\nabla\tilde u$ from $L^{p_0}$ to $L^{\frac{q_0}{q_0-2}p_0}$.
Define $p_i=\frac{q_0}{q_0-2}p_{i-1}$. Using $q_0>n+1$, we iterate finitely many times until $p_i>\frac{(n+1)(q_0-2)}{q_0-(n+1)}$, i.e. $\frac{p_iq_0}{p_i+q_0-2}>n+1$. The Sobolev inequality gives $\nabla \tilde u\in L^\infty$.
So if $q_0>n+1$, we get $\nabla\tilde u\in L^\infty$. Rescaling back, we get \eqref{GradBound}.
By \eqref{RescaledLqCondition} where $(q_0> n+1\geq2)$ and $\nabla\tilde u\in L^\infty$, we have $\tilde f\in L^{q_0}$. Standard Calderon--Zygmund estimates give
	\begin{align*}
	\|\nabla\tilde u\|_{C^{0,1-\frac{n+1}{q_0}}(B_1(x_0))}\leq\|\tilde u\|_{W^{2,q_0}(B_1(x_0))}\leq \|\tilde f\|_{L^{q_0}(B_1(x_0))}+\|W'(\tilde u)\|_{L^{q_0}(B_1(x_0))}<\infty,
	\end{align*}
which gives \eqref{GradHolderBound}.

Consider now the case $n<q_0\leq n+1$. For any
	\begin{align}\label{ConditionToIteratePi}
	p_i\leq \frac{2(n+1)}{n+1-q_0}-\delta,
	\end{align} we have
	\begin{align*}
	p_i+q_0-2-p_i\frac{q_0}{n+1}&=p_i\frac{n+1-q_0}{n+1}+q_0-2\\
	&=\left(\frac{2(n+1)}{n+1-q_0}-\delta\right)\frac{n+1-q_0}{n+1}+q_0-2\\
	&=q_0-\frac{n+1-q_0}{n+1}\delta.
	\end{align*}
And thus
	\begin{align}
	\frac{q_0}{p_i+q_0-2-p_i\frac{q_0}{n+1}}p_i\geq\frac{q_0}{q_0-\frac{n+1-q_0}{n+1}\delta}p_i\geq p_i.
	\end{align}
So \eqref{LpGradu} increases the integrability of $\nabla\tilde u$ from $L^{p_i}$ to $L^{\frac{q_0}{q_0-\frac{n+1-q_0}{n+1}\delta}p_i}$.
And we can iterate until \eqref{ConditionToIteratePi} fails, namely
	\begin{align}\label{IntegralGradBound}
	\|\nabla \tilde u\|_{L^{\frac{2(n+1)}{n+1-q_0}-\delta}(B_1(x_0))}\leq C(c_0,\Lambda_0,q_0,n)\e^\frac{q_0-n}{q_0}\leq C(c_0,\Lambda_0,q_0,n),
	\end{align}
for any $x_0\in B_{\frac{1}{\e}-2}$(so that the condition in the claim above is satisfied).
			By Sobolev inequalities, we then have for any $x_0\in B_{\frac{1}{\e}-2}$
	\begin{align*}
	\|\tilde u\|_{C^{0,\frac{1}{2}}(B_1(x_0))}&\leq C\|\tilde u\|_{W^{1,2(n+1)}(B_1(x_0))}\\
	&\leq C\|\tilde u\|_{W^{1,\frac{2(n+1)}{n+1-q_0}-\delta}(B_1(x_0))}\\
	&\leq C(c_0,\Lambda,0,q_0,n)\e^\frac{q_0-n}{q_0}\leq C(c_0,\Lambda,0,q_0,n).
	\end{align*}
Rescaling back gives
	\begin{align*}
	\e^\frac{1}{2}\|u\|_{C^{0,\frac{1}{2}}(B_{1-\e})}\leq \|\tilde u\|_{C^{0,\frac{1}{2}}(B_{\frac{1}{\e}-1})}\leq C(c_0,\Lambda_0,q_0,n)\e^\frac{q_0-n}{q_0}\leq C(c_0,\Lambda_0,q_0,n),
	\end{align*}
which is \eqref{HolderBound}.
By \eqref{fLp} we improve the integrability of $\tilde f$ in \eqref{fLp} up to
	\begin{align*}
	\|\tilde f\|_{L^\frac{p_iq_0}{p_i+q_0-2}(B_1(x_0))}\leq C\e^\frac{q_0-n}{q_0},
	\end{align*}
for $p_i\leq\frac{2(n+1)}{n+1-q_0}-\delta$.
So if $q_0\in(n,n+1]$, by choosing $p_i=2(n+1)$, we have
	\begin{align}\label{fOptimalIntegralityPower}
	\frac{p_iq_0}{p_i+q_0-2}=\frac{p_i}{\frac{p_i-2}{q_0}+1}>\frac{p_i}{\frac{p_i-2}{n}+1}=\frac{2(n+1)}{\frac{2(n+1)-2}{n}+1}=\frac{2(n+1)}{3},
	\end{align}
rearranging gives $\frac{p_iq_0}{p_i+q_0-2}>\frac{2(n+1)}{3}\geq\frac{n+1}{2}+\delta_0$ for some $\delta_0>0$.
On the other hand, if $q_0>n+1$, using \eqref{RescaledLqCondition} and the uniform gradient bound of $u$ in Theorem \ref{GradientBound}, we have $\|\tilde f\|_{L^{q_0}(B_1(x_0))}\leq C\e^\frac{q_0-n}{q_0}$, where $q_0>n+1>\frac{n+1}{2}+\delta_0$.
Combining both cases, for any $q_0>n$
	\begin{align}\label{fLpBound2}
	\|\tilde f\|_{L^{\frac{n+1}{2}+\delta_0}(B_1(x_0))}\leq C\e^\frac{q_0-n}{q_0}.
	\end{align}
and
	\begin{align}\label{fLpBound3}
	\|\tilde f\|_{L^{\frac{n+1}{2}+\delta_0}(B_{\frac{1}{\e}-1}(x_0))}\leq C\e^\frac{q_0-n}{q_0}\e^{-n-1},
	\end{align}
Rescaling back gives the bound on $f$,
	\begin{align*}
	\| f\|_{L^{\frac{n+1}{2}+\delta_0}(B_{1-\e}(x_0))}\leq C\e^{-\frac{n}{q_0}}.
	\end{align*}
	\end{proof}
			Since in the case $q_0\in(n,n+1]$, we lack gradient bounds of $u$ as in the case $q_0>n+1$. In order to get better estimates of the discrepancy terms in the almost monotonicity formula, we use some ideas from \cite{roger2006modified}.
We will apply the following Lemma to \eqref{RescaledEquation} for $\e$ sufficiently small such that $C\e^\frac{q_0-n}{q_0}\leq\omega$. \begin{lemma}[cf {\cite[Lemma 3.2]{roger2006modified}}]\label{lem_L1discrepancydeltabound}
Let $ n+1\geq 3, 0<\delta\leq\delta_1$ and $R(\delta)= \frac{1}{\delta^{p_1}}, \omega(\delta) = \delta^{p_2}$, where $p_1=5, p_2=35$. If $\tilde u\in C^2(B_R), \tilde f\in C^0(B_R), B_R=B_R(0)\subset \mathbb R^{n+1}$ where
	\begin{align*}
	\begin{array}{cc}
-\Delta \tilde u +W'(\tilde u)=\tilde f&\text{ in $B_R$}, \\
	|\tilde u| \leq c_0&\text{ in $B_R$}, \\
	\|\tilde f\|_{L^{\frac{n+1}{2}+\delta_0}(B_R)}\leq \omega,
\end{array}
	\end{align*}
$c_0$ is as assumed in condition (2) of Theorem \ref{main} and $\delta_0$ is as in Theorem \ref{GradientBound}.
Then
	\begin{align}\label{eqn_L1discrpancydeltabound}
	\int_{B_1} \left(\frac{|\nabla \tilde u|^2}{2} - W(\tilde u)\right)_+ \leq C\delta.
	\end{align}
And for $ \tau = \delta^{p_3}$, where $p_3=\frac{2\delta_0}{(n+1)^2+(n+1)\delta_0+6\delta_0}$ , we get
	\begin{align}\label{eqn_L1discrepancytaubound}
	\int_{B_{\frac{1}{2}}}\left( \frac{|\nabla \tilde u|^2}{2} - W(\tilde u) \right)_+&\leq c\tau \int_{B_{\frac{1}{2}}}\left( \frac{|\nabla\tilde u|^2}{2} + W(\tilde u) \right) + \int_{B_{\frac{1}{2}}\cap \{|\tilde u| \geq 1- \tau\} }\frac{|\nabla \tilde u|^2}{2}.
	\end{align}
\end{lemma}
\begin{proof}
Let us consider the auxiliary function $\psi$ which solves the Dirichlet problem
	\begin{align}\label{Auxiliary}
	\Delta\psi&=-\tilde f, \quad \text{ in $B_R$}\\\nonumber
\psi&=0, \quad \text{ on $\partial B_R$}.
	\end{align}
The auxiliary function will allows us to control the inhomogeneous part of the equation.
\begin{claim} The function $\psi$ defined in \eqref{Auxiliary} satisfies the bounds
	\begin{align}\label{Suppsi}
	\|\psi\|_{L^\infty(B_R)}&\leq C\delta^{25+5\frac{n+1}{\frac{n+1}{2}+\delta_0}}\ll 1,
	\end{align}
	\begin{align}\label{Gradpsi}
	\|\nabla\psi\|_{L^\frac{(n+1)(n+1+2\delta_0)}{n+1-2\delta_0}(B_R)}&\leq C\omega=C\delta^{35}.
	\end{align}
\end{claim}
\begin{proof}
Rescaling by $\frac{1}{R}$, we have
	\begin{align}\label{AuxiliaryR}
	\begin{split}
	\Delta\psi_R&=\tilde f_R, \quad \text{ in $B_1$}\\
	\psi_R&=0, \quad \text{ on $\partial B_1$},
	\end{split}
	\end{align}
where $\psi_R(x)=\psi(Rx), \tilde f_R(x)=R^2\tilde f(Rx)$. Standard Calderon--Zygmund estimates give
	\begin{align*}
	\|\psi_R\|_{W^{2,\frac{n+1}{2}+\delta_0}(B_1)}\leq\|\tilde f_R\|_{L^{\frac{n+1}{2}+\delta_0}(B_1)}=R^{2-\frac{n+1}{\frac{n+1}{2}+\delta_0}}\|\tilde f\|_{L^{\frac{n+1}{2}+\delta_0}(B_R)}\leq CR^{2-\frac{n+1}{\frac{n+1}{2}+\delta_0}}\omega,
	\end{align*}
where $2-\frac{n+1}{\frac{n+1}{2}+\delta_0}>0$. Rescaling back yields
	\begin{align*}
	&\|\psi\|_{L^{\frac{n+1}{2}+\delta_0}(B_R)}+R\|\nabla\psi\|_{L^{\frac{n+1}{2}+\delta_0}(B_R)}+R^2\|\nabla^2\psi\|_{L^{\frac{n+1}{2}+\delta_0}(B_R)}\\
	&=R^\frac{n+1}{\frac{n+1}{2}+\delta_0}\|\psi_R\|_{L^{\frac{n+1}{2}+\delta_0}(B_1)}+R^\frac{n+1}{\frac{n+1}{2}+\delta_0}\|\nabla\psi_R\|_{L^{\frac{n+1}{2}+\delta_0}(B_1)}+R^\frac{n+1}{\frac{n+1}{2}+\delta_0}\|\nabla^2\psi_R\|_{L^{\frac{n+1}{2}+\delta_0}(B_1)}\\
	&= R^\frac{n+1}{\frac{n+1}{2}+\delta_0}\|\psi_R\|_{W^{2,\frac{n+1}{2}+\delta_0}(B_1)}\\
	&\leq CR^\frac{n+1}{\frac{n+1}{2}+\delta_0}R^{2-\frac{n+1}{\frac{n+1}{2}+\delta_0}}\omega\\
	&=CR^2\omega\\
	&= C\delta^{25}.
	\end{align*}
\emph{Here we prove \eqref{Suppsi}:} by the Sobolev inequality since $\delta_0>0 \implies \frac{n+1}{2} + \delta_0 >\frac{n+1}{2}$, we have
	\begin{align*}
	\|\psi\|_{L^\infty(B_R)}=\|\psi_R\|_{L^\infty(B_1)}&\leq C\|\psi_R\|_{W^{2,\frac{n+1}{2}+\delta_0}(B_1)}\\\nonumber
&\leq CR^{2-\frac{n+1}{\frac{n+1}{2}+\delta_0}}\omega\\\nonumber
&=C\delta^{25+5\frac{n+1}{\frac{n+1}{2}+\delta_0}}\ll 1,
	\end{align*}
due to the choice of $\omega$, where we used $\frac{(n+1)(n+1+2\delta_0)}{n+1-2\delta_0}>n+1$.
\emph{Here we prove the gradient bound \eqref{Gradpsi}:}
	\begin{align*}
	\|\nabla\psi\|_{L^\frac{(n+1)(n+1+2\delta_0)}{n+1-2\delta_0}(B_R)}&\leq R^{\frac{n+1-2\delta_0}{n+1+2\delta_0}-1}\|\nabla\psi_R\|_{L^\frac{(n+1)(n+1+2\delta_0)}{n+1-2\delta_0}(B_1)}\\\nonumber
&\leq CR^{\frac{n+1-2\delta_0}{n+1+2\delta_0}-1}\|\psi_R\|_{W^{2,\frac{n+1}{2}+\delta_0}(B_1)}\\\nonumber
&\leq CR^{\frac{n+1-2\delta_0}{n+1+2\delta_0}-1}R^{2-\frac{n+1}{\frac{n+1}{2}+\delta_0}}\omega\\\nonumber
&=CR^0\omega\\
	&=C\omega=C\delta^{35}.
	\end{align*}
\end{proof}
We define $\tilde u_0:=\tilde u+\psi\in W^{2,\frac{n+1}{2}+\delta_0}(B_R)$. By \eqref{Auxiliary}, \eqref{Suppsi}, $\tilde u_0$ satisfies
	\begin{align}\label{u0}
	|\tilde u_0|&\leq c_0+1,\\\nonumber
\Delta\tilde u_0&=W'(\tilde u).
	\end{align}
We compute for any $\beta>0$,
	\begin{align*}
	\frac{|\nabla\tilde u|^2}{2}-W(\tilde u)&=\frac{|\nabla\tilde u_0-\nabla\psi|^2}{2}-W(\tilde u_0-\psi)\\
	&\leq\left(\frac{1}{2}+\beta\right)|\nabla\tilde u_0|^2+\left(\frac{1}{2}+\frac{1}{\beta}\right)|\nabla\tilde \psi|^2-W(\tilde u_0)+C|\psi|,
	\end{align*}
for some $C>0$. Thus by \eqref{Suppsi} and \eqref{Gradpsi}, we have
	\begin{align*}
	&\int_{B_1}\left(\frac{|\nabla\tilde u|^2}{2}-W(\tilde u)\right)_+\\
	&\leq \int_{B_1}\left(\frac{|\nabla\tilde u_0|^2}{2}-W(\tilde u_0)\right)_+ + \int_{B_1}\left(\beta|\nabla\tilde u_0|^2+C|\psi|+\left(\frac{1}{2}+\frac{1}{\beta}\right)|\nabla \psi|^2\right)\\
	&\leq \int_{B_1}\left(\frac{|\nabla\tilde u_0|^2}{2}-W(\tilde u_0)\right)_++C\left(\beta+R^{2-\frac{n+1}{\frac{n+1}{2}+\delta_0}}\omega+\left(\frac{1}{2}+\frac{1}{\beta}\right)\omega^2\right).
	\end{align*}
By choosing $\beta=\omega\leq\delta^{p_2}$ and using our hypothesis on $\omega : R^{2-\frac{n+1}{\frac{n+1}{2}+\delta_0}}\omega=\delta^{25+\frac{5(n+1)}{\frac{n+1}{2}+\delta_0}}$.
By our choice of $p_1=2,p_2=15$, we ensure
	\begin{align*}
	\beta&=\delta^{35}\leq C\delta,\\
	R^{2-\frac{n+1}{\frac{n+1}{2}+\delta_0}}\omega&=\delta^{25+5\frac{n+1}{\frac{n+1}{2}+\delta_0}}\leq C\delta,\\
	\left(\frac{1}{2}+\frac{1}{\beta}\right)\omega^2&=\frac{1}{2}\delta^{70}+\delta^{35}\leq C\delta,
	\end{align*}
for $n\geq2$. Thus
	\begin{align*}
	\int_{B_1}\left(\frac{|\nabla\tilde u|^2}{2}-W(\tilde u)\right)_+\leq\int_{B_1}\left(\frac{|\nabla\tilde u_0|^2}{2}-W(\tilde u_0)\right)_+ +C\delta.
	\end{align*}
To prove \eqref{eqn_L1discrpancydeltabound}, it suffices to show
	\begin{align}\label{eqn_u0L1discrepancydeltabound}
	\int_{B_1}\left(\frac{|\nabla\tilde u_0|^2}{2}-W(\tilde u_0)\right)_+\leq C\delta.
	\end{align}
Here we estimate $\tilde u$. Define $\tilde u_R(x)=\tilde u(Rx)$ then by the Calderon--Zygmund estimates we have
	\begin{align}\label{w2pu}
	\|\tilde u_R\|_{W^{2,{\frac{n+1}{2}+\delta_0}}(B_\frac{1}{2})}&\leq C\|\Delta\tilde u_R\|_{L^{\frac{n+1}{2}+\delta_0}(B_1)}+C\|\tilde u_R\|_{L^{\frac{n+1}{2}+\delta_0}(B_1)}\\\nonumber
&\leq C\left(R^{2-\frac{n+1}{\frac{n+1}{2}+\delta_0}}\|\Delta\tilde u\|_{L^{\frac{n+1}{2}+\delta_0}(B_R)}+1\right)\\\nonumber
&\leq C\left(R^{2-\frac{n+1}{\frac{n+1}{2}+\delta_0}}\left(\|W'(\tilde u)\|_{L^{\frac{n+1}{2}+\delta_0}(B_R)}+\|\tilde f\|_{L^{\frac{n+1}{2}+\delta_0}(B_R)}\right)+1\right)\\\nonumber
&\leq C\left(R^{2-\frac{n+1}{\frac{n+1}{2}+\delta_0}}(R^\frac{n+1}{\frac{n+1}{2}+\delta_0}+\omega)+1\right)\\\nonumber
&\leq CR^2.
	\end{align}
By the Sobolev embedding
	\begin{align}\label{Gradu}
	\|\nabla\tilde u\|_{L^\frac{(n+1)(n+1+2\delta_0)}{n+1-2\delta_0}(B_\frac{R}{2})}&\leq R^{\frac{n+1-2\delta}{n+1+2\delta}-1}\|\nabla\tilde u_R\|_{L^\frac{(n+1)(n+1+2\delta_0)}{n+1-2\delta_0}(B_\frac{1}{2})}\\\nonumber
&\leq R^{\frac{n+1-2\delta_0}{n+1+2\delta_0}-1}\|\tilde u_R\|_{W^{2,{\frac{n+1}{2}+\delta_0}}(B_\frac{1}{2})}\\\nonumber
&\leq R^{\frac{n+1-2\delta_0}{n+1+2\delta_0}-1}\cdot CR^2=C R^{\frac{n+1-2\delta}{n+1+2\delta}+1}.
	\end{align}
We define
	\begin{align*}
	\tilde f_0&:=-\Delta\tilde u_0+W'(\tilde u_0)\\
	&=-\Delta\psi-\Delta\tilde u+W'(\tilde u)+W'^\prime(\tilde u)\psi+\frac{1}{2}W^{(3)}(\tilde u)\psi^2+\frac{1}{6}W^{(4)}(\tilde u)\psi^3\\
	&=W'^\prime(\tilde u)\psi+\frac{1}{2}W^{(3)}(\tilde u)\psi^2+\frac{1}{6}W^{(4)}(\tilde u)\psi^3,
	\end{align*}
since the derivatives of order $5$ or higher of the potential $W(u)=\frac{(1-u^2)^2}{2}$ vanish. By \eqref{Suppsi}, \eqref{Gradpsi} and \eqref{Gradu}, we have
	\begin{align}\label{Supf0}
	\|\tilde f_0\|_{L^\infty(B_R)}\leq C\|\tilde \psi\|^3_{L^\infty(B_R)}\leq C\left(R^{2-\frac{n+1}{\frac{n+1}{2}+\delta_0}}\omega\right)^3\leq C\delta^{75+15\frac{n+1}{\frac{n+1}{2}+\delta_0}}\ll 1,
	\end{align}
and
	\begin{align}\label{Gradf0}
	\|\nabla\tilde f_0\|_{L^\frac{(n+1)(n+1+2\delta_0)}{n+1-2\delta_0}(B_R)}&\leq C\left(\|\nabla\tilde u\|_{L^\frac{(n+1)(n+1+2\delta_0)}{n+1-2\delta_0}(B_R)}\cdot\|\psi\|_{L^\infty(B_R)}+\|\nabla\psi\|_{L^\frac{(n+1)(n+1+2\delta_0)}{n+1-2\delta_0}(B_R)}\right)\\\nonumber
&\leq C\left(R^{\frac{n+1-2\delta_0}{n+1+2\delta_0}+1}\cdot R^{2-\frac{n+1}{\frac{n+1}{2}+\delta_0}}\omega+\omega\right)\\\nonumber
&=C(R^2\omega+\omega)\\\nonumber
	&\leq CR^2\omega\\\nonumber
	&=C\delta^{25}\ll1.
	\end{align}
Since we have $|\tilde u_0| \leq c_0$, we apply Calderon--Zygmund to \eqref{u0}, for any $B_1(x)\subset B_R$ and $1<r<\infty$ and we get
	\begin{align}\label{u0Bounds}
	\|\tilde u_0\|_{W^{2,r}(B_\frac{1}{2}(x))}\leq C_r.
	\end{align}
	
Hence by the Morrey embedding
	\begin{align*}
	\|\nabla\tilde u_0\|_{L^\infty(B_{R-1})}\leq C.
	\end{align*}
We define a modified discrepancy
	\begin{align}\label{GDiscrepancy}
	\xi_G:=\frac{|\nabla\tilde u_0|^2}{2}-W(\tilde u_0)-G(\tilde u_0)-\varphi,
	\end{align}
for some function $G\in C^\infty(\mathbb R)$ and $\varphi\in W^{2,2}(B_R)$ that we choose as in the following claims
\begin{claim}
If we make the following choice of $G$,
	\begin{align} \label{eqn_Gdelta}
	G_\delta(r):=\delta\left(1+\int_{-c_0-1}^r\exp\left(-\int_{-c_0-1}^t\frac{|W'(s)|+\delta}{2(W(s)+\delta)}ds\right)dt\right)
	\end{align}
then we have the properties
	\begin{align}\label{eqn_Gdeltabounds}
	\begin{split}
\delta&\leq G_\delta(\tilde u_0)\leq C\delta,\\
	0&<G_\delta'(\tilde u_0)\leq\delta,\\
	0&<-G_{\delta}''(\tilde u_0)=G_\delta'(\tilde u_0)\frac{|W'(\tilde u_0)|+\delta}{2(W(\tilde u_0)+\delta)}\leq C.
\end{split}
	\end{align}
Furthermore we have
	\begin{align}\label{eqn_Gdeltadiffinequality}
	G_\delta' W'-2G''_\delta(W+G_\delta)&\geq\delta G'_\delta
	\end{align}
and
	\begin{align} \label{eqn_Gdashdelta}
	G_\delta'(\tilde u_0)&\geq C\delta^3.
	\end{align}
\end{claim}
\begin{proof}[Proof of Claim]
The first three equations of \eqref{eqn_Gdeltabounds} follow from the direct computations.
	For \eqref{eqn_Gdeltadiffinequality}, since $G_\delta\geq\delta$, we obtain
	\begin{align*}
	G_\delta' W'-2G''_\delta(W+G_\delta)&=G_\delta'\left(W'+\frac{|W'|+\delta}{(W+\delta)}(W+G_\delta)\right)\\
	&\geq G_\delta'\left(W'+\frac{|W'|+\delta}{(W+\delta)}(W+\delta)\right) \\
	&= G_\delta'\left(W'+|W'| + \delta \right)\\
	&\geq\delta G'_\delta.
	\end{align*}
For \eqref{eqn_Gdashdelta}, from the definition of $G_\delta$ \eqref{eqn_Gdelta} and the bound $|\tilde u_0|\leq c_0+1$, we compute
	\begin{align*}
	\begin{split}
G_\delta'(\tilde u_0)&\geq\delta\exp\left(-\int_{-c_0-1}^{c_0+1}\frac{|W'(s)|+\delta}{2(W(s)+\delta)}ds\right)\\
	&\geq\delta\exp\left(-\int_{-c_0-1}^{-1}\left|\frac{d}{ds}\log(W(s)+\delta)\right|ds-\int_{-1}^0\left|\frac{d}{ds}\log(W(s)+\delta)\right|ds-(c_0+1)\right)\\
	&\geq\delta\exp\left(-\left(\log(W(-c_0-1)+\delta)-\log\delta\right)-\left(\log(1+\delta)-\log\delta\right)-(c_0+1)\right)\\
	&\geq\delta\exp\left(\tilde C-\log(\delta^2)\right)\\
	&\geq C\delta^3,
\end{split}
	\end{align*}
where we used $W$ is an even function, increasing in $[-1,0]$ and decreasing in $[-c_0-1,-1]$.
\end{proof}
\begin{claim}
If we choose $\varphi$ to satisfy the Dirichlet problem
	\begin{align}\label{eqn_varphi}
	\begin{split}
-\Delta\varphi&=|\langle\nabla\tilde u_0,\nabla\tilde f_0\rangle-(W'+G_\delta')\tilde f_0|>0 \quad \text{ in $B_\frac{R}{2}$,}\\
	\varphi&=0 \quad \text{ on $\partial B_\frac{R}{2}$}
\end{split}
	\end{align}
then we have
	\begin{align}\label{eqn_varphipositive}
	\varphi\geq0 \quad \text{ in $B_\frac{R}{2}$}
	\end{align}
and
	\begin{align}\label{Gradphi}
	\|\varphi\|_{W^{1,\infty}(B_\frac{R}{2})}
&\leq CR^{4-\frac{n+1-2\delta_0}{n+1+2\delta_0}}\omega=C\delta^{15+5\frac{n+1-2\delta_0}{n+1+2\delta_0}}.
	\end{align}
\end{claim}
\begin{proof}
Since we have $\varphi\geq0$ in $\partial B_\frac{R}{2}$ by applying the maximum principle, we have $\varphi\geq0$ in $B_\frac{R}{2}$ which gives us \eqref{eqn_varphipositive}. The estimates \eqref{u0Bounds}, \eqref{Supf0} and \eqref{Gradf0} bound the right hand side of \eqref{eqn_varphi}, that is
	\begin{align*}
	\|\Delta\varphi\|_{L^\frac{(n+1)(n+1+2\delta_0)}{n+1-2\delta_0}(B_\frac{R}{2})}&=|\langle\nabla\tilde u_0,\nabla\tilde f_0\rangle-(W'+G_\delta')\tilde f_0|_{L^\frac{(n+1)(n+1+2\delta_0)}{n+1-2\delta_0}(B_\frac{R}{2})}\leq CR^2\omega=C\delta^{25}.
	\end{align*}
Denote by $\varphi_R(x)=\varphi(\frac{Rx}{2})$, then the Calderon--Zygmund estimates give
	\begin{align*}
	\|\varphi\|_{W^{1,\infty}(B_\frac{R}{2})}=\|\varphi_R\|_{W^{1,\infty}(B_1)}&\leq C\|\varphi_R\|_{W^{2,\frac{(n+1)(n+1+2\delta_0)}{n+1-2\delta_0}}(B_1)}\\
	&\leq C\|\Delta\varphi_R\|_{L^\frac{(n+1)(n+1+2\delta_0)}{n+1-2\delta_0}(B_1)}\\
	&\leq CR^{2-\frac{n+1-2\delta_0}{n+1+2\delta_0}}\|\Delta\varphi\|_{L^\frac{(n+1)(n+1+2\delta_0)}{n+1-2\delta_0}(B_\frac{R}{2})}\\
	&\leq CR^{4-\frac{n+1-2\delta_0}{n+1+2\delta_0}}\omega\\
	&=C\delta^{15+5\frac{n+1-2\delta_0}{n+1+2\delta_0}}
	\end{align*}
and hence we obtain \eqref{Gradphi}.
\end{proof}
We choose $\varphi$ according to \eqref{eqn_varphi}. Notice if $\xi_G> 0$, then we have $\nabla\tilde u_0\neq0$ and
	\begin{align}\label{GradLowerBoundu0}
	W(\tilde u_0)\leq\frac{1}{2}|\nabla\tilde u_0|^2.
	\end{align}
The case $\xi_G\leq0$ immediately gives us our desired estimate since we are seeking an upper bound.
\begin{claim}
For the choice of $G$ as in \eqref{eqn_Gdelta}
and $\varphi$ as in \eqref{eqn_varphi} we have the differential inequality
	\begin{align}\label{LaplacianXiLowerBound}
	\Delta\xi_G&\geq-C\left(1+\frac{\delta}{|\nabla\tilde u_0|}\right)(|\nabla\xi_G|+R^{4-\frac{n+1-2\delta_0}{n+1+2\delta_0}}\omega)+C(\delta^6+\delta^4)
	\end{align}
in $B_\frac{R}{2}\cap\{\xi_G>0\}\cap\{\nabla\tilde u_0\neq0\}$.
\end{claim}
\begin{proof}
We compute the Laplacian of the modified discrepancy
	\begin{align}\label{EquationGDiscrepancy}
	\begin{split}
	 \Delta\xi_G&=|\nabla^2\tilde u_0|^2+\langle\nabla\tilde u_0,\nabla\Delta\tilde u_0\rangle-\Delta\varphi-(W'+G')\Delta\tilde u_0-(W'^\prime+G'^\prime)|\nabla\tilde u_0|^2\\
	&=|\nabla^2\tilde u_0|^2+\langle\nabla\tilde u_0,W'^\prime\nabla\tilde u_0-\nabla\tilde f_0\rangle-\Delta\varphi-(W'+G')(W'(\tilde u_0)-\tilde f_0)-(W'^\prime+G'^\prime)|\nabla\tilde u_0|^2\\
	&=|\nabla^2\tilde u_0|^2-\langle\nabla\tilde u_0,\nabla\tilde f_0\rangle-\Delta\varphi-(W'+G')(W'(\tilde u_0)-\tilde f_0)-G'^\prime|\nabla\tilde u_0|^2.
	\end{split}
	\end{align}
By differentiating \eqref{GDiscrepancy}, we have
	\begin{align*}
	 \nabla\xi_G=\nabla^2\tilde u_0\nabla\tilde u_0-(W'+G')\nabla\tilde u_0-\nabla\varphi,
	\end{align*}
and thus
	\begin{align*}
	|\nabla^2\tilde u_0|^2|\nabla\tilde u_0|^2&\geq |\nabla^2\tilde u_0\nabla\tilde u_0|^2\\
	&\geq |\nabla \xi_G+(W'+G')\nabla\tilde u_0+\nabla\varphi|^2\\
	&\geq 2(W'+G')\left\langle\nabla\tilde u_0,\nabla(\xi_G+\varphi)\right\rangle+(W'+G')^2|\nabla\tilde u_0|^2.
	\end{align*}
Dividing by $|\nabla\tilde u_0|^2$, the first term in \eqref{EquationGDiscrepancy}, $|\nabla^2\tilde u_0|^2$ , is bounded as follows
	\begin{align*}
	|\nabla^2\tilde u_0|^2\geq\frac{2(W'+G')}{|\nabla\tilde u_0|^2}\langle\nabla\tilde u_0,\nabla(\xi_G+\varphi)\rangle+(W'+G')^2.
	\end{align*}
The last term in \eqref{EquationGDiscrepancy} is
	\begin{align*}
	|\nabla\tilde u_0|^2=2(\xi_G+W+G+\varphi).
	\end{align*}
Substituting these into \eqref{EquationGDiscrepancy} and rearranging, we have in $B_R\subset\{\nabla\tilde u_0=0\}$
	\begin{align*}
	&\Delta\xi_G-\frac{2(W'+G')}{|\nabla\tilde u_0|^2}\langle\nabla\tilde u_0,\nabla\xi_G\rangle+2G'^\prime\xi_G\\
	&\geq(W'+G')^2-W'(W'+G')-2G'^\prime(W+G)+\frac{2(W'+G')}{|\nabla\tilde u_0|^2}\langle\nabla\tilde u_0,\nabla\varphi\rangle\\
	&-2G'^\prime\varphi-\Delta\varphi-\langle\nabla\tilde u_0,\nabla\tilde f_0\rangle+(W'+G')\tilde f_0\\
	&=(G')^2+\left(G'W'-2G'^\prime(W+G)\right)+\frac{2(W'+G')}{|\nabla\tilde u_0|^2}\langle\nabla\tilde u_0,\nabla\varphi\rangle-2G'^\prime\varphi-\Delta\varphi\\
	&-\langle\nabla\tilde u_0,\nabla\tilde f_0\rangle+(W'+G')\tilde f_0.
	\end{align*}
We choose $G$ to be \eqref{eqn_Gdelta} which allows us to apply the estimates \eqref{eqn_Gdeltabounds} and \eqref{eqn_Gdeltadiffinequality} so that $\xi_G$ satisfies
	\begin{align}\label{EquationXiG}
	\Delta\xi_G&\geq \frac{2(W'+G')}{|\nabla\tilde u_0|^2}\left\langle\nabla\tilde u_0,(\nabla\xi_G+\nabla\varphi)\right\rangle-2G_{\delta}''\xi_G\\\nonumber
&+(G_\delta')^2+\delta G_\delta'-2G_\delta''\varphi-\Delta\varphi-\langle\nabla\tilde u_0,\nabla\tilde f_0\rangle+(W'+G_\delta')\tilde f_0,
	\end{align}
in $B_R\cap\{\nabla\tilde u_0\neq0\}$.
Furthermore we have by \eqref{GradLowerBoundu0}
	\begin{align*}
	|W'(\tilde u_0)|^2 =|\tilde u_0|^2(1-|\tilde u_0|^2)^2 \leq C W(\tilde u_0) \leq C |\nabla\tilde u_0|^2.
	\end{align*}
From \eqref{eqn_Gdeltabounds}, the bounds on $G_\delta$ and its derivatives, we get
	\begin{align}\label{W'+G'}
	\frac{|(W'+G_\delta')(\tilde u_0)\nabla\tilde u_0|}{|\nabla\tilde u_0|^2}\leq\frac{\frac{1}{2}|\nabla\tilde u_0|^3+\delta|\nabla\tilde u_0|}{|\nabla\tilde u_0|^2}\leq C\left(1+\frac{\delta}{|\nabla\tilde u_0|}\right).
	\end{align}
Substituting in \eqref{eqn_varphi}, \eqref{Gradphi}, and \eqref{W'+G'} into \eqref{EquationXiG} and using the fact that $G''<0$, we have
	\begin{align*}
	\Delta\xi_G&\geq-C\left(1+\frac{\delta}{|\nabla\tilde u_0|}\right)(|\nabla\xi_G|+|\nabla\varphi|)+(G_\delta')^2+\delta G_\delta'-\Delta\varphi+\Delta\varphi\\
	&\geq-C\left(1+\frac{\delta}{|\nabla\tilde u_0|}\right)\left(|\nabla\xi_G|+R^{4-\frac{n+1-2\delta_0}{n+1+2\delta_0}}\omega\right)+(G_\delta')^2+\delta G_\delta'.
	\end{align*}
 Thus applying equation \eqref{eqn_Gdashdelta} in $B_\frac{R}{2}\cap\{\xi_G>0\}\cap\{\nabla\tilde u_0\neq0\}$, we have \eqref{LaplacianXiLowerBound}
	\begin{align}
	\Delta\xi_G&\geq-C\left(1+\frac{\delta}{|\nabla\tilde u_0|}\right)(|\nabla\xi_G|+R^{4-\frac{n+1-2\delta_0}{n+1+2\delta_0}}\omega)+C(\delta^6+\delta^4).
	\end{align}
\end{proof}
We define
	\begin{align}\label{DefinitionEta}
	\eta:=\sup_{B_1}\xi_G
	\end{align}
and consider two cases : \newline
\textbf{case i) } $\eta:=\sup_{B_1}\xi_G< \delta$.
Since
	\begin{align*}
	\xi_G:=\frac{|\nabla\tilde u_0|^2}{2}-W(\tilde u_0)-G(\tilde u_0)-\varphi<\delta,
	\end{align*}
by \eqref{eqn_Gdeltabounds} and \eqref{Gradphi} this implies
	\begin{align*}
	\frac{|\nabla\tilde u_0|^2}{2}-W(\tilde u_0)\leq \delta +G(\tilde u_0)+\varphi \leq \delta + C\delta+ CR^{4-\frac{n+1-2\delta_0}{n+1+2\delta_0}}\omega.
	\end{align*}
Our choices give $ CR^{4-\frac{n+1-2\delta_0}{n+1+2\delta_0}}\omega=C\delta^{15+5\frac{n+1-2\delta_0}{n+1+2\delta_0}} \leq C\delta$ so
	\begin{align*}
	\frac{|\nabla\tilde u_0|^2}{2}-W(\tilde u_0)\leq C\delta
	\end{align*}
which, after integrating proves \eqref{eqn_u0L1discrepancydeltabound}.
\ \newline
\textbf{case ii) } $\eta:=\sup_{B_1}\xi_G\geq \delta > 0$.
We choose a cutoff function $\lambda\in C^2_0(B_\frac{R}{2})$ satisfying $0\leq\lambda\leq1$, $\lambda\equiv1$ on $B_\frac{R}{4}$ and $|\nabla^j\lambda|\leq CR^{-j}$ for $j=1,2$. Then $\exists x_0\in B_\frac{R}{2}$ such that
	\begin{align*}
	(\lambda\xi_G)(x_0)=\max\left\{(\lambda\xi_G)(x): x\in\bar B_\frac{R}{2}\right\}\geq\eta>0.
	\end{align*}
By \eqref{u0Bounds} we have $\xi_G\leq C$ for some $C(c_0,\Lambda_0, E_0,n)>0$ in $B_{R-1}$, and thus
	\begin{align*}
	\lambda(x_0)\geq\frac{\eta}{C}.
	\end{align*}
Moreover,
	\begin{align*}
	|\nabla\tilde u_0(x_0)|^2\geq2\xi_G(x_0)\geq2(\lambda\xi_G)(x_0)\geq2\eta\geq2\delta>0.
	\end{align*}
Since $x_0$ is a critical point, $\nabla(\lambda\xi_G)(x_0)=0$, and we get
	\begin{align*}
	|\nabla\xi_G(x_0)|=\lambda(x_0)^{-1}|\nabla\lambda(x_0)|\xi_G(x_0)\leq C(R\eta)^{-1}.
	\end{align*}
At a maximum point $x_0$, the Laplacian of the function $\lambda \xi_G$ satisfies
	\begin{align*}
	0&\geq\Delta(\lambda\xi_G)(x_0)\\
	&=\lambda(x_0)\Delta\xi_G(x_0)+2\langle\nabla\lambda(x_0),\nabla\xi_G(x_0)\rangle+\Delta\lambda(x_0)\xi_G(x_0),
	\end{align*}
and thus
	\begin{align}\label{LaplacianXiUpperBound}
	\Delta\xi_G(x_0)&\leq\lambda(x_0)^{-1}\left(C|\nabla\lambda(x_0)||\nabla\xi_G(x_0)|+|\Delta\lambda(x_0)||\xi_G(x_0)|\right)\\\nonumber
&\leq C\eta^{-1}\left(CR^{-1}(R\eta)^{-1}+CR^{-2}\right)\\\nonumber
&\leq CR^{-2}\eta^{-1}(1+\eta^{-1})\\\nonumber
&\leq CR^{-2}\eta^{-1}(1+\delta^{-1})\\\nonumber
&\leq CR^{-2}\eta^{-1}\delta^{-1},
	\end{align}
since $\delta\ll1$. Combining \eqref{LaplacianXiLowerBound} and \eqref{LaplacianXiUpperBound} we have
	\begin{align*}
	CR^{-2}\eta^{-1}\delta^{-1}&\geq-C\left(1+\frac{\delta}{|\nabla\tilde u_0(x_0)|}\right)\left(|\nabla\xi_G|+R^{4-\frac{n+1-2\delta_0}{n+1+2\delta_0}}\omega\right)+C(\delta^6+\delta^4)\\
	&\geq C\left[\left(1+\frac{\delta}{2\delta}\right)\left((R\eta)^{-1}+R^{4-\frac{n+1-2\delta_0}{n+1+2\delta_0}}\omega\right)+\delta^4\right].
	\end{align*}
Thus the last term above is bounded by
	\begin{align*}
	\delta^4\leq C\left(R^{-2}\eta^{-1}\delta^{-1}+(R\eta)^{-1}\right)+CR^{4-\frac{n+1-2\delta_0}{n+1+2\delta_0}}\omega.
	\end{align*}
By our choice of $p_1=2, p_2=15$, we have $R^{4-\frac{n+1-2\delta_0}{n+1+2\delta_0}}\omega=R^{15+5\frac{n+1-2\delta_0}{n+1+2\delta_0}}\ll\delta^4$. So
	\begin{align*}
	\delta^4&\leq C\left(R^{-2}\eta^{-1}\delta^{-1}+(R\eta)^{-1}\right),
	\end{align*}
dividing both sides by $\delta^4\eta^{-1}$ gives
	\begin{align*}
	\eta&\leq C\left(R^{-2}\delta^{-4}\delta^{-1}+R^{-1}\delta^{-4}\right)\\
	&\leq C\delta.
	\end{align*}
Namely, assuming \eqref{DefinitionEta} or not, we have
	\begin{align*}
	\xi_G\leq C\delta,
	\end{align*}
and thus by \eqref{Gradphi}
	\begin{align*}
	\frac{|\nabla\tilde u_0|^2}{2}-W(\tilde u_0)&=\xi_G+G_\delta(\tilde u_0)+\varphi\\
	&\leq C\delta+R^{4-\frac{n+1-2\delta_0}{n+1+2\delta_0}}\omega\\
	&\leq C\delta+\delta^{15+5\frac{n+1-2\delta_0}{n+1+2\delta_0}}\\
	&\leq C\delta.
	\end{align*}
This proves \eqref{eqn_u0L1discrepancydeltabound} and as a consequence \eqref{eqn_L1discrpancydeltabound}.
If $|\tilde u|\geq1-\tau$ in $B_\frac{1}{2}$, then \eqref{eqn_L1discrepancytaubound} follows because the left hand side is less than the second term on the right. So we only need to consider the case there exists $x_0\in B_\frac{1}{2}$ with $\tilde u(x_0)\leq 1-\tau$. By the Sobolev inequality and Calderon--Zygmund estimates we bound $\tilde u$ in the H\"older norm as follows
	\begin{align*}
	\|\tilde u\|_{C^{0,\frac{2\delta_0}{(n+1)+\delta_0}}(B_1)}&\leq\|\tilde u\|_{W^{2,\frac{n+1}{2}+\delta_0}(B_1)}\\
	&\leq \tilde C\left(\|W'(\tilde u)\|_{L^{\frac{n+1}{2}+\delta_0}(B_1)}+\|\tilde f\|_{L^{\frac{n+1}{2}+\delta_0}(B_1)}+\|\tilde u\|_{L^{\frac{n+1}{2}+\delta_0}(B_1)}\right)\\
	&\leq C.
	\end{align*}
Therefore $|\tilde u|\leq1-\frac{\tau}{2}$ and $W(\tilde u)\geq\frac{\tau^2}{4}$ in $B_{\left(\frac{\tau}{2C}\right)^\frac{(n+1)+\delta_0}{2\delta_0}}\subset B_1$. So
	\begin{align*}
	\int_{B_\frac{1}{2}}W(\tilde u)\geq\frac{\tau^2}{4}\left(\frac{\tau}{2\tilde C_2}\right)^\frac{(n+1)[(n+1)+\delta_0]}{2\delta_0}=C\tau^\frac{(n+1)^2+(n+1)\delta_0+4\delta_0}{2\delta_0}.
	\end{align*}
By our choice $p_3=\frac{2\delta_0}{(n+1)^2+(n+1)\delta_0+6\delta_0}$,
	\begin{align*}
	\int_{B_\frac{1}{2}}\left(\frac{|\nabla\tilde u|^2}{2}-W(\tilde u)\right)_+&\leq C\delta\\
	&\leq C\tau^\frac{(n+1)^2+(n+1)\delta_0+6\delta_0}{2\delta_0}\\
	&\leq C\tau\tau^\frac{(n+1)^2+(n+1)\delta_0+4\delta_0}{2\delta_0}\\
	&\leq C\tau\int_{B_\frac{1}{2}}\left(\frac{|\nabla\tilde u|^2}{2}+W(\tilde u)\right),
	\end{align*}
which proves \eqref{eqn_L1discrepancytaubound}.
\end{proof}
													Next we derive energy estimates away from transition regions.
\begin{proposition} [{\cite[Proposition 3.4]{roger2006modified}}]\label{prop_errorterms}
For any $ n\geq 2$ , $0\leq\delta\leq\delta_1$, $\e> 0, u_\e\in C^2(\Omega), f_\e \in C^0(\Omega)$, if
	\begin{align*}
	-\e \Delta u_\e + \frac{W'(u_\e)}{\e} = f_\e \quad \text{ in $\Omega$}
	\end{align*}
and
	\begin{align*}
	\Omega'\subset\subset \Omega, 0 < r \leq d(\Omega',\partial \Omega)
	\end{align*}
then
	\begin{align*}
	&\int_{\{|u_\e|\geq 1 - \delta \} \cap\Omega' }\left( \e |\nabla u_\e|^2 + \frac{W(u_\e)}{\e} + \frac{W'(u_\e)^2}{\e} \right) \\
	&\leq C\delta \int_{\{| u_\e| \leq 1 -\delta \} \cap \Omega} \e |\nabla u_\e|^2 + C \e \int_\Omega |f_\e|^2 + C\left(\frac{\delta}{r} + \frac{\delta^2}{r^2} \right) \e \mathcal L^{n+1}(\Omega)+\frac{C\e}{r^2} \int_{\{|u_\e|\geq 1\} \cap \Omega}W'(u_\e)^2.
	\end{align*}
(Notice the power of $f_\e$ in the above inequality will still be 2 instead of $\frac{n+1}{2}+\delta_0$.)
\end{proposition}
\begin{proof}
Define a continuous function
	\begin{align*}
	g(t)=
\begin{cases}
W'(t),&\text{ for $|t|\geq1-\delta$}\\
	0,&\text{ for $|t|\leq t_0$}\\
	\text{ linear,}&\text{ for $t\in[-1+\delta,-t_0]\cup[t_0,1-\delta]$},
\end{cases}
	\end{align*}
where $t_0=\frac{1}{\sqrt 3}$ is chosen to be the number in $(0,1)$ such that $W''(t_0)=0$. Clearly $|g|\leq |W'|$.
For $\eta\in C_0^1(\Omega)$ satisfying $0\leq\eta\leq 1$, $\eta\equiv1$ on $\Omega'$ and $|\nabla\eta|\leq Cr^{-1}$, we get by integration by parts
	\begin{align}\label{IntegrationByPartErrorTermEstimate}
	\begin{split}
\int_{\Omega} f_\e g(u_\e)\eta^2&=\int_{\Omega}\left(-\e\Delta u_\e+\frac{W'(u_\e)}{\e}\right)g(u_\e)\eta^2\\
	&=\int_{\Omega}\e g'(u_\e)|\nabla u_\e|^2\eta^2+2\int_{\Omega}\e g(u_\e)\eta\langle\nabla u_\e,\nabla\eta\rangle+\int_{\Omega}\frac{W'(u_\e)}{\e}g(u_\e)\eta^2.
\end{split}
	\end{align}
The left hand side of \eqref{IntegrationByPartErrorTermEstimate} can be bounded by
	\begin{align}\label{LHSErrorTermEstimate}
	\int_\Omega f_\e g(u_\e)\eta^2&\leq\frac{\e}{2}\int_\Omega |f_\e|^2+\frac{1}{2\e}\int_\Omega g(u_\e)^2\eta^2\leq\frac{\e}{2}\int_\Omega |f_\e|^2+\frac{1}{2\e}\int_\Omega W'(u_\e)g(u_\e)\eta^2.
	\end{align}
By the definition of $g$ above, we have
	\begin{align*}
	|g(t)|&\leq |g(1-\delta)|=W'(1-\delta)\leq C\delta,\\
	|g'(t)|&\leq \frac{|g(1-\delta)|}{1-\delta}\leq \frac{|g(1-\delta)|}{1-\delta_1}\leq C\delta,
	\end{align*}
for $|t|\leq1-\delta$. Applying these estimates to the second term on the right hand side of \eqref{IntegrationByPartErrorTermEstimate} we get the bound
	\begin{align}\label{RHSErrorTermEstimate}
	&\left|2\int_{\Omega}\e g(u_\e)\eta\langle\nabla u_\e,\nabla\eta\rangle\right|\\\nonumber
&\leq 2\delta\int_{\{|u_\e|\leq1-\delta\}}\e\eta|\nabla u_\e||\nabla\eta|+\left|\int_{\{|u_\e|\geq1-\delta\}}\e W'(u_\e)\langle\nabla u_\e,\nabla\eta\rangle\right|\\\nonumber
&\leq C\delta\int_{\{|u_\e|\leq1-\delta\}}\e|\nabla u_\e|^2+\e\delta r^{-1}\mathcal L^{n+1}(\Omega)+\tau\int_{\{|u_\e|\geq1-\delta\}}\e|\nabla u_\e|^2\eta^2+C\e \tau^{-1} r^{-2}\int_{\{|u_\e|\geq1-\delta\}}W'(u_\e)^2,
	\end{align}
for $\tau>0$.
As $g'(t)=W''(t)\geq C_W>0$ for $|t|\geq1-\delta$, we obtain from \eqref{IntegrationByPartErrorTermEstimate}, \eqref{LHSErrorTermEstimate} and \eqref{RHSErrorTermEstimate}
	\begin{align*}
	C_W\int_{\{|u_\e|\geq1-\delta\}}\e |\nabla u_\e|^2&+\frac{1}{2\e}\int_\Omega W'(u_\e)g(u_\e)\eta^2\\
	&\leq C_W\delta\int_{\{|u_\e|\leq1-\delta\}}\e|\nabla u_\e|^2+\tau\int_{\{|u_\e|\geq1-\delta\}}\e|\nabla u_\e|^2\eta^2+\frac{\e}{2}\int_\Omega |f_\e|^2\\
	&+\left(\delta r^{-1}+C\delta^2\tau^{-1}r^{-2}\right)\mathcal L^{n+1}(\Omega)+C\e\tau^{-1}r^{-2}\int_{\{|u_\e|\geq1\}}W'(u_\e)^2.
	\end{align*}
Choosing $\tau=\frac{C_W}{2}$, and using $W(t)\leq C_WW'(t)^2$ for $|t|\geq1-\delta$ we get
	\begin{align*}
	&\int_{\{|u_\e|\geq1-\delta\}\cap\Omega'}\left(\e|\nabla u_\e|^2+\frac{W(u_\e)}{\e}+\frac{W'(u_\e)^2}{\e}\right)\\
	&\leq C\int_{\{|u_\e|\geq1-\delta\}\cap\Omega'}\left(\e|\nabla u_\e|^2+\frac{W'(u_\e)^2}{\e}\right)\\
	&\leq C\delta\int_{\{|u_\e|\leq1-\delta\}}\e|\nabla u_\e|^2+C\e\int_\Omega |f_\e|^2+C\e\left(\delta r^{-1}+\delta^2r^{-2}\right)\mathcal L^{n+1}(\Omega)\\
	&+C\e r^{-2}\int_{\{|u_\e|\geq1\}}W'(u_\e)^2,
	\end{align*}
which completes the proof.
\end{proof}
The following proposition shows for all $\e$ sufficiently small, if $u_\e$ satisfies the inhomogeneous Allen--Cahn equation then we can control the last term $\int_{\{|u_\e|\geq 1\} \cap \Omega'} W'(u_\e)^2$ in Proposition \ref{prop_errorterms} by applying the proposition inductively.
\begin{proposition} [{\cite[Proposition 3.5]{roger2006modified}}]\label{prop_errorbound2}
For $n\geq 2, \e> 0, u_\e\in C^2(\Omega), f_\e \in C^0(\Omega)$, if
	\begin{align*}
	-\e \Delta u_\e + \frac{W'(u_\e)}{\e} = f_\e \quad \text{ in $\Omega$}
	\end{align*}
and $\Omega'\subset\subset \Omega, 0 < r \leq d(\Omega',\partial \Omega)$ then
	\begin{align*}
	\int_{\{|u_\e|\geq 1\} \cap \Omega'} W'(u_\e)^2 \leq C_k( 1 + r^{-2k}\e^{2k}) \e^2 \int_{\Omega'_{i-1}} |f_\e|^2 + C_k r^{-2k} \e^{2k} \int_{\{|u_\e| \geq 1\} \cap \Omega} W'(u_\e)^2
	\end{align*}
for all $ k \in \mathbb N_0$.
\end{proposition}
\begin{proof}
For any $k\in\mathbb N^+$ we choose a sequence of open sets
	\begin{align*}
	\Omega_i':=\begin{cases}
\Omega&\text{ for }i=0\\\left\{x\in\Omega|d(x,\Omega')<\frac{(k-i)r}{k}\right\}&\text{ for } i=1,...,k-1,\\
	\Omega'&\text{ for } i=k.
\end{cases}
	\end{align*}
This sequence satisfies
	\begin{align*}
	\Omega'=\Omega_k'\subset\subset\Omega_{k-1}'\subset\subset...\subset\subset\Omega_0'=\Omega,
	\end{align*}
with $d(\Omega_i',\Omega_{i-1}')\geq\frac{r}{k}$ for $i=1,...,k$. Applying Proposition \eqref{prop_errorterms} with $\delta=0$, we have
	\begin{align*}
	\int_{\{|u_\e|\geq1\}\cap\Omega_i'}W'(u_\e)^2\leq C\e^2\int_{\Omega'_{i-1}} |f_\e|^2+Ck^2r^{-2}\e^2\int_{\{|u_\e|\geq1\}\cap\Omega_{i-1}'}W'(u_\e)^2,
	\end{align*}
for $i=1,...,k$.
The conclusion is obtained by applying the above inequality inductively $k$ times.
\end{proof}
We conclude with the following integral bound for positive part of discrepancy measure.
\begin{lemma}[{\cite[Lemma 3.1]{roger2006modified}} for all $n$]\label{DiscrepancyEstimateAllScale}
Let $n\geq2$, $0<\delta\leq\delta_1$ (where $\delta_1$ given as in Lemma \ref{lem_L1discrepancydeltabound}), $0<\e\leq\rho$, $\rho_0:=\max\{2,1+\delta^{-M}\e\}\rho$ for some large universal constant $M$. If $u_\e\in C^2(B_{\rho_0}), f_\e\in C^0(B_{\rho_0})$ satisfies \eqref{PFVe} in $B_{\rho_0}(0)$ then the positive part of the discrepancy measure satisfies
	\begin{align*}
	\rho^{-n}\int_{B_\rho}\left(\frac{\e|\nabla u_\e|^2}{2}-\frac{W(u_\e)}{\e}\right)_+&\leq C\delta^{p_3}\rho^{-n}\int_{B_{2\rho}}\left(\frac{\e|\nabla u_\e|^2}{2}+\frac{W(u_\e)}{\e}\right)+C\delta^{-M}\e\rho^{-n}\int_{B_{\rho_0}}|f_\e|^2\\
	&+C\delta^{-M}\rho^{-n}\int_{B_{\rho_0}\cap\{|u_\e|\geq1\}}\frac{W'(u_\e)^2}{\e}+C\left(\frac{\e}{\rho}\right)\delta.
	\end{align*}
\end{lemma}
\begin{proof}
We prove the case $0<\e\leq\rho=1$. The case for other $\rho>0$ follows by rescaling to $\rho=1$.
For $0<\delta\leq\delta_1$ we choose $R(\delta)=\frac{1}{\delta^{p_1}}$ and $\omega(\delta)=C_\omega\delta^{p_2}$ as in Lemma \ref{lem_L1discrepancydeltabound}. Let $\{x_i\}_{i\in \mathbf I}\subset B_1, \mathbf I\subset\mathbb N$ be a maximal collection of points satisfying
	\begin{align*}
	\min_{i\neq j}|x_i-x_j|\geq\frac{\e}{2}.
	\end{align*}
Since $\e\leq1$, we have
	\begin{align*}
	&B_1(0)\subset\cup_{i\in\mathbf I}\bar{B}_{\frac{\e}{2}}(x_i)\subset B_{\frac{3}{2}}(0),\\
	&\sum_{i\in\mathbf I}\chi_{B_\e(x_i)}\leq C_n\chi_{B_2(0)},\\
	&\sum_{i\in\mathbf I}\chi_{B_{2R\e}(x_i)}\leq C_nR^{n+1}\chi_{B_{1+2R\e}(0)}.
	\end{align*}
For $i\in\mathbf I$ and $x\in B_{2R}$, we define the rescaled and translated functions as
	\begin{align*}
	\tilde u_i(x)&:=u_\e(x_i+\e x),\\
	\tilde f_i(x)&:=\e f_\e(x_i+\e x),
	\end{align*}
which satisfy the rescaled equation
	\begin{align}\label{RescaledEquationUi}
	-\Delta\tilde u_i+W'(\tilde u_i)=\tilde f_i,\quad\text{ in $B_{2R}(0)$}.
	\end{align}
For $\tilde u_i, \tilde f_i$ to be well-defined, we choose $M\geq5n+6$ and $\delta_1\leq\frac{1}{2}$ so that
	\begin{align*}
	x_i+\e x\in B_{1+2R\e}(0)\subset B_{1+\delta^{-M}\e}(0)\subset B_{\rho_0}(0).
	\end{align*}
We decompose the index set $\mathbf I$ into
	\begin{align*}
	\mathbf I_1&:=\{i\in\mathbf I:\|f_\e\|_{L^{\frac{n+1}{2}+\delta_0}(B_{2R\e}(x_i))}<\e^{\frac{n+1}{2}-1+\delta_0}\omega, \|(|u_\e|-1)_+\|_{L^1(B_{2R\e}(x_i))}<C_\omega\e^{n+1}\},\\
	\mathbf I_2&:=\mathbf I\setminus\mathbf I_1.
	\end{align*}
For $i\in\mathbf I_1$, we have
	\begin{align*}
	&\|\tilde f_i\|_{L^{\frac{n+1}{2}+\delta_0}(B_{2R}(0))}=\e^{-\frac{n+1}{2}-\delta_0}\|\e f_\e\|_{L^{\frac{n+1}{2}+\delta_0}(B_{2R\e}(x_i))}<\omega\leq C_\omega,\\
	&\|(|\tilde u_i|-1)_+\|_{L^1(B_{2R}(x_i))}=\e^{-n-1}\|(| u_\e|-1)_+\|_{L^1(B_{2R\e}(x_i))}<C_\omega.
	\end{align*}
By the condition $\|u\|_{L^\infty}\leq c_0$ in the condition of Theorem \ref{main}, and choosing $C_\omega$ sufficiently small, we have
	\begin{align*}
	\|\tilde u_i\|_{L^\infty(B_R)}\leq 1+C\cdot C_\omega\leq 2.
	\end{align*}
Applying Lemma \ref{lem_L1discrepancydeltabound} to $\tilde u_i$ gives (with $p_3$ from Lemma \ref{lem_L1discrepancydeltabound})
	\begin{align*}
	&\int_{B_\frac{1}{2}}\left(\frac{|\nabla\tilde u_i|^2}{2}-W(\tilde u_i)\right)_+\\
	&\leq C{\delta^{p_3}}\int_{B_\frac{1}{2}}\left(\frac{|\nabla\tilde u_i|^2}{2}+W(\tilde u_i)\right)+\int_{B_\frac{1}{2}\cap\{|\tilde u_i|\geq1-\delta\}}\frac{|\nabla \tilde u_i|^2}{2}.
	\end{align*}
Rescaling back, we get
	\begin{align*}
	&\int_{B_\frac{\e}{2}(x_i)}\left(\frac{\e|\nabla u_\e|^2}{2}-\frac{W( u_\e)}{\e}\right)_+\\
	&\leq C{\delta^{p_3}}\int_{B_\frac{\e}{2}(x_i)}\left(\frac{\e|\nabla u_\e|^2}{2}+\frac{W( u_\e)}{\e}\right)+\int_{B_\frac{\e}{2}(x_i)\cap\{| u_\e|\geq1-\delta\}}\frac{\e|\nabla u_\e|^2}{2}.
	\end{align*}
Summing over $i\in\mathbf I_1$ and noticing $B_{\frac{\e}{2}}(x_i)$ are disjoint, we get
	\begin{align}\label{SumI1}
	&\sum_{i\in\mathbf I_1}\int_{B_\frac{\e}{2}(x_i)}\left(\frac{\e|\nabla u_\e|^2}{2}-\frac{W( u_\e)}{\e}\right)_+\\\nonumber
&\leq C\delta^{p_3}\int_{B_\frac{3}{2}(0)}\left(\frac{\e|\nabla u_\e|^2}{2}+\frac{W( u_\e)}{\e}\right)+C\int_{B_\frac{3}{2}(0)\cap\{|u_\e|\geq1-\delta\}}\frac{\e|\nabla u_\e|^2}{2}\\\nonumber
&\leq C\delta^{p_3}\int_{B_2(0)}\left(\frac{\e|\nabla u_\e|^2}{2}+\frac{W( u_\e)}{\e}\right)+C\e\int_{B_2(0)}|f_\e|^2+C\e\left(\delta+\int_{B_2(0)\cap\{|u_\e|\geq1\}}W'(u_\e)^2\right),
	\end{align}
where we used Proposition \ref{prop_errorterms} in the last line.
Since for $n\geq3$ (the $n=2$ case requires $\delta_0\geq\frac{1}{2}$, but has already been addressed in \cite{roger2006modified})
	\begin{align*}
	W'(t)^2\geq 4t^2(1+t)^2(1-t)^2\geq C_Wt^2(|t|-1)^2\geq C_W(|t|-1)_+^{\frac{n+1}{2}+\delta_0}.
	\end{align*}
Thus for $i\in\mathbf I_2$ (at least one of the bounds in $\mathbf I_1$ does not hold), we have
	\begin{align*}
	C_\omega&\leq \int_{B_{2R}(0)}(|\tilde u_i|-1)_+^{\frac{n+1}{2}+\delta_0}+\omega^{-2}\int_{B_{2R}}\tilde f_i^2\\
	&\leq C\int_{B_{2R}(0)\cap\{|\tilde u_i|\geq1\}}W'(\tilde u_i)^2+\omega^{-2}\int_{B_{2R}}\tilde f_i^2.
	\end{align*}
By elliptic estimates applied to the rescaled equation \eqref{RescaledEquationUi}, we get
	\begin{align*}
	\int_{B_{\frac{1}{2}}}|\nabla\tilde u_i|^2&\leq \tilde C\int_{B_1}\left(W'(\tilde u_i)^2+\tilde u_i^2+\tilde f_i^2\right)\\
	&\leq \tilde C\int_{B_{2R}}\left(W'(\tilde u_i)^2+\tilde f_i^2\right)+\tilde C\omega_n c_0^2C_\omega^{-1}C_\omega\\
	&\leq C\int_{B_{2R}}\left(W'(\tilde u_i)^2+\omega^{-2}\tilde f_i^2\right),
	\end{align*}
where we used $\|\tilde u_i\|_{L^\infty}\leq c_0$. Rescaling back gives
	\begin{align*}
	\int_{B_\frac{\e}{2}(x_i)}\e|\nabla u_\e|^2\leq C\int_{B_{2R\e}(x_i)}\left(\frac{W'(u_\e)^2}{\e}+\e\omega^{-2}|f_\e|^2\right).
	\end{align*}
Then summing over $i\in\mathbf I_2$ we get
	\begin{align}\label{SumI2}
	\sum_{i\in\mathbf I_2}\int_{B_\frac{\e}{2}(x_i)}\e|\nabla u_\e|^2&\leq\sum_{i\in\mathbf I_2}C\int_{B_{2R\e}(x_i)}\left(\frac{W'(u_\e)^2}{\e}+\e\omega^{-2}|f_\e|^2\right)\\\nonumber
&\leq CR^{n+1}\int_{B_{1+2R\e}(0)}\left(\frac{W'(u_\e)^2}{\e}+\e\omega^{-2}|f_\e|^2\right)\\\nonumber
&\leq C\delta^{-M}\int_{B_{1+\delta^{-M}\e}(0)}\left(\frac{W'(u_\e)^2}{\e}+\e |f_\e|^2\right),
	\end{align}
for large enough $M$ since both $R=\delta^{-p_1}$ and $\omega=\delta^{p_2}$ are fixed powers of $\delta$.
Combining \eqref{SumI1} and \eqref{SumI2} we get
	\begin{align*}
	&\int_{B_1}\left(\frac{\e|\nabla u_\e|^2}{2}-\frac{W(u_\e)}{\e}\right)_+\\
	&\leq \sum_{i\in\mathbf I_1}\int_{B_\frac{\e}{2}(x_i)}\left(\frac{\e|\nabla u_\e|^2}{2}-\frac{W( u_\e)}{\e}\right)_++\sum_{i\in\mathbf I_2}\int_{B_\frac{\e}{2}(x_i)}\frac{\e|\nabla u_\e|^2}{2}\\
	&\leq C\delta^{p_3}\int_{B_2(0)}\left(\frac{\e|\nabla u_\e|^2}{2}+\frac{W( u_\e)}{\e}\right)+C\e\int_{B_2(0)}|f_\e|^2+C\e\left(\delta+\int_{B_2(0)\cap\{|u_\e|\geq1\}}W'(u_\e)^2\right)\\
	&+C\delta^{-M}\int_{B_{1+\delta^{-M}\e}(0)}\left(\frac{W'(u_\e)^2}{\e}+\e |f_\e|^2\right)\\
	&\leq C\delta^{p_3}\int_{B_2(0)}\left(\frac{\e|\nabla u_\e|^2}{2}+\frac{W( u_\e)}{\e}\right)+C\e\delta+C\e\delta^{-M}\int_{B_{\max\{2,1+\delta^{-M}\e\}}(0)}|f_\e|^2\\
	&+C\delta^{-M}\int_{B_{\max\{2,1+\delta^{-M}\e\}}(0)}\frac{W'(u_\e)^2}{\e}.
	\end{align*}
This completes the proof for $\rho=1$ and rescaling gives the cases for other $\rho>0$.
\end{proof}
As a result of these, we have the $L^1$ convergence of the positive part of the discrepancy measure as $\e\rightarrow0$.
\begin{lemma}\label{L1PositiveDiscrepancyVanishing}
If we consider $ \xi_\e=\xi_{\e,+} - \xi_{\e,-}$ the decomposition of $\xi_\e$ into positive and negative variations then
	\begin{align*}
	\xi_{\e,+}\rightarrow0 \quad \text{ as $\e\rightarrow0$}.
	\end{align*}
Furthermore this shows $\xi\leq0$.
\end{lemma}
\begin{proof}
For $B_{2\rho}=B_{2\rho}(x) \subset \Omega'\subset\subset \Omega, 0 < \delta<\delta_0$ and $0<\e\leq \delta^M$ then applying Lemma \ref{DiscrepancyEstimateAllScale} we have
	\begin{align}\label{DiscrepancyEstimateAllScale2}
	\begin{split}
\int_{B_\rho}\left(\frac{\e|\nabla u_\e|^2}{2}-\frac{W(u_\e)}{\e}\right)_+&\leq C\delta^{p_3}\int_{B_{2\rho}}\left(\frac{\e|\nabla u_\e|^2}{2}+\frac{W(u_\e)}{\e}\right)+C\delta^{-M}\e\int_{B_\rho}|f_\e|^2\\
	&\quad +C\delta^{-M}\int_{B_\rho\cap\{|u_\e|\geq1\}}\frac{W'(u_\e)^2}{\e}+C\left(\frac{\e}{\rho}\right)\delta \rho^n.
\end{split}
	\end{align}
Proposition \ref{prop_errorbound2} gives us
	\begin{align*}
	\int_{\{|u_\e|\geq 1\} \cap B_\rho} W'(u_\e)^2 \leq C_k( 1 + \rho^{-2k}\e^{2k}) \e^2 \int_{B_{2\rho}} |f_\e|^2 + C_k \rho^{-2k} \e^{2k} \int_{\{|u_\e| \geq 1\} \cap B_{2\rho}} W'(u_\e)^2
	\end{align*}
for all $ k \in \mathbb N_0$. Choosing $k=2$ and applying the bound
	\begin{align*}
	\int_{\{|u_\e| \geq 1\} \cap B_{2\rho}} W'(u_\e)^2 \leq C(\Omega')
	\end{align*}
and inserting these estimates into \eqref{DiscrepancyEstimateAllScale2}, we get
	\begin{align*}
	\int_{B_\rho}\left(\frac{\e|\nabla u_\e|^2}{2}-\frac{W(u_\e)}{\e}\right)_+&\leq C\delta^{p_3}\int_{B_{2\rho}}\left(\frac{\e|\nabla u_\e|^2}{2}+\frac{W(u_\e)}{\e}\right)+C(\delta^{-M}\e+\e^2)\int_{B_\rho}|f_\e|^2\\
	&\quad +C\delta^{-M}\e^3+C\left(\frac{\e}{\rho}\right)\delta \rho^n.
	\end{align*}
By the H\"older inequality with exponent $q_0/2$, we estimate
	\begin{align}\label{WillmoreTermBound}
	\e \int_{B_{r/2}} |f_\e|^2&= \e^2 \int_{B_{r/2}}\left(\frac{f_\e}{|\e |\nabla u_\e|} \right)^2 \e |\nabla u_\e|^2\\\nonumber
&\leq \e^2 \left(\int_{B_{r/2}}\left|\frac{f_\e}{|\e |\nabla u_\e|} \right|^{q_0} \e |\nabla u_\e|^2\right)^{2/q_0}\left( \int_{B_{r/2}} \e |\nabla u_\e|^2 \right)^\frac{q_0}{q_0-2}\\\nonumber
&\leq \e^2 C(\Lambda_0, E_0),
	\end{align}
and obtain
	\begin{align*}
	\int_{B_\rho}\left(\frac{\e|\nabla u_\e|^2}{2}-\frac{W(u_\e)}{\e}\right)_+&\leq \tilde C\delta^{p_3}+\tilde C\delta^{-M}\e^2+\tilde C\e^2 +\tilde C\delta^{-M}\e^3+\tilde C\delta\e\\
	&\leq \tilde C \delta.
	\end{align*}
Letting $\e\rightarrow0$ we get $\xi_{\e,+}(B_\rho)\rightarrow 0$.
\end{proof}
\section{Rectifiability}\label{sec:rectifiability}
We will proceed by proving upper and lower density bounds for the energy measure. Combining the estimates obtained in the previous section, we get an upper bound on the density ratio of the limit energy measure.
\begin{theorem}\label{DensityUpperBound}
If we consider $\Omega'\subset\subset\Omega$ and $r_0(\Omega'):=\min\left\{1,\frac{d(\Omega',\partial\Omega)}{2}\right\} $ then for all $x_0\in\Omega', 0<r<r_0$ there exists a function $\phi(\e)$ with $\lim_{\e \rightarrow 0} \phi(\e)=0$ such that
	\begin{align}\label{eqn_eDensityUpperBound}
	r^{-n}\mu_\e(B_r(x_0))\leq C(\Lambda_0,\Omega') + \frac{\phi(\e)}{r^n}.
	\end{align}
Letting $\e\rightarrow 0$ we get
	\begin{align*}
	r^{-n}\mu(B_r(x_0))\leq C(\Lambda_0,\Omega'),
	\end{align*}
where $\mu=\lim_{\e\rightarrow 0} \mu_\e$ is the weak-* limit of $\mu_\e=\left(\frac{\e|\nabla u_\e|^2}{2}+\frac{W(u_\e)}{\e}\right)dx$ in the sense of Radon measures.
\end{theorem}
\begin{proof} For the sake of simplicity we set $ x_0=0$ and set $ B_\rho(0)=B_\rho$.
By the almost monotonicity formula \eqref{AlmostMonotonicity}, Lemma \ref{L1PositiveDiscrepancyVanishing} and Holder's inequality
	\begin{align}\label{MonotonicityWithLpMeanCurvature}
	&\frac{d}{d\rho}\left(\frac{\mu_\e(B_\rho)}{\rho^n}\right)=-\frac{1}{\rho^{n+1}}\xi_\e(B_{\rho})+\frac{\e}{\rho^{n+2}}\int_{\partial B_\rho}\langle x,\nabla u\rangle^2-\frac{1}{\rho^{n+1}}\int_{B_\rho}\langle x,\nabla u\rangle f_\e.
	\end{align}
We estimate the last term above as follows
	\begin{align}\label{ConicalityTermBound}
	\frac{1}{\rho^{n+1}}\left| \int_{B_\rho}\langle x,\nabla u\rangle f_\e\right|&\leq \frac{1}{\rho^{n+1}}\int_{B_{\rho}}|\langle x,\nabla u\rangle|\left|\frac{f_\e}{\e|\nabla u|}\right|\e|\nabla u|\\\nonumber
&\leq \frac{1}{\rho^n}\int_{B_{\rho}}\left|\frac{f_\e}{\e|\nabla u|}\right|\e|\nabla u|^2\\\nonumber
&\leq\frac{1}{\rho^n}\left(\int_{B_{\rho}}\left|\frac{f_\e}{\e|\nabla u|}\right|^{q_0}\e|\nabla u|^2\right)^\frac{1}{q_0}\left(\int_{B_\rho}\e|\nabla u|^2\right)^\frac{q_0-1}{q_0}\\\nonumber
&\leq \left(\frac{1}{\rho^n}\right)^\frac{1}{q_0}\left(\int_{B_{\rho}}\left|\frac{f_\e}{\e|\nabla u|}\right|^{q_0}\e|\nabla u|^2\right)^\frac{1}{q_0}\left[\frac{1}{\rho^n}\right]^\frac{q_0-1}{q_0}\left(2\mu_\e(B_\rho)\right)^\frac{q_0-1}{q_0}\\\nonumber
&\leq C(\Lambda_0)\rho^{-\frac{n}{q_0}}\left(\frac{\mu_\e(B_\rho)}{\rho^n}\right)^\frac{q_0-1}{q_0}\\\nonumber
&\leq C(\Lambda_0)\rho^{-\frac{n}{q_0}}\left(1+\frac{\mu_\e(B_\rho)}{\rho^n}\right)
	\end{align}
where we used the inequality $a^{1-\frac{1}{q_0}} \leq 1+a$ which holds for all $ a \geq 0$.
Inserting this inequality into \eqref{MonotonicityWithLpMeanCurvature} and discarding the positive second term on the right had side, we get
	\begin{align}\label{MonotonicityWithLpMeanCurvature2}
	\frac{d}{d\rho}\left(1+\frac{\mu_\e(B_\rho)}{\rho^n}\right)&=\frac{d}{d\rho}\left(\frac{\mu_\e(B_\rho)}{\rho^n}\right)\geq -\frac{1}{\rho^{n+1}}\xi_\e(B_{\rho})- C(\Lambda_0)\rho^{-\frac{n}{q_0}}\left(1+\frac{\mu_\e(B_\rho)}{\rho^n}\right).
	\end{align}
 Multiplying both sides by $ \exp\left( \int C(\Lambda_0)\rho^{-\frac{n}{q_0}} d\rho\right)= \exp\left(\frac{q_0}{q_0-n} C(\Lambda_0)\rho^{1-\frac{n}{q_0}}\right) $ we have
	\begin{align*}
	\frac{d}{d\rho}\left[ \exp\left(\frac{q_0}{q_0-n} C(\Lambda_0)\rho^{1-\frac{n}{q_0}}\right) \left(1+\frac{\mu_\e(B_\rho)}{\rho^n}\right)\right]\geq - \exp\left(\frac{q_0}{q_0-n} C(\Lambda_0)\rho^{1-\frac{n}{q_0}}\right) \frac{\xi_\e(B_{\rho})}{\rho^{n+1}}.
	\end{align*}
Integrating from $r$ to $r_0$ gives
	\begin{align*}
	&\exp\left(\frac{q_0}{q_0-n} C(\Lambda_0)r_0^{1-\frac{n}{q_0}}\right) \left(1+\frac{\mu_\e(B_{r_0})}{r_0^n}\right)- \exp\left(\frac{q_0}{q_0-n} C(\Lambda_0)r^{1-\frac{n}{q_0}}\right) \left(1+\frac{\mu_\e(B_r)}{r^n}\right)\\
	&\geq -\int_r^{r_0} \exp\left(\frac{q_0}{q_0-n} C(\Lambda_0)\rho^{1-\frac{n}{q_0}}\right) \frac{\xi_{\e,+}(B_{\rho})}{\rho^{n+1}}\\
	&\geq -\exp\left(\frac{q_0}{q_0-n} C(\Lambda_0)r_0^{1-\frac{n}{q_0}}\right)\int_r^{r_0} \frac{\xi_{\e,+}(B_{\rho})}{\rho^{n+1}}.
	\end{align*}
Namely
	\begin{align}\label{AlmostMonotonicityWithExponentialConstant}
	\exp\left(\frac{q_0}{q_0-n} C(\Lambda_0)r_0^{1-\frac{n}{q_0}}\right) \left(1+\frac{\mu_\e(B_{r_0})}{r_0^n}\right)-\frac{\mu_\e(B_r)}{r^n}&\geq- C(\Lambda_0,\Omega')\int_r^{r_0} \frac{\xi_{\e,+}(B_{\rho})}{\rho^{n+1}}\\\nonumber
&\geq- C(\Lambda_0,\Omega')\int_r^{r_0} \frac{\xi_{\e,+}(B_{r_0})}{\rho^{n+1}},
	\end{align}
where we used $\exp\left(\frac{q_0}{q_0-n} C(\Lambda_0)r^{1-\frac{n}{q_0}}\right)>1$ for $r>0$.
Passing to the limit as $\e\rightarrow0$ and using Lemma \ref{L1PositiveDiscrepancyVanishing}, we have
	\begin{align*}
	\frac{\mu(B_r)}{r^n}\leq C(\Lambda_0,\Omega',n,q_0).
	\end{align*}
					\end{proof}
Next, we obtain estimates of the discrepancy measure for each $\e$.
\begin{proposition}\label{prop_DiscrepancyEnergyRatioLowerBound}
Let $\delta=\rho^\gamma,\e\leq\rho\leq r$ for $0<\gamma<\frac{1}{M}\leq\frac{1}{2}$, we have $\delta^{-M}\e\leq\rho^{1-M\gamma}\leq1$. For $ B_{3\rho^{1-\beta}}(x)\subset\subset\Omega$, we have
	\begin{align}\label{eqn_DiscrepancyUpperBound}
	\begin{split}
\rho^{-n-1} \xi_{\e,+}(B_\rho(x))&\leq C\rho^{p_3\gamma-n-1}\mu_\e(B_{2\rho}(x))\\
	&+\tilde C_k\e\rho^{-M\gamma-n-1}\int_{B_{3\rho^{1-\beta}}(x)}|f_\e|^2+\tilde C_\beta\e\rho^{\gamma-2}\left(1+\int_{\{|u_\e| \geq 1\} \cap B_{3r^{1-\beta}}(x)} W'(u_\e)^2\right).
\end{split}
	\end{align}
	\end{proposition}
\begin{proof}
For $0<\gamma<\frac{1}{M}\leq\frac{1}{2}$, by choosing $ \delta^{-M}\e \leq \rho^{1-M\gamma}\leq 1$ we get $ \max \{2, 1+ \delta^{-M}\e\} =2$.
Therefore substituting $\delta = \rho^\gamma$ into Lemma \ref{DiscrepancyEstimateAllScale} we have
	\begin{align*}
	\begin{split}
\rho^{-n-1}\xi_{\e,+}(B_\rho)&= \rho^{-n-1}\int_{B_\rho(x)}\left(\frac{\e|\nabla u_\e|^2}{2}-\frac{W(u_\e)}{\e}\right)_+\\\nonumber
&\leq C\rho^{p_3\gamma-n-1}\int_{B_{2\rho}(x)}\left(\frac{\e|\nabla u_\e|^2}{2}+\frac{W(u_\e)}{\e}\right)+C\e\rho^{-M\gamma-n-1}\int_{B_{2\rho}(x)}|f_\e|^2\\
	&+C\e^{-1}\rho^{-M\gamma-n-1}\int_{B_{2\rho}(x)\cap\{|u_\e|\geq1\}}W'(u_\e)^2+C\e\rho^{\gamma-2}.
\end{split}
	\end{align*}
On the other hand we have by Proposition \ref{prop_errorbound2} with $r:=d(B_{2\rho}(x),\partial B_{3\rho^{1-\beta}}(x))=3 \rho^{1-\beta} - 2 \rho \geq\rho^{1-\beta}$
	\begin{align*}
	\int_{\{|u_\e|\geq 1\} \cap B_{2\rho}} W'(u_\e)^2 \leq C_k( 1 + \rho^{-2k(1-\beta)}\e^{2k}) \e^2 \int_{B_{3\rho^{1-\beta}} } |f_\e|^2 + C_k \rho^{-2k(1-\beta)} \e^{2k} \int_{\{|u_\e| \geq 1\} \cap B_{3\rho^{1-\beta}}} W'(u_\e)^2.
	\end{align*}
Substituting this into our above estimate, we get
	\begin{align*}
	\rho^{-n-1}\xi_{\e,+}(B_\rho)&\leq C\rho^{p_3\gamma-n-1}\int_{B_{2\rho}(x)}\left(\frac{\e|\nabla u_\e|^2}{2}+\frac{W(u_\e)}{\e}\right)+\tilde C_k\e\rho^{-M\gamma-n-1}\int_{B_{3\rho^{1-\beta}}(x)}|f_\e|^2\\\nonumber
&+C\e^{-1}\rho^{-M\gamma-n-1}\tilde C_k\rho^{2k\beta-2k}\e^{2k} \int_{\{|u_\e| \geq 1\} \cap B_{3\rho^{1-\beta}}(x)} W'(u_\e)^2+C\e\rho^{\gamma-2}\\\nonumber
&\leq C\rho^{p_3\gamma-n-1}\mu_\e(B_{2\rho}(x))+\tilde C_k\e\rho^{-M\gamma-n-1}\int_{B_{3\rho^{1-\beta}}(x)}|f_\e|^2\\
	&+C\left(\e\rho^{\gamma-2}+\e^{-1}\rho^{-M\gamma-n-1}\e^{2k\beta}\int_{\{|u_\e| \geq 1\} \cap B_{3\rho^{1-\beta}}(x)} W'(u_\e)^2\right)\\\nonumber
&\leq C\rho^{p_3\gamma-n-1}\mu_\e(B_{2\rho}(x))+\tilde C_{k,\beta}\e\rho^{-M\gamma-n-1}\int_{B_{3\rho^{1-\beta}}(x)}|f_\e|^2\\
	&+\tilde C_\beta\e\rho^{\gamma-2}\left(1+\int_{\{|u_\e| \geq 1\} \cap B_{3\rho^{1-\beta}}(x)} W'(u_\e)^2\right),
	\end{align*}
where we have chosen $-M\gamma-n+2k\beta +1 \geq \gamma -2$ or $k>\frac{\gamma-2+M\gamma+n+1}{2\beta}$ sufficiently large.
\end{proof}
In the following theorem we prove the density lower bound for the limit measure.
\begin{theorem}\label{DensityLowerBound}
There exists $\bar\theta>0$ such that for any $\Omega'\subset\subset\Omega$ and $r_1(\Omega')\leq \frac{d(\Omega',\partial\Omega)}{2}$ sufficiently small, we have
	\begin{align*}
	r^{-n}\mu(B_r(x))\geq\bar\theta-Cr^\gamma,
	\end{align*}
for some $\gamma>0$, and all $x\in\spt\mu\cap\Omega'$ and $0<r\leq r_1$.
In particular,
	\begin{align*}
	\theta^n_*(\mu)\geq\frac{\bar\theta}{\omega_n}
	\end{align*}
for $\mu$-a.e. in $\Omega$.
\end{theorem}
\begin{proof}
Without loss of generality, we assume $0\in\spt\mu\cap\Omega'$ and want to prove a density lower bound at $0$.
We first integrate \eqref{MonotonicityWithLpMeanCurvature2} from $s$ to $r$. 	\begin{align}\label{IntegralMonotonicitywithL2MeanCurvature}
	\begin{split}
&\frac{\mu_\e(B_r(x))}{r^n}-\frac{\mu_\e(B_s(x))}{s^n}\\
	&\geq-\int_s^r\frac{1}{\rho^{n+1}}\xi_{\e,+}(B_{\rho}(x))d\rho-\int_s^rC(\Lambda_0)\rho^{-\frac{n}{q_0}}\left(\frac{\mu_\e(B_\rho(x))}{\rho^n}\right)^\frac{q_0-1}{q_0}.
\end{split}
	\end{align}
By \eqref{eqn_DiscrepancyUpperBound} in Proposition \ref{prop_DiscrepancyEnergyRatioLowerBound}, the discrepancy term
	\begin{align}\label{eqn_IntegralNegativeDiscrepancyLowerBound}
	\begin{split}
-\int_s^r \rho^{-n-1} \xi_{\e,+}(B_\rho(x))&\geq -\int_s^rC\rho^{p_3\gamma-n-1}\mu_\e(B_{2\rho}(x))-\int_s^r\tilde C_k\e\rho^{-M\gamma-n-1}\int_{B_{3\rho^{1-\beta}}(x)}|f_\e|^2\\&-\int_s^r\tilde C_\beta\e\rho^{\gamma-2}\left(1+\int_{\{|u_\e| \geq 1\} \cap \Omega} W'(u_\e)^2\right).
\end{split}
	\end{align}
By the $\e$-Upper Density Bound \eqref{eqn_eDensityUpperBound} we get
 	\begin{align*}
	-\int_s^r\rho^{p_3\gamma-n-1}\mu_\e(B_{2\rho}(x))&= - \int_s^r2^n\rho^{p_3\gamma-1} \frac{\mu_\e(B_{2\rho}(x)) }{( 2 \rho)^n} \\
	&\geq - \int_s^r2^n\rho^{p_3\gamma-1} \left(C(\Lambda_0,\Omega') + \frac{\phi(\e)}{\rho^n} \right)\\
	&\geq - C(\Lambda_0,\Omega') \left( r^{p_3\gamma} - s^{p_3\gamma}\right)- \frac{\phi(\e)}{p_3\gamma-n+1}\left( r^{p_3\gamma-n} - s^{p_3\gamma -n } \right).
	\end{align*}
The last term in \eqref{eqn_IntegralNegativeDiscrepancyLowerBound} may be estimated as follows
	\begin{align*}
	-\int_s^r\tilde C_\beta\e\rho^{\gamma-2}\left(1+\int_{\{|u_\e| \geq 1\} \cap \Omega} W'(u_\e)^2\right)\geq -\tilde{C}_\beta \int_s^t \rho^{\gamma-1} d \rho \leq -\tilde{C}_\beta( r^\gamma-s^\gamma).
	\end{align*}
Using the bound
	\begin{align*}
	\left(\frac{\mu_\e(B_\rho(x))}{\rho^n}\right)^\frac{q_0-1}{q_0} \leq \left(1+\frac{\mu_\e(B_\rho(x))}{\rho^n}\right)
	\end{align*}
and the $\e$-Upper Density Bound \eqref{eqn_eDensityUpperBound}, we get
	\begin{align*}
	-\int_s^rC(\Lambda_0)\rho^{-\frac{n}{q_0}}\left(\frac{\mu_\e(B_\rho(x))}{\rho^n}\right)^\frac{q_0-1}{q_0}&\geq -\int_s^r C(\Lambda_0,\Omega')\rho^{-\frac{n}{q_0}} \left(1+\frac{\mu_\e(B_\rho(x))}{\rho^n}\right)\\
	&\geq -\int_s^r C(\Lambda_0,\Omega')\rho^{-\frac{n}{q_0}} \left(1+C(\Lambda_0,\Omega') + \frac{\phi(\e)}{\rho^n}\right)\\
	&\geq -C(\Lambda_0,\Omega')\left( r^{1 - \frac{n}{q_0}}-s^{1-\frac{n}{q_0}}\right)\\
	&- C(\Lambda_0,\Omega')\phi(\e) \left( r^{1 -n - \frac{n}{q_0}}-s^{1-n-\frac{n}{q_0}}\right).
	\end{align*}
Thus, plug all the above estimates of terms in \eqref{IntegralMonotonicitywithL2MeanCurvature}, we get
	\begin{align}\label{IntegralMonotonicitywithL2MeanCurvature2}
	\begin{split}
\frac{\mu_\e(B_r(x))}{r^n}-\frac{\mu_\e(B_s(x))}{s^n}&\geq - C(\Lambda_0,\Omega') \left( r^{p_3\gamma} - s^{p_3\gamma}\right)- \frac{\phi(\e)}{p_3\gamma-n+1}\left( r^{p_3\gamma-n} - s^{p_3\gamma -n } \right)\\
	&-\int_s^r\tilde C_{\beta}\e\rho^{-M\gamma-n-1}\left(\int_{B_{3\rho^{1-\beta}}(x)}|f_\e|^2\right)d\rho-\tilde{C}_\beta(r^\gamma-s^\gamma)\\
	&-C(\Lambda_0,\Omega')\left( r^{1 - \frac{n}{q_0}}-s^{1-\frac{n}{q_0}}\right) - C(\Lambda_0,\Omega')\phi(\e) \left( r^{1-n - \frac{n}{q_0}}-s^{1-n-\frac{n}{q_0}}\right).
\end{split}
	\end{align}
Next, we estimate the term $\int_s^r\tilde C_\beta\e\rho^{-M\gamma-n-1}\left(\int_{B_{3\rho^{1-\beta}}(x)}|f_\e|^2\right)d\rho$ in the following claim.
\begin{claim}
There exists $x\in B_\frac{r}{2}$ such that
	\begin{align}\label{eqn_LowerUpperBound}
	\e^{-n}\mu_\e(B_\e(x))\geq2\bar\theta_0>\bar\theta_0\geq \int_\e^{\frac{r}{4}}\tilde C_\beta\e\rho^{-M\gamma-n-1}\left(\int_{B_{3\rho^{1-\beta}}(x)}|f_\e|^2\right)d\rho,
	\end{align}
for some universal constant $\bar\theta_0>0$.
\end{claim}
\begin{proof}[Proof of Claim]
Consider a point $x\in B_\frac{r}{2}$ with $|u_\e(x)|\leq 1-\tau$, for some $0<\tau<1$.
We can assume $\e^{-n}\mu_\e(B_\e(x))\leq1$(otherwise the conclusion automatically follows), and so
	\begin{align*}
	\e^{-n-1}\int_{B_\e(x)}u_\e^p\leq\e^{-n-1}\int_{B_\e(x)}c_0^p\leq c_0^p\omega_{n+1}, \forall p>1.
	\end{align*}
From Theorem \ref{GradientBound} we have
	\begin{align*}
	\e^\frac{1}{2}\| u\|_{C^{0,\frac{1}{2}}(B_{1-\e}(x))}\leq C,
	\end{align*}
and thus
	\begin{align*}
	|u_\e|\leq 1-\frac{\tau}{2}, \quad \text{ in $B_{\frac{\tau^2\e}{4C^2}}(x)$}.
	\end{align*}
So since $W(t) =(1-t^2)^2 =(1+t)^2(1-t)^2$ we find in $B_{\frac{\tau^2\e}{4C^2}}(x)$
	\begin{align*}
	W(u_\e)=(1+|u_\e|)^2(1-|u_\e|)^2 \geq \frac{\tau^2}{4}
	\end{align*}
	\begin{align}\label{LowerBoundAtScaleEpsilon}
	\e^{-n}\mu_\e(B_\e(x))\geq\e^{-n}\int_{B_{\frac{\tau^2\e}{4C^2}}(x)}\frac{W(u_\e)}{\e}&\geq\e^{-n-1}\omega_{n+1}\left(\frac{\tau^2\e}{4C^2}\right)^{n+1}\frac{\tau^2}{4}\\\nonumber
&\geq C_n\tau^{2n+4}.
	\end{align}
Denote
	\begin{align*}
	2\bar\theta_0:=\min\{1,C_n\tau^{2n+4}\},
	\end{align*}
then for $x\in B_{\frac{r}{2}}\cap\{|u_\e|\leq1-\tau\}$ the first inequality in the conclusion of the claim holds.
Applying the error estimates Proposition \ref{prop_errorterms} with the choice $\Omega'=B_\frac{r}{4}$ and $\Omega=B_\frac{r}{2}$, for sufficiently small $\tau$
	\begin{align*}
	\mu_\e(B_\frac{r}{4})&=\mu_\e\left(B_\frac{r}{4}\cap\{|u_\e|<1-\tau\}\right)+\mu_\e\left(B_\frac{r}{4}\cap\{|u_\e|\geq1-\tau\}\right)\\
	&\leq C\mu_\e\left(B_\frac{r}{4}\cap\{|u_\e|<1-\tau\}\right)+C\e\int_{B_{\frac{r}{2}}}|f_\e|^2+C\e(\tau r^n+\tau^2 r^{n-1})+Cr^{-2}\e.
	\end{align*}
Notice by \eqref{WillmoreTermBound}, the second term $\e \int_{B_{r/2}} |f_\e|^2 \leq \e^2 C(\Lambda_0, E_0)$. So the last three terms are at most of order $O(\e)$.
Hence, as $0\in\spt\mu$, by passing to limit $\e\rightarrow0$ we have
	\begin{align*}
	0<\mu(B_{\frac{r}{4}})\leq\liminf_{\e\rightarrow0}\mu_\e(B_\frac{r}{4})\leq\liminf_{\e\rightarrow0}\mu_\e\left(B_\frac{r}{4}\cap\{|u_\e|<1-\tau\}\right).
	\end{align*}
And in the set $\{|u_\e|\leq1-\tau\}$, we get by Lemma \ref{L1PositiveDiscrepancyVanishing} that
	\begin{align}\label{MeasureLowerBound}
	&\liminf_{\e\rightarrow0}\e^{-1}\mathcal L^{n+1}(B_\frac{r}{2}\cap\{|u_\e|\leq1-\tau\})\\\nonumber
&\geq\liminf_{\e\rightarrow0}\e^{-1}\int_{B_\frac{r}{2}\cap\{|u_\e|\leq1-\tau\}}\frac{W(u_\e)}{\tau^2}\\\nonumber
&=\liminf_{\e\rightarrow0}\frac{1}{\tau^2}\left(\mu_\e-\xi_\e\right)(B_\frac{r}{2}\cap\{|u_\e|\leq1-\tau\})\\\nonumber
&\geq\frac{1}{\tau^2}\liminf_{\e\rightarrow0}\mu_\e\left(B_\frac{r}{4}\cap\{|u_\e|<1-\tau\}\right){\color{red}-}\liminf_{\e\rightarrow0}\frac{1}{\tau^2}\xi_{\e,+}(B_\frac{r}{2}\cap\{|u_\e|\leq1-\tau\})\\\nonumber
&\geq \frac{\mu(B_\frac{r}{4})}{\tau^2}>0.
	\end{align}
(This guarantees we can always choose such a point $x\in B_\frac{r}{2}$ with $|u_\e(x)|\leq1-\tau$ if $0\in\spt\mu$.)
To complete the proof, we define for $0<\rho<r_1$ the convolution
	\begin{align*}
	\omega_{\e,\rho}(x):=\rho^{-n-1}\left(\chi_{B_\rho}*\frac{1}{\e}|f_\e|^2\right)(x)=\rho^{-n-1}\int_{B_\rho(x)}\frac{1}{\e}|f_\e|^2,
	\end{align*}
with
	\begin{align*}
	\|\omega_{\e,\rho}(x)\|_{L^1(B_\frac{r_1}{2})}\leq\int_{B_{\frac{r_1}{2}+r_1}}\frac{1}{\e}|f_\e|^2\leq C(\Lambda_0,E_0)<\infty,
	\end{align*}
by \eqref{WillmoreTermBound}. Denote by $\omega_\e(x):=\int_0^{r_1}\omega_{\e,\rho}(x)d\rho$, we have
	\begin{align*}
	\|\omega_\e (x)\|_{L^1(B_\frac{r_0}{2})}\leq r_1 C(\Lambda_0,E_0)<\infty.
	\end{align*}
Now we can estimate the term on the right hand side in the claim, by a change of variables $t=3\rho^{1-\beta}$. Here $\beta:=\beta(r_1)$ is chosen small enough such that $3\left(\frac{r_1}{4}\right)^{1-\beta}\leq r_1$.
							We calculate, setting $ t = 3 \rho^{1-\beta}$
	\begin{align*}
	\int_\e^{\frac{r}{4}}\rho^{-M\gamma-n-1}\left(\int_{B_{3\rho^{1-\beta}}(x)}\frac{1}{\e}|f_\e|^2\right)d\rho&=\int_{3\e^{1-\beta}}^{3\left(\frac{r}{4}\right)^{1-\beta}}\left(\frac{t}{3}\right)^\frac{-M\gamma-n-1}{1-\beta}\left(\int_{B_t(x)}\frac{1}{\e}|f_\e|^2\right)d\left(\frac{t}{3}\right)^\frac{1}{1-\beta}\\
	&\leq C_\beta\int_{3\e^{1-\beta}}^{3\left(\frac{r}{4}\right)^{1-\beta}} t^\frac{-M\gamma-n-1+\beta}{1-\beta}\left(\int_{B_t}\frac{1}{\e}|f_\e|^2\right)dt\\
	&\leq C_\beta\int_{3\e^{1-\beta}}^{3\left(\frac{r_1}{4}\right)^{1-\beta}} t^{\frac{-M\gamma-n-1+\beta}{1-\beta}+(n+1)}\omega_{\e,t}(x)dt.
	\end{align*}
We find
	\begin{align*}
	\frac{-M\gamma-n-1+\beta}{1-\beta}+(n+1)= \frac{- M \gamma - n \beta}{1-\beta}< 0
	\end{align*}
so that $ t^{\frac{- M \gamma - n \beta}{1-\beta} }$ is a decreasing function.
Hence we get the bound
	\begin{align}\label{eqn_L2IntegralUpperBound}
	\begin{split}
\int_\e^{\frac{r}{4}}\rho^{-M\gamma-n-1}\left(\int_{B_{3\rho^{1-\beta}}(x)}\frac{1}{\e}|f_\e|^2\right)d\rho&\leq C_\beta\int_{3\e^{1-\beta}}^{3\left(\frac{r_1}{4}\right)^{1-\beta}} \left(3\e^{1-\beta}\right)^\frac{-M\gamma-n\beta}{1-\beta}\omega_{\e,t}(x)dt\\
	&\leq C_\beta \e^{-M\gamma-n\beta}\int_0^{r_1}\omega_{\e,t}(x)dt\\
	&\leq C_\beta \e^{-M\gamma-n\beta}\omega_\e(x).
\end{split}
	\end{align}
Choosing $M\gamma<\frac{1}{2}$ and $\beta$ sufficiently small so that $M\gamma+n\beta<\frac{1}{2}$, and applying the weak $L^1$ inequality for the distribution function and \eqref{eqn_L2IntegralUpperBound}, we get for some $\tilde C_\beta$ depending on $\beta$
	\begin{align}\label{MeasureUpperBound}
	&\mathcal L^{n+1}\left(B_{\frac{r}{2}}\cap\left\{\int_\e^{\frac{r}{4}}\tilde C_\beta\e\rho^{-M\gamma-n-1}\left(\int_{B_{3\rho^{1-\beta}}(x)}|f_\e|^2\right)d\rho\geq\bar\theta_0\right\}\right)\\\nonumber
&\leq\mathcal L^{n+1}\left(B_{\frac{r}{2}}\cap\left\{C_\beta\e^2\e^{-M\gamma-n\beta}\omega_\e(x)\geq\bar\theta_0\right\}\right)\\\nonumber
&\leq C_\beta\e^{2-(M\gamma+n\beta)}\bar\theta_0^{-1}\|\omega_\e\|_{L^1(B_\frac{r}{2})}\\\nonumber
&\leq C_\beta\e^{2-(M\gamma+n\beta)}\bar\theta_0^{-1}\|\omega_{\e,\rho}(x)\|_{L^1(B_\frac{r_1}{2})}\\\nonumber
&\leq C_\beta\e^{2-(M\gamma+n\beta)}\bar\theta_0^{-1} C(\Lambda_0,E_0)\\
	&\rightarrow0,\nonumber
	\end{align}
as $\e\rightarrow0$.
This guarantees we can always choose such a point $x'\in B_\frac{r}{2}$ with
	\begin{align*}
	\left\{\int_\e^{\frac{r}{4}}\tilde C_\beta\e\rho^{-M\gamma-n-1}\left(\int_{B_{3\rho^{1-\beta}}(x')}|f_\e|^2\right)d\rho\leq\bar\theta_0\right\}.
	\end{align*}
We can thus combine \eqref{MeasureLowerBound} with \eqref{MeasureUpperBound} to find an $x\in B_\frac{r}{2}$ so that the upper bound and lower bound in the claim holds.
\end{proof}
With this claim, we proceed with the proof of the density lower bound. For the $\bar\theta_0$ obtained from the claim, we denote by $s:=\sup\{0\leq\rho\leq\frac{r}{4}:\frac{\mu_\e(B_\rho(x))}{\rho^n}\geq2\bar\theta_0\}$. And it is obvious from \eqref{LowerBoundAtScaleEpsilon}
	\begin{align*}
	s\geq\e.
	\end{align*}
By this choice of $s$, we have
	\begin{align*}
	\frac{\mu_\e(B_s(x))}{s^n}&\geq2\bar\theta_0,\\
	\frac{\mu_\e(B_\rho(x))}{\rho^n}&\leq2\bar\theta_0, \forall \rho\in\left[s,\frac{r}{4}\right].
	\end{align*}
Substituting $\frac{r}{4}$ for $r$ in the integral form of the almost monotonicity formula \eqref{IntegralMonotonicitywithL2MeanCurvature2}, we get from \eqref{eqn_LowerUpperBound} the following density lower bound
	\begin{align*}
	2^n\left[\frac{\mu_\e(B_\frac{r}{2}(x))}{\left(\frac{r}{2}\right)^n}\right]&\geq\frac{\mu_\e(B_\frac{r}{4}(x))}{\left(\frac{r}{4}\right)^n}\\
	&\geq \frac{\mu_\e(B_s(x))}{s^n} - C(\Lambda_0,\Omega') \left( \left(\frac{r}{4}\right)^{p_3\gamma} - s^{p_3\gamma}\right)- \frac{\phi(\e)}{p_3\gamma-n+1}\left( \left(\frac{r}{4}\right)^{p_3\gamma-n} - s^{p_3\gamma -n} \right)\\
	&-\int_s^{r/4}\tilde C_{\beta}\e\rho^{-M\gamma-n-1}\left(\int_{B_{3\rho^{1-\beta}}(x)}|f_\e|^2\right)d\rho-\tilde{C}_\beta(\left(\frac{r}{4}\right)^\gamma-s^\gamma)\\
	&-C(\Lambda_0,\Omega')\left( \left(\frac{r}{4}\right)^{1 - \frac{n}{q_0}}-s^{1-\frac{n}{q_0}}\right) - C(\Lambda_0,\Omega')\phi(\e) \left( \left(\frac{r}{4}\right)^{1 -n - \frac{n}{q_0}}-s^{1-n-\frac{n}{q_0}}\right)\\
	&\geq 2 \overline{\theta}_0 - C(\Lambda_0,\Omega')r^{\gamma_n} - C(\Lambda_0, \Omega') \phi(\e) r^{-n-\frac{n}{q_0} }-C(\Lambda_0,\Omega') \phi(\e) r^{p_3\gamma-n+1}\\
	&-\overline{\theta}_0 \\
	&\geq \overline{\theta}_0 - C(\Lambda_0,\Omega')r^{\gamma_n} - C(\Lambda_0, \Omega') \phi(\e) r^{-n-\frac{n}{q_0} }-C(\Lambda_0,\Omega') \phi(\e) r^{p_3\gamma-n+1},
	\end{align*}
where $\gamma_n:=\min\{p_3\gamma,\gamma,1-\frac{n}{q_0}\}>0$, and $\phi(\e)\rightarrow$ as $\e\rightarrow0$ by Theorem \ref{DensityUpperBound}.
As $ B_{\frac{r}{2}}(x) \subseteq B_r(0)$ we let $\e\rightarrow 0$ and get for some $\gamma_n>0$
	\begin{align*}
	\frac{\mu(\overline{B_r})}{r^n}\geq\limsup_{\e \rightarrow 0}\frac{\mu_\e(B_r)}{r^n}\geq\limsup_{\e \rightarrow 0}\frac{\mu_\e(B_{\frac{r}{2}}(x))}{r^n} \geq C_n\bar\theta_0- C_nr^{\gamma_n}.
	\end{align*}
Approximating $r'\nearrow r$ we get for $ 0 < r< r_1(\Omega')$
	\begin{align*}
	\frac{\mu(B_r(0))}{r^n } \geq c_0 \overline{\theta}_0
	\end{align*}
and hence
	\begin{align*}
	\theta_*^n(\mu) \geq \frac{\overline{\theta}}{\omega_n} \quad \text{$\mu$-a.e. in $\Omega$}.
	\end{align*}
which completes the proof.
\end{proof}
Before proving the rectifiability of the limit measure, we need to show that the full discrepancy vanishes as the limit $\e\rightarrow0$.
\begin{proposition}\label{VanishingDiscrepancy}
	\begin{align*}
	|\xi_\e|\rightarrow0 \quad \&\quad |\xi|=0.
	\end{align*}
\end{proposition}
\begin{proof}
We first prove the lower $n$-dimensional density of the discrepancy measure vanishes. Namely
	\begin{align*}
	\theta_*^n(|\xi|)=\liminf_{\rho\rightarrow0}\frac{|\xi|(B_\rho)}{\rho^n}=0.
	\end{align*}
If not, there exists $0<\rho_0,\delta<1$ and $B_{\rho_0}\subset\Omega$ such that
	\begin{align*}
	\frac{|\xi|(B_\rho(x))}{\rho^n}\geq\delta,\quad \forall 0<\rho\leq\rho_0.
	\end{align*}
Multiplying both sides of \eqref{MonotonicityWithLpMeanCurvature} by an integrating factor and integrating from $r$ to $\rho_0$ as in the proof of Theorem \ref{DensityUpperBound} we get
	\begin{align*}
	C(\Lambda_0,\Omega')\left(\frac{\mu_\e(B_{\rho_0})}{\rho_0^n}\right)-C(\Lambda_0,\Omega')\left(\frac{\mu_\e(B_r)}{r^n}\right)&\geq -C(\Lambda_0,\Omega')\int_r^{\rho_0}\frac{\xi_\e (B_{r_0})}{\rho^{n+1}}d\rho.
	\end{align*}
Using Lemma \ref{L1PositiveDiscrepancyVanishing}, that is $\xi_+=0$ and Theorem \ref{DensityUpperBound}, we have when passing to the limit $\e\rightarrow0$
	\begin{align*}
	\tilde C(\Lambda_0,\Omega')\geq C(\Lambda_0,\Omega')\int_r^{\rho_0}\frac{\xi_-(B_\rho)}{\rho^{n+1}}d\rho&= C(\Lambda_0,\Omega')\int_r^{\rho_0}\frac{\xi_-(B_\rho)+\xi_+(B_\rho)}{\rho^{n+1}}d\rho\\
	&= C(\Lambda_0,\Omega')\int_r^{\rho_0}\frac{|\xi|(B_\rho)}{\rho^{n+1}}d\rho\\
	&\geq C(\Lambda_0,\Omega')\int_r^{\rho_0}\frac{\delta}{\rho}d\rho\\
	&= C(\Lambda_0,\Omega')\delta\ln\left(\frac{\rho_0}{r}\right).
	\end{align*}
This gives a contradiction by letting $r\rightarrow0$. By the density lower bound Theorem \ref{DensityLowerBound} and differentiation theorem for measures, we have
	\begin{align*}
	D_\mu|\xi|(x)=\liminf_{\rho\rightarrow0}\frac{|\xi|(B_\rho(x))}{\mu(B_\rho(x))}&\leq\frac{\liminf_{\rho\rightarrow0}\frac{|\xi|(B_\rho(x))}{\rho^n}}{\limsup_{\rho\rightarrow0}\frac{\mu(B_\rho(x))}{\rho^n}}\\
	&\leq\frac{\theta_*^n(|\xi|,x)\omega_n}{\bar\theta}=0
	\end{align*}
and this shows
	\begin{align*}
	|\xi|=D_\mu|\xi|\cdot\mu=0.
	\end{align*}
\end{proof}
\begin{proposition} \label{prop_FirstVariation}
We choose a Borel measurable function $ \nu_\e : \Omega\rightarrow \partial B_1(0)$ extending $ \frac{\nabla u_\e} {| \nabla u_\e|}$ on $ \nabla u_\e\neq 0$ and consider the varifold $ V_\e =\mu_\e \otimes \nu_\e$ that is
	\begin{align}\label{eqn_Varifold}
	\int_{\{| \nabla u |\neq 0\}} \phi \left(x,I - \frac{\nabla u(x)}{|\nabla u(x)|}\otimes\frac{\nabla u(x)}{|\nabla u(x)|}\right) d\mu_i(x), \quad \phi \in C_c(G_n(\Omega)).
	\end{align}
The first variation is given by
	\begin{align}\label{eqn_FirstVariation}
	\delta V_\e( \eta) =-\int f_\e\langle\nabla u_\e,\eta\rangle dx+\int \nabla\eta\left(\frac{\nabla u_\e}{|\nabla u_\e|},\frac{\nabla u_\e}{|\nabla u_\e|}\right)d\xi_\e, \quad \forall \eta \in C_c^1(\Omega\times \mathbb R^{n+1}).
	\end{align}
\end{proposition}
\begin{proof}
By equation \eqref{eqn_FirstVariationGeneral}, we have
	\begin{align*}
	\delta V_\e( \eta)&= \int_{\Omega\times G(n+1,n)} \Div_S \eta(x) dV_\e(x,S) \\
	&= \int_{\Omega}( \Div\eta - \nabla\eta(\nu_\e,\nu_\e))d \mu_\e\\
	&= \int_{\Omega}( \Div\eta - \nabla\eta(\nu_\e,\nu_\e)) \left( \frac{\e|\nabla u_\e|^2}{2} + \frac{W(u_\e)}{\e} \right) d \mathcal L^{n+1}.
	\end{align*}
The Stress-Energy tensor for the Allen--Cahn equation is given by
	\begin{align*}
	T_{i j}&=\e\frac{\left|\nabla u_\e\right|^2}{2} \delta_{i j}-\e \nabla_i u_\e \nabla_j u_\e +W\left(u_\e \right) \delta_{i j},\\
	\nabla_i T_{i j}&=\e \nabla_i \nabla_k u_\e \nabla_k u_\e \delta_{ij}-\e \Delta u_\e \nabla_j u_\e-\e \nabla_i u_\e \nabla_i \nabla_k u_\e\\
	&+W^{\prime}\left(u_\e\right) \nabla_i u_\e \delta_{ij}\\
	&=\left(-\e \Delta u_\e +W^{\prime}\left(u_\e\right)\right) \nabla_j u_\e.
	\end{align*}
Now
	\begin{align*}
	T_{i j} \nabla_i \eta_j&=\left(\frac{\e | \nabla u_\e|^2}{2}+W\left(u_\e\right)\right) \Div \eta-\e \nabla \eta\left(\nabla u_\e, \nabla u_\e\right) \\
	&=\left(\frac{\e | \nabla u_\e|^2}{2}+W\left(u_\e\right)\right) \Div \eta-\nabla \eta(\nu_\e,\nu_\e)\e |\nabla u_\e|^2.
	 \end{align*}
Integrating by parts, we get
	\begin{align*}
	\int_{\Omega} \left(\frac{\e | \nabla u_\e|^2}{2}+W\left(u_\e\right)\right) \Div \eta-\nabla \eta(\nu_\e,\nu_\e)\e |\nabla u_\e|^2&= - \int_{\Omega} \nabla_i T_{i j} \eta_j\\
	&= \int_\Omega \left(-\e \Delta u_\e +W^{\prime}\left(u_\e\right)\right) \langle \nabla u_\e , \eta \rangle.
	\end{align*}
Hence inserting this into our expression for the first variation we get
	\begin{align*}
	\delta V_\e( \eta)&= - \int_\Omega\left( - \e \Delta u_\e + \frac{W'(u_\e)}{\e} \right) \langle \nabla u_\e, \eta \rangle d \mathcal L^{n+1} + \int_{\Omega} \nabla \eta(\nu_\e,\nu_\e) d \xi_\e.
	\end{align*}
\end{proof}
Combining Theorem \ref{DensityUpperBound}, Theorem \ref{DensityLowerBound} and Proposition \ref{VanishingDiscrepancy}, we obtain
\begin{theorem}\label{Rectifiability}
After passing to a subsequence, the associated varifolds $V_\e\rightarrow V$ where $V$ is a rectifiable $n$-varifold with the weak mean curvature in $L_{loc}^{q_0}(\mu_V)$.
\end{theorem}
\begin{proof}
We first compute the first variation of the associated varifolds $V_\e$ to the energy measure $\mu_\e$(c.f. \cite[Proposition 4.10]{roger2006modified}, \cite[Equation 4.3]{Tonegawa2020}). For any $\eta\in C_0^1(\Omega;\mathbb R^{n+1})$, using Proposition \ref{prop_FirstVariation} and Proposition \ref{VanishingDiscrepancy}
	\begin{align}\label{LocallyBoundedMeanCurvature}
	\begin{split}
|(\delta V)(\eta)|&=|\lim_{\e\rightarrow0}(\delta V_\e)(\eta)|\\
	&=\left|\lim_{\e\rightarrow0}\left(-\int f_\e\langle\nabla u_\e,\eta\rangle dx+\int \nabla\eta\left(\frac{\nabla u_\e}{|\nabla u_\e|},\frac{\nabla u_\e}{|\nabla u_\e|}\right)d\xi_\e\right)\right|\\
	&\leq\lim_{\e\rightarrow0}\int |f_\e| |\nabla u_\e||\eta| dx+\lim_{\e\rightarrow0}\int|\nabla\eta|d|\xi_\e|\\
	&\leq\lim_{\e\rightarrow0}\int \left|\frac{f_\e}{\e|\nabla u|}\right| |\eta| \e|\nabla u_\e|^2dx\\
	&\leq\lim_{\e\rightarrow0}\left(\int\left|\frac{f_\e}{\e|\nabla u_\e|}\right|^{q_0}\e|\nabla u_\e|^2\right)^\frac{1}{q_0}\left(\int|\eta|^\frac{q_0}{q_0-1}\e|\nabla u_\e|^2\right)^\frac{q_0-1}{q_0}\\
	&\leq\Lambda_0^\frac{1}{q_0}\|\eta\|_{L^\frac{q_0}{q_0-1}(\mu_V)}\left(\leq C(\Lambda_0,E_0)|\eta|\right).
\end{split}
	\end{align}
So we see the limit varifold has locally bounded first variation, combining with the density lower bound Theorem \ref{DensityLowerBound} we conclude the limit varifold is rectifiable by Allard's rectifiability theorem.
Moreover, the above calculation shows $\delta V$ is a bounded linear functional on $L_{loc}^\frac{q_0}{q_0-1}(\mu_V)$ and thus itself is in $L_{loc}^{q_0}(\mu_V)$.
\end{proof}
\section{Integrality}\label{sec:integrality}
In this section, we prove the integrality of the limit varifold.
\begin{theorem}\label{thm_integrality}
Let $\mu$ be defined by \eqref{eqn_Varifold}. Then $\frac{1}{\alpha}\mu$ is an integral $n$-varifold where $\alpha=\int_{-\infty}^\infty(\tanh' x)^2dx$ is the total energy of the heteroclinic $1$-d solution.
\end{theorem}
From the previous section, we have already shown the limiting varifold $V$ is rectifiable. And thus for a.e. $x_0\in\spt\mu_V$, we have for any sequence $\rho_i\rightarrow0$
	\begin{align*}
	\mathcal D_{\rho_i,\#}\circ\mathcal T_{x_0,\#}(\mu_V)\rightarrow \theta_{x_0} P_0,\quad\text{ for some $P_0\in G(n+1,n)$},
	\end{align*}
where $\mathcal D_{\rho_i}(x)=\rho_i^{-1}x$ and $\mathcal T_{x_0}(x)=x-x_0$ represent dilations and translations in $\mathbb R^{n+1}$ and $\theta_{x_0}$ is the density of $\mu_V$ at $x_0$.
By choosing a sequence of rescaling factors $\rho_i$ such that
	\begin{align}\label{TildeEpsilonSequence}
	\tilde\e_i:=\frac{\e_i}{\rho_i}\rightarrow0,
	\end{align}
the new sequence $\tilde u_{\tilde \e_i}(x):=u_{\e_i}(\rho_i x+x_0), \tilde f_{\tilde \e_i}(x):=\rho_i\tilde f_i(\rho_ix+x_0)$ satisfies
	\begin{align*}
	\tilde \e_i\Delta \tilde u_{\tilde \e_i}-\frac{W'(\tilde u_{\tilde \e_i})}{\tilde \e_i}=\tilde f_{\tilde \e_i}
	\end{align*}
and the associated varifold $\tilde V_i$ of this new sequence $\tilde u_{\tilde \e_i}$ converges to $\theta_{x_0} P_0$. By \eqref{WillmoreTermBound}, we also have
	\begin{align*}
	\frac{1}{\tilde \e_i}\int_{B_\rho} f_{\tilde \e_i}^2&\leq C\left(\int_{B_\rho}\left(\frac{f_{\tilde\e_i}}{\tilde\e_i|\nabla u_{\tilde\e_i}|}\right)^{q_0}\tilde\e_i|\nabla u_{\tilde \e_i}|^2\right)^\frac{2}{q_0}\\
	&= C\left(\rho_i^{q_0+1-(n+1)}\int_{B_{\rho_i\rho}}\left(\frac{f_{\e_i}}{\e_i|\nabla u_{\e_i}|}\right)^{q_0}\e_i|\nabla u_{\e_i}|^2\right)^\frac{2}{q_0}\\
	&\leq C \rho_i^\frac{2(q_0-n)}{q_0}\rightarrow0,
	\end{align*}
as $q_0>n$.
Furthermore, by choosing more carefully so that $\rho_i:=\tilde\e_i^\frac{(n-1)q_0}{2(q_0-n)}=\e_i^\frac{1}{1+\frac{2(q_0-n)}{(n-1)q_0}}$, we have
	\begin{align*}
	\frac{1}{\tilde \e_i}\int_{B_\rho} f_{\tilde \e_i}^2\leq\tilde\e_i^{n-1},\quad\text{ for $\rho>\tilde\e_i$} 	\end{align*}
and thus
	\begin{align}\label{ConditionWillmoreTermDecay}
	\frac{1}{\tilde \e_i}\int_{B_\rho} f_{\tilde \e_i}^2\leq\rho^{n-1}.
	\end{align}
Therefore we have reduced Theorem \ref{thm_integrality} to the following proposition
\begin{proposition}\label{ReductionOfIntegrality}
If the limit varifold is $\theta_0\mathcal H^n\lfloor P_0$ for some $P_0\in G(n+1,n)$ and $\theta_0>0$, then $\alpha^{-1}\theta_0$ is a nonnegative integer, where $\alpha=\int_{-\infty}^\infty(\tanh' x)^2dx$ is the total energy of the heteroclinic $1$-d solution.
\end{proposition}
In order to prove Proposition \ref{ReductionOfIntegrality}, we need two lemmas. The first Lemma \ref{MultiSheetMonotonicity} is a multi-sheet monotonicity formula (c.f. \cite[Theorem 6.2]{Allard1972} for the version for integral varifolds, which is used to prove the integrality of the limits of sequences of integral varifolds). The second Lemma \ref{AlmostEnergyQuantization} says at small scales, the energy of each layers are almost integer multiple of the 1-d solution.
We first gather some apriori bounds on energy ratio for $\mu_\e$.
\begin{proposition}\label{prop_EnergyRatioBound}
Let $\delta=\rho^\gamma,\e\leq\rho\leq r$ for $0<\gamma<\frac{1}{M}\leq\frac{1}{2}$, we have $\delta^{-M}\e\leq\rho^{1-M\gamma}\leq1$. Furthermore we choose $r:=d(B_{2\rho}(x),\partial B_{3\rho^{1-\beta}}(x))\geq\rho^{1-\beta}$.
Then
	\begin{align}
	\begin{split}
Cr^{-n}\mu_\e(B_r(x))&\geq s^{-n}\mu_\e(B_s(x))-C\int_s^r\rho^{p_3\gamma-n-1}\mu_\e(B_{2\rho}(x))d\rho\\
	&-C_\beta\e\int_s^r\rho^{-M\gamma-n-1}\left(\int_{B_{3\rho^{1-\beta}}(x)}|f_\e|^2\right)d\rho\\
	&-\tilde C_\beta\left(1+\int_{\{|u_\e| \geq 1\} \cap B_{3\rho^{1-\beta}}(x)} W'(u_\e)^2\right)\int_s^r\rho^{\gamma-1}d\rho-C.
\end{split}
	\end{align}
\end{proposition}
\begin{proof}
Substitute \eqref{eqn_DiscrepancyUpperBound} into the equation \eqref{AlmostMonotonicityWithExponentialConstant} in the proof of Theorem \ref{DensityUpperBound}, we have for $\e\leq s\leq\rho\leq r\leq 1$
	\begin{align}\label{eqn_EnergyRatioBound}
	C(\Lambda_0,q_0)\left(\frac{\mu_\e(B_r)}{r^n}\right)&\geq \left(\frac{\mu_\e(B_s)}{s^n}\right)-C(\Lambda_0,q_0)-C\int_s^r\frac{\xi_+(B_\rho)}{\rho^{n+1}}\\\nonumber
&\geq \left(\frac{\mu_\e(B_s)}{s^n}\right)-C(\Lambda_0,q_0)-C\int_s^r\rho^{p_3\gamma-n-1}\mu_\e(B_{2\rho}(x))d\rho\\\nonumber
&-C_\beta\e\int_s^r\rho^{-M\gamma-n-1}\left(\int_{B_{3\rho^{1-\beta}}(x)}|f_\e|^2\right)d\rho\\
	&-\int_s^r\tilde C_\beta\e\rho^{\gamma-2}\left(1+\int_{\{|u_\e| \geq 1\} \cap B_{3\rho^{1-\beta}}(x)} W'(u_\e)^2\right)d\rho.
	\end{align}
Noticing $\e\leq\rho$ in the last term, we then conclude the desired energy ratio bound.
\end{proof}
As a corollary, we have
\begin{corollary}\label{CorollaryEpsilonEnergyRatioBound}
If in addition to the conditions in Proposition \ref{prop_EnergyRatioBound}, we assume
	\begin{align}\label{ConditionWillmoreDecay}
	\frac{1}{\e}\int_{B_\rho} f_\e^2\leq\rho^{n-1},\quad \text{ for $\rho\geq\e$},
	\end{align}
and
	\begin{align*}
	\beta\in\left(0,\frac{1-M\gamma}{2(n-1)}\right),
	\end{align*}
then the following upper bound for the energy ratio for $\mu_\e$ holds
	\begin{align}\label{UpperRatioBoundEpsilonEnergy}
	\frac{\mu_\e(B_s(x))}{s^n}\leq C\frac{\mu_\e(B_r(x))}{r^n}+C(\Lambda_0,E_0,q_0,n),
	\end{align}
for $\e\leq s\leq r$.
\end{corollary}
\begin{proof}
We have
	\begin{align*}
	p_3\gamma-1,-M\gamma+\beta(n-1),\gamma-1>-1.
	\end{align*}
Thus by Proposition \ref{prop_EnergyRatioBound} and $\e\leq\rho$, we have
	\begin{align*}
	C\left(\frac{\mu_\e(B_r(x))}{r^n}\right)&\geq \left(\frac{\mu_\e(B_s(x))}{s^n}\right)-C\int_s^r\rho^{p_3\gamma-1}\left(\frac{\mu_\e(B_{2\rho}(x))}{\rho^n}\right)d\rho\\
	&-C_\beta\e^2\int_s^r\rho^{-M\gamma-2-\beta(n-1)}\left(\frac{\int_{B_{3\rho^{1-\beta}}(x)}|f_\e|^2}{\e\rho^{(1-\beta)(n-1)}}\right)d\rho\\
	&-\tilde C_\beta\left(1+\int_{\{|u_\e| \geq 1\} \cap \Omega} W'(u_\e)^2\right)\int_s^r\rho^{\gamma-1}d\rho-C\\
	&\geq \left(\frac{\mu_\e(B_s(x))}{s^n}\right)-C\int_s^r\rho^{p_3\gamma-1}\left(\frac{\mu_\e(B_{2\rho}(x))}{\rho^n}\right)d\rho-C.
	\end{align*}
The conclusion then follows by substituting in \eqref{ConditionWillmoreDecay} and applying Gronwall's inequality to the above differential inequality.
\end{proof}
\begin{lemma}\label{MultiSheetMonotonicity}
For any $N\in\mathbb N$, $\delta>0$ small, $\Lambda>0$ large and $\beta\in(0,\frac{1-M\gamma}{2(n-1)})$ where $M,\gamma$ are from Proposition \ref{prop_DiscrepancyEnergyRatioLowerBound}, there exists $\omega>0$ such that the following holds:
Suppose $u_\e$ satisfies \eqref{PFVe} and the conditions(1)-(3) in Theorem \ref{main} are satisfied, then for any finite set $X\subset\{0^n\}\times\mathbb R\subset\mathbb R^{n+1}$, and the number of elements in $X$ is no more than $N$. If moreover for some $0<\e\leq d\leq R\leq\omega$, the followings are satisfied
	\begin{align}
	&\mathrm{diam}(X)<\omega R,\\
	&|x-y|>3d, \quad \text{ for $x,y\in X$ and $x\neq y$},
	\end{align}
	\begin{align}
	&|\xi_\e|(B_\rho(x))+\int_{B_{\rho}(x)}\e|\nabla u_\e|^2\sqrt{1-\nu^2_{\e,n+1}}\leq\omega\rho^n,\quad \text{ for $x\in X$ and $d\leq\rho\leq R$},
	\end{align}
	\begin{align}
	&\frac{1}{\e}\int_{B_{\rho}(x)}|f_\e|^2\leq\Lambda\rho^{n-1},\quad\text{ for $3d^{1-\beta}\leq\rho\leq 3R^{1-\beta}$}.
	\end{align}
Then we have
	\begin{align}
	\sum_{x\in X}d^{-n}\mu_\e(B_d(x))\leq(1+\delta)R^{-n}\mu_\e(\cup_{x\in X}B_R(x))+\delta.
	\end{align}
\end{lemma}
The proof of the lemma is based on an inductive application of the sheets-separation proposition, along with appropriate choices of parameters $\gamma$ and $\omega$. To simplify notation in the remainder of this section, we introduce a shorthand for the sheets-separation term
	\begin{align}\label{SheetsSeparation}
	\mathcal S_{y,x}=:(y_{n+1}-x_{n+1})\left(\frac{\e|\nabla u_\e|^2}{2}+\frac{W(u_\e)}{\e}\right)-\e\frac{\partial u_\e}{\partial x_{n+1}}\langle y-x,\nabla u_\e\rangle,
	\end{align}
for any pair of points $x,y\in\mathbb R^{n+1}$.
\begin{proposition}\label{prop_MultiSheetMonotonicity}
Suppose the conditions in Theorem \ref{main} are satisfied and let $X\subset\{0^n\}\times[t_1+d,t_2-d]\subset\mathbb R^{n+1}$ consist of no more than $N\in\mathbb N$ elements and $\cup_{x\in X}B_{3R^{1-\beta}}\subset\Omega\subset\mathbb R^{n+1}$. Furthermore suppose for $-\infty\leq t_1<t_2\leq\infty, 0<\e\leq d\leq R\leq\frac{1}{2},\beta\in(0,\frac{1-M\gamma}{2(n-1)})$ the following are satisfied:
	\begin{align}
	(\Gamma+1)\mathrm{diam}(X)&<R, \quad \text{ for some $\Gamma\geq1$},\\
	|x-y|&>3d, \quad \text{ for $x\neq y\in X$},
	\end{align}
	\begin{align}\label{HorizontalErrorBound}
	&\int_d^R\rho^{-n-1}\left|\int_{B_\rho(x)\cap\{y_{n+1}=t_j\}}\mathcal S_{y,x}d\mathcal H^n_y\right|d\rho\leq\omega
	\end{align}
for any $x\in X$, $j=1,2$ and for some $\omega>0$,
	\begin{align}\label{SmallDiscrepancyAndTilt}
	&|\xi_\e|(B_\rho(x))+\int_{B_{\rho}(x)}\e|\nabla u_\e|^2\sqrt{1-\nu_{\e,n+1}^2}\leq\omega\rho^n, \quad \text{ for $d\leq\rho\leq R$}\\\label{WillmoreDecayCondition}
&\frac{1}{\e}\int_{B_\rho(x)}|f_\e|^2\leq\Lambda\rho^{n-1}, \quad \text{ for $3d^{1-\beta}\leq\rho\leq3R^{1-\beta}$},\\
	&\frac{\mu_\e(B_{2R}(x))}{R^n}\leq\Lambda, \quad \forall x\in X \quad \text{(this is implied by Corollary \ref{CorollaryEpsilonEnergyRatioBound} as $R\geq\e$)}.
	\end{align}
Then by denoting $S_t^{t'}:=\{t\leq y_{n+1}\leq t'\}$, we have
	\begin{align}\label{DensityBoundInSlab}
	d^{-n}\mu_\e(B_d(x))\leq R^{-n}\mu_\e(B_R(x)\cap S_{t_1}^{t_2})+ CR^{\gamma_0}+2\omega,
	\end{align}
for some $\gamma_0>0$ and for all $x\in X$.
Furthermore, if $X$ consists of more than one point, then there exists $t_3\in(t_1,t_2)$ such that $\forall x\in X$
	\begin{align}
	&|x_{n+1}-t_3|>d,
	\end{align}
	\begin{align}
	&\int_d^{\tilde R}\rho^{-n-1}\int_{B_\rho(x)\cap\{y_{n+1}=t_3\}}\left|\mathcal S_{y,x}\right|d\mathcal H^n_yd\rho\leq3N\Gamma\omega,
	\end{align}
where $\tilde R:=\Gamma\mathrm{diam}(X)$ and $\mathcal S_{y,x}$ as defined in \eqref{SheetsSeparation}. Moreover, both $X\cap X_{t_1}^{t_3}$ and $X\cap X_{t_3}^{t_2}$ are non-empty and
	\begin{align*}
	\tilde R^{-n}&\left(\mu_\e(\cup_{x\in X\cap X_{t_1}^{t_3}}B_{\tilde R}(x)\cap S_{t_1}^{t_3})+\mu_\e(\cup_{x\in X\cap X_{t_3}^{t_2}}B_{\tilde R}(x)\cap S_{t_3}^{t_2})\right)\leq\\
	&\left(1+\frac{1}{\Gamma}\right)^nR^{-n}\mu_\e\left(\cup_{x\in X}B_R(x)\cap S_{t_1}^{t_2}\right)+CR^{\gamma_0}+2\omega.
	\end{align*}
\end{proposition}
\begin{proof}
First we choose $\phi$ to be a non-increasing function satisfying
	\begin{align*}
	\phi_{\delta,\rho}=
\begin{cases}
1,&\text{ on $[0,\rho]$}\\
	0,&\text{ on $[\rho+\delta,\infty)$},
\end{cases}
	\end{align*}
and $\chi_{\delta}$ satisfying
	\begin{align*}
	&\chi_\delta\equiv
\begin{cases}
 1,&\text{ on $[t_1+\delta, t_2-\delta]$},\\
	 0,&\text{ on $(-\infty,t_1]\cup[t_2,\infty)$},
 \end{cases}
	\end{align*}
with $ \chi_\delta'\geq0$ on $[t_1,t_1+\delta]$ and $\chi_\delta'\leq0 $ on $[t_2-\delta,t_2]$. Then we multiply \eqref{PFVe} on both sides by $\langle\nabla u,\eta\rangle$, where $\eta\in C_0^1(\Omega,\mathbb R^{n+1})$ is defined by $\eta(y):=(y-x)\phi_{\delta,\rho}(|y-x|)\chi_\delta(y_{n+1})$. Using integration by parts, we have
	\begin{align*}
	&\int f_\e\langle y-x,\nabla u_\e\rangle\phi_{\delta,\rho}(|y-x|)\chi_{\delta}(y_{n+1})\\&=\int f_\e\langle\nabla u,\eta\rangle\\
	&=\int\left(\frac{\e|\nabla u_\e|^2}{2}+\frac{W(u_\e)}{\e}\right)\mathrm{div}\eta-\e\nabla u\otimes\nabla u:\nabla\eta\\
	&=\int\left(|y-x|\phi'_{\delta,\rho}\chi_{\delta}+(n+1)\phi_{\delta,\rho}\chi_{\delta}+(y_{n+1}-x_{n+1})\phi_{\delta,\rho}\chi_{\delta}'\right)d\mu_\e\\
	&-\int\e\frac{\phi_{\delta,\rho}'\chi_{\delta}}{|y-x|}\langle y-x,\nabla u_\e\rangle^2-\int\e|\nabla u_\e|^2\phi_{\delta,\rho}\chi_{\delta}-\int\e\frac{\partial u}{\partial x_{n+1}}\langle y-x,\nabla u_\e\rangle\phi_{\delta,\rho}\chi_{\delta}'.
	\end{align*}
Letting $\delta\rightarrow0$, we have
	\begin{align*}
	&\int_{B_\rho(x)\cap S_{t_1}^{t_2}}f_\e\langle y-x,\nabla u_\e\rangle\\
	&=-\int_{\partial B_\rho\cap S_{t_1}^{t_2}}\rho d\mu_\e+(n+1)\int_{B_\rho\cap S_{t_1}^{t_2}}d\mu_\e\\
	&+\int_{B_\rho\cap\{y_{n+1}=t_2\}}(y_{n+1}-x_{n+1})d\mu_\e-\int_{B_\rho\cap\{y_{n+1}=t_1\}}(y_{n+1}-x_{n+1})d\mu_\e\\
	&+\int_{\partial B_\rho\cap S_{t_1}^{t_2}}\e\rho^{-1}\langle y-x,\nabla u_\e\rangle^2-\int_{B_\rho\cap S_{t_1}^{t_2}}\e|\nabla u_\e|^2\\
	&+\int_{B_\rho\cap\{y_{n+1}=t_2\}}\e\frac{\partial u}{\partial x_{n+1}}\langle y-x,\nabla u_\e\rangle-\int_{B_\rho\cap\{y_{n+1}=t_1\}}\e\frac{\partial u}{\partial x_{n+1}}\langle y-x,\nabla u_\e\rangle.
	\end{align*}
Dividing both sides by $\rho^{n+1}$ and rearranging gives the following weighted monotonicity formula
	\begin{align}\label{WeightedMonotonicitySlab}
	&\frac{d}{d\rho}\left(\rho^{-n}\mu_\e(B_\rho(x)\cap S_{t_1}^{t_2})\right)\\\nonumber
&=-n\rho^{-n-1}\mu_\e(B_\rho(x)\cap S_{t_1}^{t_2})+\rho^{-n}\mu_\e(\partial B_\rho\cap S_{t_1}^{t_2})\\\nonumber
&=-(n+1)\rho^{-n-1}\int_{B_\rho(x)\cap S_{t_1}^{t_2}}d\mu_\e+\rho^{-n}\int_{\partial B_\rho(x)\cap S_{t_1}^{t_2}}d\mu_\e\\\nonumber
&+\rho^{-n-1}\int_{B_\rho(x)\cap S_{t_1}^{t_2}}\e|\nabla u_\e|^2-\rho^{-n-1}\int_{B_\rho(x)\cap S_{t_1}^{t_2}}d\xi_\e\\\nonumber
&=\rho^{-n-1}\int_{B_\rho\cap\{y_{n+1}=t_2\}}(y_{n+1}-x_{n+1})d\mu_\e-\rho^{-n-1}\int_{B_\rho\cap\{y_{n+1}=t_1\}}(y_{n+1}-x_{n+1})d\mu_\e\\\nonumber
&+\rho^{-n-1}\int_{B_\rho\cap\{y_{n+1}=t_2\}}\e\frac{\partial u}{\partial x_{n+1}}\langle y-x,\nabla u_\e\rangle-\rho^{-n-1}\int_{B_\rho\cap\{y_{n+1}=t_1\}}\e\frac{\partial u}{\partial x_{n+1}}\langle y-x,\nabla u_\e\rangle\\\nonumber
&-\rho^{-n-1}\int_{B_\rho(x)\cap S_{t_1}^{t_2}}d\xi_\e-\rho^{-n-1}\int_{B_\rho(x)\cap S_{t_1}^{t_2}}f_\e\langle y-x,\nabla u_\e\rangle+\rho^{-n-1}\int_{\partial B_\rho\cap S_{t_1}^{t_2}}\e\rho^{-1}\langle y-x,\nabla u_\e\rangle^2.
	\end{align}
By the condition given by \eqref{HorizontalErrorBound}, the sum of norms of the first fours terms are bounded by $2\omega$. And by \eqref{eqn_DiscrepancyUpperBound} and \eqref{WillmoreDecayCondition}, the discrepancy term is bounded by
	\begin{align*}
	&\rho^{-n-1}\int_{B_\rho(x)\cap S_{t_1}^{t_2}}d\xi_{\e,+}\\
	&\leq C\rho^{p_3\gamma-n-1}\mu_\e(B_{2\rho}(x))+\tilde C_k\e\rho^{-M\gamma-n-1}\int_{B_{3\rho^{1-\beta}}(x)}|f_\e|^2+\tilde C_\beta\e\rho^{\gamma-2}\left(1+\int_{\{|u_\e| \geq 1\} \cap \Omega} W'(u_\e)^2\right)\\
	&\leq C\rho^{p_3\gamma-1}\left(\frac{\mu_\e(B_{2\rho}(x)\cap S_{t_1}^{t_2})}{\rho^n}\right)+C\e\rho^{-M\gamma-n-1}\Lambda\e\rho^{(1-\beta)(n-1)}+C\e\rho^{\gamma-2}\\
	&\leq C\rho^{p_3\gamma-1}\left(\frac{\mu_\e(B_{2\rho}(x)\cap S_{t_1}^{t_2})}{\rho^n}\right)+C\rho^{2-M\gamma-n-1+(n-1)-\beta(n-1)}+C\rho^{\gamma-1}\\
	&\leq C\rho^{p_3\gamma-1}+C\rho^{-1+\frac{1-M\gamma}{2}}+C\rho^{\gamma-1},
	\end{align*}
where we used \eqref{UpperRatioBoundEpsilonEnergy} and $\e\leq\rho$ in the last line.
By \eqref{ConicalityTermBound} in the proof of Theorem \ref{DensityUpperBound} and \eqref{UpperRatioBoundEpsilonEnergy}, we have
	\begin{align*}
	\rho^{-n-1}\left|\int_{B_\rho(x)\cap S_{t_1}^{t_2}}f_\e\langle y-x,\nabla u_\e\rangle\right|\leq C\rho^{-\frac{n}{q_0}}\left(1+\frac{\mu_\e(B_{\rho}\cap S_{t_1}^{t_2})}{\rho^n}\right)\leq\tilde C\rho^{-\frac{n}{q_0}}.
	\end{align*}
											By integrating \eqref{WeightedMonotonicitySlab} from $d$ to $R$ and noting $B_d(x)\subset S_{t_1}^{t_2}$,
we obtain the following upper bound of energy density for $\mu_\e$,
	\begin{align*}
	d^{-n}\mu_\e(B_d(x))=d^{-n}\mu_\e(B_d(x)\cap S_{t_1}^{t_2})\leq R^{-n}\mu_\e(B_R(x)\cap S_{t_1}^{t_2})+ CR^{\gamma_0}+2\omega,
	\end{align*}
where $\gamma_0=\min\{\frac{q_0-n}{q_0},p_3\gamma,\frac{1-M\gamma}{2},\gamma\}>0$. This proves \eqref{DensityBoundInSlab}. \\
	Next, if $X$ contains more than one point, then we can choose $x_\pm\in X$ such that $x_{+,n+1}-x_{-,n+1}>\frac{\mathrm{diam} X}{N}$ (where $x_{\pm,n+1}$ denotes the $(n+1)$-th coordinate of $x_\pm$) and there is no other element of $X$ in $\{0\}\times(x_{-,n+1},x_{+,n+1})$.
Let $\tilde t_1:=x_{-,n+1}+\frac{x_{+,n+1}-x_{-,n+1}}{3}$ and $\tilde t_2:=x_{+,n+1}-\frac{x_{+,n+1}-x_{-,n+1}}{3}$. For $x\in X, y\in B_\rho(x), d\leq\rho\leq\tilde R$, we have
	\begin{align*}
	|\mathcal S_{y,x}|&=\left|(y_{n+1}-x_{n+1})\left(\frac{\e|\nabla u_\e|^2}{2}+\frac{W(u_\e)}{\e}\right)-\e\frac{\partial u_\e}{\partial x_{n+1}}\langle y-x,\nabla u_\e\rangle\right|\\
	&=\left|(y_{n+1}-x_{n+1})\left(\frac{W(u_\e)}{\e}-\frac{\e|\nabla u_\e|^2}{2}\right)+|(y_{n+1}-x_{n+1})\e|\nabla u_\e|^2-\e\frac{\partial u_\e}{\partial x_{n+1}}\langle y-x,\nabla u_\e\rangle\right|\\
	&\leq\rho\left|\frac{\e|\nabla u_\e|^2}{2}-\frac{W(u_\e)}{\e}\right|+\e|\nabla u_\e|^2|\langle y-x,e_{n+1}\rangle-\langle y-x,\nu_\e\rangle\langle e_{n+1},\nu_\e\rangle|\\
	&\leq\rho\left|\frac{\e|\nabla u_\e|^2}{2}-\frac{W(u_\e)}{\e}\right|+\e|\nabla u_\e|^2|y-x|\sqrt{1-\nu_{\e,n+1}^2}\\
	&\leq\rho\left|\frac{\e|\nabla u_\e|^2}{2}-\frac{W(u_\e)}{\e}\right|+\rho\e|\nabla u_\e|^2\sqrt{1-\nu_{\e,n+1}^2}.
	\end{align*}
And thus by condition \eqref{SmallDiscrepancyAndTilt}, we have
	\begin{align*}
	&\int_{\tilde t_1}^{\tilde t_2}\int_d^{\tilde R}\rho^{-n-1}\int_{B_\rho(x)\cap\{y_{n+1}=t\}}\left|\mathcal S_{y,x}\right|d\mathcal H^n_{\{y_{n+1}=t\}}d\rho dt\\
	&=\int_d^{\tilde R}\rho^{-n-1}\int_{B_\rho(x)\cap S_{\tilde t_1}^{\tilde t_2}}\left|(y_{n+1}-x_{n+1})\left(\frac{\e|\nabla u_\e|^2}{2}+\frac{W(u_\e)}{\e}\right)-\e\frac{\partial u_\e}{\partial x_{n+1}}\langle y-x,\nabla u_\e\rangle\right|dyd\rho\\
	&\leq\int_d^{\tilde R}\rho^{-n}\int_{B_\rho(x)\cap S_{\tilde t_1}^{\tilde t_2}}\left|\frac{\e|\nabla u_\e|^2}{2}-\frac{W(u_\e)}{\e}\right|+\e|\nabla u_\e|^2\sqrt{1-\nu_{\e,n+1}^2}dyd\rho\\
	&\leq\int_d^{\tilde R}\rho^{-n}\omega\rho^nd\rho\\
	&\leq\omega\tilde R.
	\end{align*}
So there must exist a $t_3\in[\tilde t_1,\tilde t_2]$ such that
	\begin{align*}
	&\int_d^{\tilde R}\rho^{-n-1}\int_{B_\rho(x)\cap\{y_{n+1}=t_3\}}\left|\mathcal S_{y,x}\right|d\mathcal H^n_{\{y_{n+1}=t\}}d\rho\\
	&\leq\frac{\omega\tilde R}{(\tilde t_2-\tilde t_1)}\leq\frac{3N\omega\tilde R}{\mathrm{diam}(X)}=3N\Gamma\omega.
	\end{align*}
By the choice of $t_3\in[\tilde t_1,\tilde t_2]$, we automatically have $|x_{n+1}-t_3|>d$ for all $x\in X$. Finally, by denoting
	\begin{align*}
	X_+:=\{x\in X,x_n\geq t_3\}, X_-:=\{x\in X,x_n< t_3\},
	\end{align*}
we have $X_\pm\neq\emptyset$ and
	\begin{align*}
	\left(\cup_{x\in X_-}B_{\tilde R}(x)\cap S_{t_1}^{t_3}\right)\cup \left(\cup_{x\in X_+}B_{\tilde R}(x)\cap S_{t_3}^{t_2}\right)\subset B_{\tilde R+\mathrm{diam}(X)}(x_0)\cap S_{t_1}^{t_2},
	\end{align*}
for any $x_0\in X$.
By \eqref{DensityBoundInSlab}(with $\tilde R+\mathrm{diam}(X)$ in place of $d$), we then have
	\begin{align*}
	&\tilde R^{-n}\left(\mu_\e\left(\cup_{x\in X_-}B_{\tilde R}(x)\cap S_{t_1}^{t_3}\right)+\mu_\e \left(\cup_{x\in X_+}B_{\tilde R}(x)\cap S_{t_3}^{t_2}\right)\right)\\
	&\leq \tilde R^{-n}\mu_\e(B_{\tilde R+\mathrm{diam}(X)}(x_0)\cap S_{t_1}^{t_2})\\
	&=\left(1+\frac{1}{\Gamma}\right)^n(\tilde R+\mathrm{diam}(X))^{-n}\mu_\e(B_{\tilde R+\mathrm{diam}(X)}(x_0)\cap S_{t_1}^{t_2})\\
	&\leq\left(1+\frac{1}{\Gamma}\right)^n\left(R^{-n}\mu_\e(B_R(x_0)\cap S_{t_1}^{t_2})+CR^{\gamma_0}+2\omega\right)\\
	&\leq\left(1+\frac{1}{\Gamma}\right)^n\left(R^{-n}\mu_\e(\cup_{x\in X}B_R(x)\cap S_{t_1}^{t_2})\right)+CR^{\gamma_0}+2\omega.
	\end{align*}
\end{proof}
The next Lemma taken from \cite{roger2006modified} shows the energy ratio at small scales are very close to the $1$-d solution.
\begin{lemma}[Lemma 5.5 of \cite{roger2006modified}]\label{AlmostEnergyQuantization}
Suppose the conditions in Theorem \ref{main} are satisfied. For any $\tau\in(0,1$),$\delta>0$ small, $\Lambda>0$ large, there exists $\omega>0$ sufficiently small and $L>1$ sufficiently large such that the following holds: Suppose $u_\e$ satisfies condition of Theorem \ref{main} in $B_{4L\e}(0)\subset\mathbb R^{n+1}$ and
	\begin{align}
	&|u_\e(0)|\leq 1-\tau\\
	\label{eqn_TiltDiscrepancyBoundTransition}&|\xi_\e(B_{4L\e}(0))|+\int_{B_{4L\e}(0)}\e|\nabla u_\e|^2\sqrt{1-\nu_{\e,n+1}^2}\leq\omega(4L\e)^n\\
	\label{eqn_WillmoreBoundTransition}&\frac{1}{\e}\int_{B_{4L\e}(0)}|f_\e|^2\leq\Lambda(4L\e)^{n-2}\\
	\label{eqn_EnergyBoundTransition}&\mu_\e({B_{4L\e}(0)})\leq\Lambda(4L\e)^n.
	\end{align}
Then by denoting $(0,t)\in\mathbb R^{n+1}$ to be the point with first $n$-th coordinate functions being $0$ and the $(n+1)$-th coordinate functions being $t$, we have
	\begin{align}\label{LayerSeparation}
	&|u(0,t)|\geq1-\frac{\tau}{2},\quad \text{ for all $L\e\leq|t|\leq 3L\e$}\\\label{WholeEnergyQuantization}
&\left|\frac{1}{\omega_n(L\e)^n}\mu_\e(B_{L\e}(0))-\alpha\right|\leq\delta\\\label{PotentialEnergyQuantization}
&\left|\int_{-L\e}^{L\e}W(u_\e(0,t))dt-\frac{\alpha}{2}\right|\leq\delta.
	\end{align}
\end{lemma}
\begin{proof}
First we consider the $1$-dimensional solution
	\begin{align*}
	q_0'(t)&= \sqrt{W(q_0(t))}\quad \forall t \in \mathbb R, \\
	q_0(0)&= u(0).
	\end{align*}
We will use $q_0$ to choose $L$ depending on $ \tau, \delta > 0$. On $ \mathbb R^{n+1}$ we write $q(x) = q_0(x_{n+1})$ and choose $ L >1$ large enough depending on $ \tau, \delta$ so that
	\begin{align}\label{eqn_1DEstimates}
	\begin{split}
|q(0,t)|&\geq 1 - \frac{\tau}{3}, \quad \text{ for all $ L \leq |t| \leq 3 L$},\\
	\left| \frac{1}{\omega_n L^{n-1}}\int_{B_L(0)} \left( \frac{|\nabla q|^2}{2} + W(q) \right) - \alpha\right|&\leq \frac{\delta}{2}\\
	\left| \int_{-L}^L W(q(0,t)) dt - \frac{\alpha}{2} \right|&\leq \frac{\delta}{2}
\end{split}
	\end{align}
whenever $|q(0)| \leq 1 - \tau$.
		The function $u$ satisfies the Allen--Cahn equation
	\begin{align*}
	-\Delta u + W'(u) = f ,
	\end{align*}
and by our condition $(2)$ in Theorem \ref{main} we get $\|u_\e\|_{L^\infty(B_{1/2}(x))}\leq c_0$.
Hence by Calderon--Zygmund estimates we get uniform $W^{2,\frac{n+1}{2} + \delta_0}$ estimates on $B_{3L}(0)$ of the form
	\begin{align}\label{eqn_W2}
	\| u\|_{W^{2,\frac{n+1}{2} + \delta_0}(B_{3L}(0))} \leq C(\Lambda,L).
	\end{align}
If there is no such $\omega>0$ such that \eqref{LayerSeparation}, \eqref{WholeEnergyQuantization} and \eqref{PotentialEnergyQuantization} holds then this implies there exists $\omega_j\rightarrow 0$ and $u_j, f_j$ satisfying the above estimates but that do not satisfy \eqref{LayerSeparation}, \eqref{WholeEnergyQuantization} and \eqref{PotentialEnergyQuantization}. By \eqref{eqn_W2}, we get after passing to a suitable subsequence that $ u_j \rightharpoonup u$ weakly in $W^{2,\frac{n+1}{2} + \delta_0}(B_{3L}(0)) $ and $ f_j \rightharpoonup f $ weakly in $L^{\frac{n+1}{2} + \delta_0}(B_{3L}(0))$.
By the Sobolev embedding we have $W^{2,\frac{n+1}{2} + \delta_0}(B_{3L}(0)) \hookrightarrow C^0 $ for $\delta_0>0$ and hence we get $u_j\rightarrow u$ uniformly in $C^0(B_{3L}(0))$.
\begin{claim}
The functions $ u_j\rightarrow u = q $ strongly in $W^{1,2}(B_{3L}(0))$.
\end{claim}
\begin{proof}
Writing $ \nabla =(\nabla', \partial_{n+1})$ we get \eqref{eqn_TiltDiscrepancyBoundTransition}
	\begin{align*}
	\int_{B_{3L(0)}}\left| \frac{|\nabla u|^2}{2} - W(u)\right|&\leq \liminf_{j\rightarrow \infty} \int_{B_{3L}(0)} \left| \frac{|\nabla u_j|^2}{2} - W(u_j)\right| \\
	&\leq \liminf_{j\rightarrow \infty}|\xi_j |(B_{3L}(0)) = 0
	\end{align*}
and
	\begin{align*}
	\int_{B_{3L}(0)} |\nabla ' u |&\leq \liminf_{j\rightarrow \infty} \int_{B_{3L}(0)} |\nabla ' u_j | \leq C(L) \left(\int_{B_{3L}(0)} |\nabla u_j|^2 \sqrt{1 - \nu_{j,n}^2} \right)^{1/2} =0,
	\end{align*}
where $ \nu_j = \frac{\nabla u_j}{|\nabla u|}$ for $ \nabla u_j \neq 0$. Therefore $ |\nabla u|^2 = 2 W(u)$ and $ u(y,t) = u_0(t)$ for some $ u_0 \in W^{2,\frac{n+1}{2} + \delta_0}((-L,L))\hookrightarrow C^{1,\alpha}((-L,L))$ and $ | u'_0| = 2 \sqrt{2W(u_0)}$. As $ |u_0(0)| \leq 1 - \tau$ by uniform convergence, we see $ |u_0| < 1$ and $ |u'_0| > 0$. After a reflection of the form $(y,x_n)\mapsto(y,-x_n)$ if necessary, we may assume $u_0'>0$ and hence $ u_0'= \sqrt{2W(u_0)}$. This gives us $ u_0 = q_0$ and $ u = q$. This shows $ u_j\rightarrow u = q $ strongly in $W^{1,2}(B_{3L}(0))$.
\end{proof}
From this claim and \eqref{eqn_1DEstimates} we conclude $u_j$ satisfies \eqref{LayerSeparation}, \eqref{WholeEnergyQuantization} and \eqref{PotentialEnergyQuantization} for sufficiently large $j$ which is a contradiction.
\end{proof}
Now we prove Proposition \ref{ReductionOfIntegrality}.
\begin{proof}[Proof of Proposition \ref{ReductionOfIntegrality}]
Without loss of generality, we assume $P_0=\{x\in\mathbb R^{n+1},x_{n+1}=0\}$ and let $ \pi:\mathbb R^{n+1}\rightarrow P_0$ denote the associated orthogonal projection.
 Furthermore we know
	\begin{align*}
	V_\e = \mu_\e \otimes \nu_\e \rightarrow V
	\end{align*}
is rectifiable and
	\begin{align*}
	\mu_V&= \mu\\
	 V&= \theta_0 \mathcal H^n \llcorner P_0 \otimes \delta_{P_0}
 \end{align*}
 and
 \begin{equation}\label{eqn_TiltDecay}
 \lim_{\e\rightarrow 0}\int_{B_4(0)} \e |\nabla u_\e|^2 \sqrt{1 - \nu_{\e, n+1}^2} = 0.
\end{equation}
Let $ N\in \mathbb N$ be the smallest integer with
\begin{align*}
N > \frac{\theta_0}{\alpha}
	\end{align*}
and let $ 0 < \delta\leq 1$ be small.
By Proposition \ref{prop_errorterms} and the $L^\infty$ bound condition of $u_\e$ in Theorem \ref{main}, we can fix $ \tau >0$ such that $\forall \e(\delta)>0$ sufficiently small we have
	\begin{align*}
	\int_{\{|u_\e|\geq 1 -\tau\} \cap B_4(0)}\frac{W'^2(u_\e)}{\e}+\frac{W(u_\e)}{\e}\leq \delta.
	\end{align*}
We have by Lemma \ref{L1PositiveDiscrepancyVanishing},
	\begin{align}\label{EnergyBoundAwayFromTransition}
	\mu_\e(\{|u_\e|\geq 1 -\tau \cap B_4(0)\}) \leq | \xi_\e(B_4(0))|+ 2\int_{\{|u_\e|\geq 1 -\tau\} \cap B_4(0)}\frac{W(u_\e)}{\e} \leq 3 \delta.
	\end{align}
We want to apply Lemma \ref{MultiSheetMonotonicity} and Lemma \ref{AlmostEnergyQuantization}.
We choose $ 0 < \omega =\omega(N,\delta, \frac{1}{2},\frac{1}{2}, C)$ and $ \omega(\delta, \tau, C) \leq 1$ where $L = L(\delta, \tau)$ which are the parameters that appear in Lemma \ref{MultiSheetMonotonicity} and Proposition \ref{prop_MultiSheetMonotonicity} and $C$ is the constant so that
	\begin{align*}
	\mu_\e(\Omega) + \frac{1}{\e}\int_{\Omega}|f_\e|^2 \leq C \quad \Omega = B_4(0).
	\end{align*}
We define $A_\e$ to be the set where the hypotheses for our Propositions hold, that is
	\begin{align*}
	A_\e =
\left\{x \in B_1(0) \left|
\begin{array}{l}
|u_\e(x)|\leq 1 - \tau, \\
	\forall \e \leq \rho \leq 3 : |\xi_\e(B_{\rho}(x))|+\int_{B_{\rho}(x)}\e|\nabla u_\e|^2\sqrt{1-\nu_{\e,n+1}^2} \leq \omega \rho^n, \\
	\forall \e \leq \rho \leq 3 : \frac{1}{\e}\int_{B_\rho(x)}|f_\e|^2 \leq \omega \rho^{n-1}.
\end{array}
\right.
\right\}.
	\end{align*}
We show the complement of the set $ A_\e$ has small measure.
By Besicovitch's covering theorem, we find a countable sub-covering $\cup_iB_{\rho_i}(x_i),\rho_i\in[\e,3]$ of $\{|u_\e|\leq1-\tau\}\setminus A_\e$ such that every point $x\in\{|u_\e|\leq1-\tau\}\setminus A_\e$ belongs to at most $\mathbf B_n$ balls in the covering, where $\mathbf B_n$ depends only on the dimension $n$. For each $i$, either
	\begin{align*}
	|\xi_\e(B_{\rho_i}(x_i))|+\int_{B_{\rho_i}(x_i)}\e|\nabla u_\e|^2\sqrt{1-\nu_{\e,n+1}^2} \geq \omega \rho_i^n,
	\end{align*}
or
	\begin{align*}
	\frac{1}{\e}\int_{B_{\rho_i}(x_i)}|f_\e|^2 \geq \omega \rho_i^{n-1}\geq C\omega\rho_i^n.
	\end{align*}
On the other hand, by \eqref{ConditionWillmoreTermDecay}, for sufficiently small $\e$, we have
	\begin{align*}
	\frac{1}{\e}\int_{B_{\rho}(x_i)}|f_\e|^2 \leq \omega \rho^{n-1}, \forall \rho\in[\e,3].
	\end{align*}
By \eqref{UpperRatioBoundEpsilonEnergy}, for each $i$, we obtain
	\begin{align*}
	\mu_\e\left(\overline {B_{\rho_i}(x_i)}\right)\leq C\rho_i^n.
	\end{align*}
Since the overlap in the Besicovitch covering is finite and \eqref{EnergyBoundAwayFromTransition}, we get
	\begin{align}\label{SmallnessAwayFromA}
	\begin{split}
\mu_\e\left(B_1(0)\setminus A_\e\right)&\leq3\delta+\sum_iC\rho_i^n\\
	&\leq3\delta+C\omega^{-1}\left(|\xi_\e|(B_4(0))+\int_{B_4(0)}\e|\nabla u_\e|^2\sqrt{1-\nu_{\e,n+1}^2}+\frac{1}{\e}\int_{B_4(0)}|f_\e|^2\right)\\
	&\leq4\delta,
\end{split}
	\end{align}
for $\e$ sufficiently small.
First by Lemma \ref{MultiSheetMonotonicity} and Lemma \ref{AlmostEnergyQuantization} we have $x\in A_\e, \forall L\e\leq R\leq\omega$,
	\begin{align*}
	\alpha\omega_n-\delta\leq(1+\delta)R^{-n}\mu_\e(B_R(x))+\delta.
	\end{align*}
By the reduction to the conditions in Proposition \ref{ReductionOfIntegrality}, we obtain
	\begin{align*}
	\mu_\e\left(\Omega\setminus\{|x_{n+1}|\leq\zeta\}\right)\rightarrow0, \quad \text{ for any fixed $\zeta>0$}.
	\end{align*}
Thus, for sufficiently small $\delta>0$, we get
	\begin{align*}
	A_\e\subset\{|x_{n+1}|\leq\zeta_\e\},\quad\text{ with $\zeta_\e\rightarrow0$ as $\e\rightarrow0$}.
	\end{align*}
For any $\hat y\in B^n_1(0)\subset\mathbb R^n$, consider a maximal subset
	\begin{align*}
	X=\{y\}\times\{t_1<...<t_K\}\subset A_\e\cap \pi^{-1}(y)
	\end{align*}
with $|x-x'|>3L\e$ if $x\neq x' \in X$, where $\pi$ denotes the projection to $\{x_{n+1}=0\}$.
If $K\geq N$, we apply Lemma \ref{MultiSheetMonotonicity} with $d=3L\e, R=\omega$ and Lemma \ref{AlmostEnergyQuantization} to get
	\begin{align*}
	N\alpha\omega_n-N\delta\leq(1+\delta)R^{-n}\mu_\e\left(B_R(y)\right)+\delta\leq(1+\delta)R^{-n}\mu_\e\left(B_{R+\zeta_\e}(y)\right)+\delta.
	\end{align*}
As
	\begin{align*}
	 \limsup_{\e\rightarrow0}(1+\delta)R^{-n}\mu_\e\left(B_{R+\zeta_\e}(y)\right)\leq R^{-n}\mu(\overline{B_R(y)})+C\delta=\theta\omega_n+C\delta,
	\end{align*}
and $\delta>0$ is arbitrarily small, we have
	\begin{align*}
	N\alpha\leq\theta,
	\end{align*}
which is a contradiction to our definition of $N$. So we obtain
	\begin{align*}
	K\leq N-1.
	\end{align*}
Since $X$ is maximal, we get
	\begin{align*}
	A_\e\cap\pi^{-1}(y)\subset\{y\}\times\cup_{k=1}^K(t_k-3L\e,t_k+3L\e).
	\end{align*}
By \eqref{LayerSeparation},
	\begin{align*}
	A_\e\cap\pi^{-1}(y)\cap&\left(\{y\}\times\cup_{k=1}^K(t_k-3L\e,t_k+3L\e)\right)\\
	&=A_\e\cap\pi^{-1}(y)\cap\left(\{y\}\times\cup_{k=1}^K(t_k-L\e,t_k+L\e)\right).
	\end{align*}
So
	\begin{align*}
	A_\e\cap\pi^{-1}(y)\subset\{y\}\times\cup_{k=1}^K(t_k-L\e,t_k+L\e)
	\end{align*}
and by \eqref{PotentialEnergyQuantization},
	\begin{align*}
	\int_{t_k-L\e}^{t_k+L\e}\frac{W(u_\e(y,t))}{\e}dt\leq\frac{\alpha}{2}+\delta, \quad \forall k=1,...,K.
	\end{align*}
Hence summing over $k$ gives
	\begin{align*}
	\int_{A_\e\cap\pi^{-1}(y)}\frac{W(u_\e)}{\e}d\mathcal H^1\leq\frac{(N-1)\alpha}{2}+(N-1)\delta
	\end{align*}
and integrating over $B^n_1(0)\subset\mathbb R^n$ we obtain
	\begin{align*}
	\int_{B_1^{n+1}(0)\cap A_\e}\frac{1}{\e}W(u_\e)d\mathcal H^{n+1}\leq\int_{B_1^n(0)}\int_{A_\e\cap\pi^{-1}(y)}\frac{W(u_\e)}{\e}d\mathcal H^1dy\leq\frac{(N-1)\alpha\omega}{2}+C\delta.
	\end{align*}
Recalling \eqref{SmallnessAwayFromA}, we get
	\begin{align*}
	\mu_\e\left(B_1(0)\right)&\leq \int_{B_1^{n+1}(0)\cap A_\e}\frac{1}{\e}W(u_\e)d\mathcal H^{n+1}+|\xi_\e\left(B_1(0)\right)|+\mu_\e\left(B_1(0)\setminus A_\e\right)\\
	&\leq(N-1)\alpha\omega_n+C\delta.
	\end{align*}
On the other hand, since $\lim_{\e\rightarrow0}\mu_\e\left(B_1(0)\right)=\theta\omega_n$ and $\delta>0$ is arbitrarily small, we obtain
	\begin{align*}
	\theta\leq(N-1)\alpha.
	\end{align*}
And since by definition $N$ is the smallest integer such that $\theta<N\alpha$, we have
	\begin{align*}
	\theta=(N-1)\alpha.
	\end{align*}
\end{proof}
\bibliography{AllardAllenCahn.bib}

\begin{thebibliography}{Mod87}

\bibitem[All72]{Allard1972}
William~K. Allard.
\newblock On the first variation of a varifold.
\newblock {\em Ann. of Math. (2)}, 95:417--491, 1972.

\bibitem[DG79]{DeGiorgi1979}
Ennio De~Giorgi.
\newblock Convergence problems for functionals and operators.
\newblock In {\em Proceedings of the {I}nternational {M}eeting on {R}ecent
  {M}ethods in {N}onlinear {A}nalysis ({R}ome, 1978)}, pages 131--188.
  Pitagora, Bologna, 1979.

\bibitem[HT00]{Hutchinson2000}
John~E. Hutchinson and Yoshihiro Tonegawa.
\newblock Convergence of phase interfaces in the van der
  {W}aals-{C}ahn-{H}illiard theory.
\newblock {\em Calc. Var. Partial Differential Equations}, 10(1):49--84, 2000.

\bibitem[Ilm93]{Ilmanen1993}
Tom Ilmanen.
\newblock Convergence of the {A}llen-{C}ahn equation to {B}rakke's motion by
  mean curvature.
\newblock {\em J. Differential Geom.}, 38(2):417--461, 1993.

\bibitem[MM77]{Modica1977}
Luciano Modica and Stefano Mortola.
\newblock Un esempio di {$\Gamma ^{-}$}-convergenza.
\newblock {\em Boll. Un. Mat. Ital. B (5)}, 14(1):285--299, 1977.

\bibitem[Mod87]{Modica1987}
Luciano Modica.
\newblock The gradient theory of phase transitions and the minimal interface
  criterion.
\newblock {\em Arch. Rational Mech. Anal.}, 98(2):123--142, 1987.

\bibitem[RS06]{roger2006modified}
Matthias R\"{o}ger and Reiner Sch\"{a}tzle.
\newblock On a modified conjecture of {D}e {G}iorgi.
\newblock {\em Math. Z.}, 254(4):675--714, 2006.

\bibitem[Ste88]{Sternberg1988}
Peter Sternberg.
\newblock The effect of a singular perturbation on nonconvex variational
  problems.
\newblock {\em Arch. Rational Mech. Anal.}, 101(3):209--260, 1988.

\bibitem[Ton03]{tonegawa2003integrality}
Yoshihiro Tonegawa.
\newblock Integrality of varifolds in the singular limit of reaction-diffusion
  equations.
\newblock {\em Hiroshima mathematical journal}, 33(3):323--341, 2003.

\bibitem[TT20]{Tonegawa2020}
Yoshihiro Tonegawa and Yuki Tsukamoto.
\newblock A diffused interface with the advection term in a {S}obolev space.
\newblock {\em Calc. Var. Partial Differential Equations}, 59(6):Paper No. 184,
  23, 2020.

\bibitem[Zie89]{Ziemer1989}
William~P. Ziemer.
\newblock {\em Weakly differentiable functions}, volume 120 of {\em Graduate
  Texts in Mathematics}.
\newblock Springer-Verlag, New York, 1989.
\newblock Sobolev spaces and functions of bounded variation.

\end{thebibliography}
\bibliographystyle{alpha}
\end{document}